\documentclass[pdflatex,sn-mathphys-ay]{sn-jnl}% Math and Physical Sciences Numbered Reference Style

\usepackage[utf8]{inputenc} 
\usepackage[T1]{fontenc}
\usepackage{a4}
\usepackage{amsfonts, amssymb, amsmath, amsthm}

\usepackage{enumitem}

\usepackage{mathtools} % For the barred arrow.
\usepackage{tensor} % For pullbacks.
\usepackage{scalerel} % For defining a big oslash.

\usepackage[ngerman, english]{babel}
\useshorthands{"}
\makeatletter
 \defineshorthand[ngerman]{"/}{\mbox{-}\bbl@allowhyphens}
\makeatother
\addto\extrasenglish{\languageshorthands{ngerman}}

\usepackage{tikz} % Drawing pictures.
\usetikzlibrary{decorations.markings, matrix, arrows, fit, calc}
\tikzstyle{map} = [->, font=\small]
\tikzstyle{mapsto} = [|->, font=\small]
\tikzstyle{linj} = [left hook->, font=\small]
\tikzstyle{rinj} = [right hook->, font=\small]
\tikzstyle{mono} = [>->, font=\small]
\tikzstyle{epi} = [->>, font=\small]
\tikzstyle{cell} = [double,double equal sign distance,-implies, shorten >= 4.5pt, shorten <= 4.5pt, font=\small]
\tikzstyle{cellmap} = [double,double equal sign distance,-implies, font=\small]
\tikzstyle{eq} = [double,double equal sign distance]
\tikzstyle{ps} = [shorten >= 2pt]

\tikzstyle{iso} = [above, sloped, inner sep=1.5pt]
\tikzstyle{nat} = [above, sloped, inner sep=2pt]
\tikzstyle{desc} = [fill=white, inner sep=2pt]
\tikzstyle{dots} = [black, font=]

\tikzstyle{small} = [font=\scriptsize]

\tikzstyle{textbaseline} = [baseline=-3.2pt]

\tikzstyle{barred} = [decoration={markings, mark=at position 0.5 with {\draw[-] (0,-1.5pt) -- (0,1.5pt);}}, postaction ={decorate}]

\tikzstyle{math35} = [matrix of math nodes, row sep={3.25em,between origins}, column sep={3.5em,between origins}, text height=1.5ex, text depth=0.25ex, nodes in empty cells]
\tikzstyle{minimath} = [matrix of math nodes, row sep={3em,between origins}, column sep={3.25em,between origins}, font=\scriptsize, text height=1ex, text depth=0.25ex, nodes in empty cells]
\tikzstyle{scheme} = [textbaseline, x=1.6em, y=1.6em, yshift=-2.4em, font=\scriptsize, text depth=0ex, every node/.style={overlay}, execute at end picture = { \useasboundingbox ($(current bounding box.north west) + (0,0.4em)$) rectangle ($(current bounding box.south east) - (0,0.4em)$); }]

\makeatletter
\def\slashedarrowfill@#1#2#3#4#5{%
  $\m@th\thickmuskip0mu\medmuskip\thickmuskip\thinmuskip\thickmuskip
   \relax#5#1\mkern-7mu%
   \cleaders\hbox{$#5\mkern-2mu#2\mkern-2mu$}\hfill
   \mathclap{#3}\mathclap{#2}%
   \cleaders\hbox{$#5\mkern-2mu#2\mkern-2mu$}\hfill
   \mkern-7mu#4$%
}
\def\rightslashedarrowfill@{%
  \slashedarrowfill@\relbar\relbar\mapstochar\rightarrow}
\newcommand\xslashedrightarrow[2][]{%
  \ext@arrow 0055{\rightslashedarrowfill@}{#1}{#2}}
\makeatother

\def\slashedrightarrow{\xslashedrightarrow{}}

\makeatletter
\newcommand{\xRrightarrow}[2][]{\ext@arrow 0359\Rrightarrowfill@{#1}{#2}}
\newcommand{\Rrightarrowfill@}{\arrowfill@\equiv\equiv\Rrightarrow}
\newcommand{\xLleftarrow}[2][]{\ext@arrow 3095\Lleftarrowfill@{#1}{#2}}
\newcommand{\Lleftarrowfill@}{\arrowfill@\Lleftarrow\equiv\equiv}
\makeatother

\newcommand{\conc}{%
  \mathbin{
    \mathchoice
    {\raisebox{1ex}{\scalebox{.7}{$\frown$}}}
    {\raisebox{1ex}{\scalebox{.7}{$\frown$}}}
    {\raisebox{.7ex}{\scalebox{.5}{$\frown$}}}
    {\raisebox{.7ex}{\scalebox{.5}{$\frown$}}}
  }
}

\newtheorem{theorem}{Theorem}[section]
\newtheorem{lemma}[theorem]{Lemma}
\newtheorem{corollary}[theorem]{Corollary}
\newtheorem{proposition}[theorem]{Proposition}

\theoremstyle{definition}
\newtheorem{definition}[theorem]{Definition}

\newtheorem{construction}[theorem]{Construction}

\theoremstyle{remark}
\newtheorem{remark}[theorem]{Remark}
\newtheorem{example}[theorem]{Example}
\newtheorem{convention}[theorem]{Convention}

\providecommand{\cororef}[1]{Corollary~\ref{#1}}

\providecommand{\defref}[1]{Definition~\ref{#1}}
\providecommand{\defsref}[2]{Definitions~\ref{#1} and~\ref{#2}}
\providecommand{\defthreeref}[3]{Definitions~\ref{#1}, \ref{#2} and~\ref{#3}}
\providecommand{\exref}[1]{Example~\ref{#1}}

\providecommand{\exsref}[2]{Examples~\ref{#1} and~\ref{#2}}

\providecommand{\lemref}[1]{Lemma~\ref{#1}}

\providecommand{\propref}[1]{Proposition~\ref{#1}}
\providecommand{\propsref}[2]{Propositions~\ref{#1} and~\ref{#2}}
\providecommand{\remref}[1]{Remark~\ref{#1}}
\providecommand{\thmref}[1]{Theorem~\ref{#1}}

\providecommand{\conref}[1]{Construction~\ref{#1}}

\providecommand{\secref}[1]{Section~\ref{#1}}
\providecommand{\secsref}[2]{Sections~\ref{#1} and~\ref{#2}}

\providecommand{\appref}[1]{Appendix~\ref{#1}}

\providecommand{\introref}{\hyperref[Introduction]{Introduction}}
\providecommand{\conventionsref}{\hyperref[Conventions and notation]{Conventions and notation}}
\providecommand{\outlineref}{\hyperref[Outline]{Outline}}

\providecommand{\dfn}{\coloneqq}
\providecommand{\nfd}{\eqqcolon}

\providecommand{\of}{\circ}
\providecommand{\iso}{\cong}

\providecommand{\brar}{\slashedrightarrow}
\providecommand{\xrar}{\xrightarrow}
\providecommand{\xlar}{\xleftarrow}
\providecommand{\xbrar}{\xslashedrightarrow}
\providecommand{\Rar}{\Rightarrow}
\providecommand{\xRar}{\xRightarrow}

\providecommand{\into}{\hookrightarrow}

\DeclareMathOperator{\dash}{\makebox[1.3ex]{\text{--}}}

\providecommand{\ul}[1]{\underline{#1}{}}
\providecommand{\ull}[1]{\ul{\ul{#1}}{}}

\providecommand{\brks}[1]{\lbrack #1 \rbrack}
\providecommand{\bigbrks}[1]{\bigl\lbrack #1 \bigr\rbrack}

\providecommand{\pars}[1]{\left(#1\right)}
\providecommand{\bigpars}[1]{\bigl(#1\bigr)}

\providecommand{\lns}[1]{\lvert#1\rvert}

\providecommand{\map}[3]{#1\colon#2\to#3}

\providecommand{\emb}[3]{#1\colon#2\into#3}

\providecommand{\cell}[3]{#1\colon#2\Rightarrow#3}

\providecommand{\hmap}[3]{#1\colon#2\slashedrightarrow#3}

\providecommand{\inv}[1]{{#1}^{-1}}

\DeclareMathOperator{\src}{src}
\DeclareMathOperator{\tgt}{tgt}

\newcommand{\id}{\mathrm{id}}
\newcommand{\yon}{\mathrm{y}}
\providecommand{\A}{\mathcal A}

\providecommand{\op}[1]{#1^\textup{op}}
\providecommand{\co}[1]{#1^\textup{co}}
\providecommand{\coop}[1]{#1^\textup{coop}}

\providecommand{\ps}[1]{\widehat{#1}}

\providecommand{\catvar}[1]{\mathcal{#1}}

\providecommand{\1}{\mathsf 1}
\providecommand{\2}{\mathsf 2}
\providecommand{\A}{\catvar A}

\providecommand{\C}{\catvar C} % clashes with fontenc when uploading to arXiv

\providecommand{\E}{\catvar E}

\providecommand{\K}{\catvar K}
\renewcommand{\L}{\catvar L}

\providecommand{\R}{\catvar R}
\providecommand{\s}{\catvar S}

\providecommand{\Set}{\mathsf{Set}}

\providecommand{\Cls}{\mathsf{Cls}}
\providecommand{\Cat}{\mathsf{Cat}} % Category of small categories
\providecommand{\Catun}{\Cat'_{\textup{un}}} % Sesquicategory of unnatural transformations

\providecommand{\inCat}[1]{\Cat(#1)}
\providecommand{\TwoCat}{\mathsf{2Cat}}
\providecommand{\TwoCatLax}{\TwoCat'_\textup{lax}}
\providecommand{\TwoCatOpl}{\TwoCat'_\textup{opl}}
\providecommand{\TwoCatPs}{\TwoCat'_\textup{ps}}
\providecommand{\SpTwoFib}{\mathsf{Sp2Fib}}
\providecommand{\SpCoopTwoFib}{\mathsf{SpCoop2Fib}}
\providecommand{\SpOpTwoFib}{\mathsf{SpOp2Fib}}
\providecommand{\SpCoTwoFib}{\mathsf{SpCo2Fib}}
\providecommand{\SpTwoTwoFib}{\mathsf{SpTwo2Fib}}
\providecommand{\LdSpTwoTwoFib}{\SpTwoTwoFib_\textup{ld}}
\providecommand{\SfLdSpTwoTwoFib}{\SpTwoTwoFib_\textup{ld,sf}}

\providecommand{\und}[1]{\lns{#1}} % undercategory
\providecommand{\inhom}[1]{\brks{#1}}

\DeclareMathOperator{\yev}{e} % Evaluation ( )_a(id_a) used in the Yoneda lemma

\DeclareMathOperator{\str}{str}
\DeclareMathOperator{\costr}{costr}

\providecommand{\Prof}{\mathsf{Prof}}

\providecommand{\inProf}[1]{\Prof(#1)}
\providecommand{\inSpTwoFib}[1]{\mathsf{SpTwoFib}(#1)}
\providecommand{\inSpFib}[1]{\mathsf{SpFib}(#1)}
\providecommand{\inSpOpFib}[1]{\mathsf{SpOpFib}(#1)}

\providecommand{\hc}{\odot}

\DeclareMathOperator{\vs}{\slash_\textup v}
\DeclareMathOperator{\hs}{\slash_\textup h}

\newcommand{\cocart}{\mathrm{cocart}}
\newcommand{\cart}{\mathrm{cart}}

\providecommand{\cur}[1]{#1^{\scriptscriptstyle\lambda}}

\begin{document}
	\title{Virtual double categories of split two-sided 2-fibrations}
	\subtitle{\it\small Dedicated to Bob Paré on the occasion of his 80th birthday}
	
	\author*{\fnm{Seerp Roald} \sur{Koudenburg}}\email{roaldkoudenburg@gmail.com}
	
	\affil{\orgdiv{Mathematics Research and Teaching Group}, \orgname{Middle East Technical University Northern Cyprus Campus}, \orgaddress{\postcode{99738} \city{Kalkanlı}, \state{Güzelyurt}, \country{TRNC}}}

	\abstract{
		This paper introduces and studies split two"/sided 2"/fibrations and locally discrete split two"/sided 2"/fibrations, using a formal categorical approach. We generalise Street's notion of split two"/sided fibration internal to a 2"/category to one internal to a sesquicategory. Given a sesquicategory we construct a virtual double category whose horizontal (loose) morphisms are its internal split two"/sided fibrations. Specialising to the sesquicategory of lax natural transformations, we obtain the virtual double category of split two"/sided 2"/fibrations, which we study in detail. We then restrict to the sub"/virtual double category of locally discrete split two"/sided 2"/fibrations and show that therein the usual Yoneda 2"/functors satisfy a double"/categorical formal notion of Yoneda morphism, which formally captures universal properties similar to those satisfied by the morphisms comprising a Yoneda structure on a 2"/category. As a consequence we obtain a `two"/sided Grothendieck correspondence' of locally discrete split two"/sided 2"/fibrations $A \brar B$ and 2"/functors $B \to \Cat^{\op A}$. Restricting to $A = 1$, the terminal 2"/category, we improve Lambert's `Grothendieck correspondence' for locally discrete split op"/2"/fibrations by extending the sense in which it is functorial.
	}
	
	\keywords{2-fibration, two-sided fibration, discrete fibration, virtual double category, Yoneda structure, formal category theory}
	\pacs[2020 Mathematics Subject Classification]{18D30, 18D70, 18N10}

	\maketitle
	%\tableofcontents
	
	\section*{Introduction} \addcontentsline{toc}{section}{\protect\numberline{}Introduction}\label{Introduction}
	Recently (split) $2$"/fibrations, a $2$"/categorical generalisation of ordinary (split) fibrations that was introduced and studied in \cite{Buckley14}, have found several uses (\cite{Lucyshyn-Wright16}, \cite{Lambert24}, \cite{Mesiti24} and \cite{Gagna-Harpaz-Lanari24}). The goal of the present paper is to similarly introduce and study the notion of split two"/sided $2$"/fibration, generalising the ordinary notion of split two"/sided fibration, by taking a formal categorical approach. As a consequence of this approach, besides their definition, we simultaneously obtain a virtual double category of split two"/sided $2$"/fibrations. Roughly, virtual double categories (\cite{Cruttwell-Shulman10}) generalise ordinary double categories by allowing multicells with paths of horizontal (loose) morphisms as their horizontal source, and the virtual double category $\SpTwoTwoFib$ obtained here is that of $2$"/categories, $2$"/functors and split two"/sided $2$"/fibrations: these form its objects, vertical (tight) morphisms and horizontal morphisms, respectively. In the \outlineref\ of the paper below we describe its construction in some more detail.
	
	Our formal categorical approach immediately is advantageous. By studying the multicells of $\SpTwoTwoFib$ for instance, we naturally obtain the right notion of (multi)morphism of split two"/sided $2$"/fibrations. Similarly, after restricting to the sub"/virtual double category of locally discrete split two"/sided $2$"/fibrations, whose fibres are small categories, we immediately have available the correct formal notion of ``Yoneda morphism'' therein. The latter notion, introduced in \cite{Koudenburg24}, isolates universal properties analogous to those satisfied by the morphisms comprising a Yoneda structure on a $2$"/category (\cite{Street-Walters78} and \cite{Weber07}).
	
	Given a locally small $2$"/category $A$, our main result (\cororef{main result}) shows that the usual Yoneda $2$"/functor $\map\yon A{\Cat^{\op A}}$, that sends $a \in A$ to the representable $2$"/functor $\map{A(\dash, a)}{\op A}\Cat$, is a formal Yoneda morphism in (roughly) the virtual double category of locally discrete split two"/sided $2$"/fibrations with small fibres. As a consequence we obtain a ``Grothendieck correspondence'' (\cororef{equivalence from yoneda embedding}) between such two"/sided $2$"/fibrations $A \brar B$ and $2$"/functors $B \to \Cat^{\op A}$, generalising the classical Grothendieck correspondence between ordinary split opfibrations over a $1$"/category $B$ and $1$"/functors $B \to \Cat$ (see e.g.\ Section~1 of \cite{Gray69}).
	
	Restricting the main result to $A = 1$, the terminal $2$"/category, we recover and improve on Lambert's correspondence of split op"/$2$"/fibrations over a $2$"/category $B$ and $2$"/functors $B \to \Cat$ (\cite{Lambert24}), which is a restriction of a more general correspondence for fibrations of bicategories (\cite{Bakovic10} and \cite{Buckley14}). Our improvement is that we extend the functoriality of the correspondence, as is explained at the end of the \outlineref\ below and in \cororef{locally discrete split op-2-fibration correspondence}.
	
	The motivation for this paper is twofold.
	\begin{enumerate}
		\item Once we have assembled $2$"/categories, $2$"/functors and (locally discrete) split two"/sided $2$"/fibrations into a virtual double category $\SpTwoTwoFib$, recent formal results obtained within the setting of virtual double categories---such as those of \cite{Koudenburg15b}, \cite{Koudenburg24}, \cite{Arkor-McDermott25} and \cite{Hoshino-Kawase24}---can be directly applied to $\SpTwoTwoFib$.
		
		In particular, the lifting theorem for algebraic Yoneda embeddings of Section~8 of \cite{Koudenburg15b} applies to lift the Grothendieck correspondence for locally discrete split two"/sided $2$"/fibrations to a correspondence for monoidal such $2$"/fibrations. The latter generalises the split variant of the Grothendieck correspondence for monoidal ordinary fibrations introduced by \cite{Moeller-Vasilakopoulou20}.
		
		All of the formal category theory developed in \cite{Koudenburg24} too can be applied, such as its results on exact cells, totality and cocompleteness; see its Sections~5, 6 and 7, respectively, as well as \remref{other formal results} below.

		\item The construction of the virtual double category of split two"/sided $2$"/fibrations presented here is set up in a general fashion that allows it to be re"/used in constructing virtual double categories of similar types of fibration. By modifying the construction slightly it is, for instance, possible to obtain a virtual double category of double split fibrations. Inside it the split variant of the Grothendieck correspondence of double fibrations, as obtained by \cite{Cruttwell-Lambert-Pronk-Szyld22}, can be considered as a formal Yoneda embedding. The latter can then be lifted to obtain an analogous correspondence for monoidal double split fibrations, which is likely to be of use in applied category theory (\cite{Patterson23}).
	\end{enumerate}
	
	In closing this introduction, we briefly explain why Yoneda structures on $2$"/categories (\cite{Street-Walters78}) are unable to formally capture the universal property of the Yoneda $2$"/functors $\map\yon A\Cat^{\op A}$. The reason is that any locally discrete split two"/sided $2$"/fibration $A \xlar p J \xrar q B$ is related to its corresponding $2$"/functor $\map{\cur J}B{\Cat^{\op A}}$ via a lax natural transformation $\chi$ of the form below left, and lax natural transformations do not form the cells of any $2$"/category since their compositions do not satisfy the interchange law. On the other hand, $\chi$ can be regarded as a $(1, 1)$"/ary cell below right in the virtual double category of split two"/sided $2$"/fibrations that we will construct, with the `unit' two"/sided $2$"/fibration as horizontal target.
	\begin{displaymath}
		\begin{tikzpicture}[baseline]
			\matrix(m)[math35]{J & A \\ B & \Cat^{\op A} \\};
			\path[map]	(m-1-1) edge node[above] {$p$} (m-1-2)
													edge node[left] {$q$} (m-2-1)
									(m-1-2) edge node[right] {$\yon$} (m-2-2)
									(m-2-1) edge node[below] {$\cur J$} (m-2-2);
			\path				(m-1-2) edge[cell, shorten >= 9pt, shorten <= 9pt] node[below right] {$\chi$} (m-2-1);
		\end{tikzpicture} \qquad \qquad \qquad \qquad \begin{tikzpicture}[baseline]
			\matrix(m)[math35, column sep={4.2em,between origins}]{ A & B \\ \Cat^{\op A} & \Cat^{\op A} \\};
			\path[map]	(m-1-1) edge[barred] node[above] {$J$} (m-1-2)
													edge node[left] {$\yon$} (m-2-1)
									(m-1-2) edge node[right] {$\cur J$} (m-2-2);
			\path				(m-2-1) edge[eq] (m-2-2);
			\path[transform canvas={xshift=2.1em}]	(m-1-1) edge[cell] node[right] {$\chi$} (m-2-1);
		\end{tikzpicture}
	\end{displaymath}
	
	\section*{Outline} \addcontentsline{toc}{section}{\protect\numberline{}Outline} \label{Outline}
	Here we briefly outline the contents of this paper. As explained the setting for our formal approach is required to contain lax natural transformations. Since such transformations do not form a $2$"/category, we regard them as cells of a sesquicategory instead: a $2$"/dimensional categorical structure that has objects, morphisms and cells exactly like a $2$"/category does, but whose cells are not equipped with a horizontal composition. We recall the definition of sesquicategory in \secref{sesquicategory section} and describe two examples, including the sesquicategory $\mathsf{2Cat}_\textup{lax}$ of $2$"/categories, $2$"/functors and lax natural transformations, which we will use throughout.
	
	The main idea of our construction of a virtual double category of split two"/sided $2$"/fibrations is to generalise Street's notion of split two"/sided fibration $A \brar B$ internal to a suitable $2$"/category $\C$ (\cite{Street74b}), to one that is internal to a sesquicategory. Street's notion can be equivalently defined as a profunctor $A^\2 \brar B^\2$ internal to the underlying $1$"/category $\und{\C}$, between the `$\2$"/cotensor' internal categories of $A$ and $B$. The notion of comma object, which generalises that of $\2$"/cotensor, is considered inside sesquicategories in \secref{comma objects section}. Comma objects in $\mathsf{2Cat}_\textup{lax}$ recover Gray's notion of lax comma $2$"/category (\cite{Gray74}).
	
	\secref{internal profunctors section} recalls the notion of profunctors internal to a suitable $1$"/category $\E$ and shows that, together with internal profunctors in $\E$, they form an `augmented' virtual double category $\inProf\E$. The latter notion (\cite{Koudenburg20}) extends that of ordinary virtual double category, whose multicells are of the form below left, by including cells with empty horizontal targets, as in the middle below. In particular, it adds `vertical' cells whose horizontal source and target are both empty, as below right. As explained, \secref{internal split two-sided fibrations section} defines split two"/sided fibrations $A \brar B$ internal to a suitable sesquicategory $\s$ to be profunctors of the form $A^\2 \brar B^\2$ internal to the underlying $1$"/category $\und\s$. We take the augmented virtual double category $\inSpTwoFib\s$ of objects, morphisms and internal split two"/sided fibrations in $\s$ to be a sub"/augmented virtual double category of $\inProf{\und{\s}}$.
	\begin{displaymath}
		\begin{tikzpicture}[baseline]
			\matrix(m)[math35]{A_0 & A_1 & A_{n'} & A_n \\ C & & & D \\};
			\path[map]	(m-1-1) edge[barred] node[above] {$J_1$} (m-1-2)
													edge node[left] {$f$} (m-2-1)
									(m-1-3) edge[barred] node[above] {$J_n$} (m-1-4)
									(m-1-4) edge node[right] {$g$} (m-2-4)
									(m-2-1) edge[barred] node[below] {$K$} (m-2-4);
			\path[transform canvas={xshift=1.75em}]	(m-1-2) edge[cell] node[right] {$\phi$} (m-2-2);
			\draw				($(m-1-2)!0.5!(m-1-3)$) node {$\dotsb$};
		\end{tikzpicture} \quad\quad\hspace{-0.2em} \begin{tikzpicture}[baseline]
			\matrix(m)[math35, column sep={1.75em,between origins}]
				{A_0 & & A_1 & & A_{n'} & & A_n \\ & & & C & & & \\};
			\path[map]	(m-1-1) edge[barred] node[above] {$J_1$} (m-1-3)
													edge node[below left] {$f$} (m-2-4)
									(m-1-5) edge[barred] node[above] {$J_n$} (m-1-7)
									(m-1-7) edge node[below right] {$g$} (m-2-4);
			\path				(m-1-4) edge[cell] node[right] {$\psi$} (m-2-4);
			\draw				(m-1-4) node {$\dotsb$};
		\end{tikzpicture} \quad\quad\hspace{-0.2em} \begin{tikzpicture}[baseline]
			\matrix(m)[math35]{A_0 \\ C \\};
			\path[map]	(m-1-1) edge[bend right=45] node[left] {$f$} (m-2-1)
													edge[bend left=45] node[right] {$g$} (m-2-1);
			\path				(m-1-1) edge[cell] node[right] {$\xi$} (m-2-1);
		\end{tikzpicture}
	\end{displaymath}
	
	In \secref{split op-2-fibrations section} we recall Lambert's notion of split op"/$2$"/fibration $J \to B$ and use one of the main results of \cite{Lambert24} to see that it coincides with that of a split two"/sided fibration $1 \brar B$ internal to the sesquicategory $\mathsf{2Cat}_\textup{lax}$.  This correspondence is used in \secref{split two-sided 2-fibrations section} to describe split two"/sided fibrations $A \brar B$ internal to $\mathsf{2Cat}_\textup{lax}$, which we define to be split two"/sided $2$"/fibrations. We also describe the (multi)cells of split two"/sided $2$"/fibrations, which form the cells of the augmented virtual double category $\SpTwoTwoFib \dfn \inSpTwoFib{\mathsf{2Cat}_\textup{lax}}$. Roughly, the $(n, 1)$"/ary cells $\phi$ left above are suitable multimorphisms of split two"/sided $2$"/fibrations (see \thmref{unary cells of split two-sided 2-fibrations}), the $(n, 0)$"/ary cells $\psi$ in the middle above are certain ``marked'' lax natural transformations (see \defref{marked lax natural transformation}), such as the transformation $\chi$ mentioned at the end of the \introref, while the $(0,0)$"/ary (vertical) cells $\xi$ above are ordinary $2$"/natural transformations $f \Rar g$.
	
	In \secref{locally discrete split two-sided 2-fibrations section} we restrict to split two"/sided $2$"/fibrations that are locally discrete and whose fibres are small, meaning that their fibres are small $1$"/categories. They generate a sub"/augmented virtual double category $\SfLdSpTwoTwoFib$ of $\SpTwoTwoFib$. In \cite{Koudenburg24} a formal Yoneda embedding is defined to be a (formally) dense morphism that satisfies a `Yoneda axiom', and the final \secsref{density section}{Yoneda axiom section} are devoted to proving that the Yoneda $2$"/functor \mbox{$\map\yon A{\Cat^{\op A}}$} is a formal Yoneda embedding in $\SfLdSpTwoTwoFib$. \secref{density section} considers the `two"/sided Grothendieck construction' of a $2$"/functor $\map gB{\Cat^{\op A}}$ as the lax comma $2$"/category $\yon \slash g$, which forms a locally discrete split two"/sided $2$"/fibration $\hmap{\yon \slash g}AB$. Formal density of $\yon$ requires the lax natural transformation defining $\yon \slash g$, regarded as a $(1, 0)$"/ary cell in $\LdSpTwoTwoFib$, to define $g$ as a left Kan extension of $\yon$ along $\yon \slash g$, in the sense of \cite{Koudenburg24}, which is shown in \thmref{Yoneda 2-functor is dense}.
	
	The formal Yoneda axiom requires, for each locally discrete split two"/sided $2$"/fibration $\hmap JAB$ with small fibres, the existence of a $2$"/functor $\map{\cur J}B{\Cat^{\op A}}$ such that $J \iso \yon \slash \cur J$. \secref{Yoneda axiom section} constructs the $2$"/functor $\cur J$, using a generalisation of the ordinary inverse Grothendieck construction, and \thmref{Yoneda axiom} proves the required isomorphism. The main result \cororef{main result} then concludes that the Yoneda $2$"/functor forms a formal Yoneda embedding. In closing \secref{Yoneda axiom section}, we list some of the consequences of the main result, including the `Grothendieck correspondence' of small"/fibred locally discrete split two"/sided $2$"/fibrations $A \brar B$ and $2$"/functors $B \to \Cat^{\op A}$. 
	
	Restricting to $A = 1$, the terminal $2$"/category, we recover and improve Lambert's Grothendieck correspondence of small"/fibred locally discrete split op"/$2$"/fibrations $1 \brar B$ and $2$"/functors $B \to \Cat$ (\cite{Lambert24}). The improvement (\cororef{locally discrete split op-2-fibration correspondence}) consists of extending the functoriality of the correspondence from the usual morphisms of $J \to K$ of locally discrete split op"/$2$"/fibrations to multimorphisms of the form below left (\cororef{cells of locally discrete split two-sided 2-fibrations}), where $(H_1, \dotsc, H_n)$ is a (possibly empty) path of locally discrete split two"/sided $2$"/fibrations. \cororef{locally discrete split op-2-fibration correspondence} shows that the latter correspond bijectively to the ``marked'' lax natural transformations $\psi$ below right (\defref{marked lax natural transformation}), where the $2$"/functors $\cur J$ and $\cur K$ correspond to $J$ and $K$, respectively. The lifting theorem for algebraic Yoneda embeddings, that was mentioned in the first motivation above, requires this extended sense of functoriality of the Grothendieck correspondence.
	\begin{displaymath}
		\begin{tikzpicture}[baseline]
				\matrix(m)[math35, column sep={1.75em,between origins}]
					{	1 & & B & & B_1 & \cdots & B_{n'} & & D \\
						& & & 1 & & D & & & \\};
				\path[map]	(m-1-1) edge[barred] node[above] {$J$} (m-1-3)
										(m-1-3) edge[barred] node[above] {$H_1$} (m-1-5)
										(m-1-7) edge[barred] node[above] {$H_n$} (m-1-9)
										(m-2-4) edge[barred] node[below] {$K$} (m-2-6);
				\path				(m-1-1) edge[transform canvas={xshift=-1pt}, eq] (m-2-4)
										(m-1-9) edge[transform canvas={xshift=1pt}, eq] (m-2-6)
										(m-1-5) edge[cell] node[right] {$\phi$} (m-2-5);
			\end{tikzpicture} \qquad\qquad \begin{tikzpicture}[baseline]
			\matrix(m)[math35, column sep={1.75em,between origins}]
				{B & & B_1 & & B_{n'} & & D \\ & & & \Cat & & & \\};
			\path[map]	(m-1-1) edge[barred] node[above] {$H_1$} (m-1-3)
													edge node[below left] {$\cur J$} (m-2-4)
									(m-1-5) edge[barred] node[above] {$H_n$} (m-1-7)
									(m-1-7) edge node[below right] {$\cur K$} (m-2-4);
			\path				(m-1-4) edge[cell] node[right] {$\psi$} (m-2-4);
			\draw				(m-1-4) node {$\dotsb$};
		\end{tikzpicture}
	\end{displaymath}
	
	\subsection*{Conventions and notation} \label{Conventions and notation} \addcontentsline{toc}{section}{\protect\numberline{}Conventions and notation}
	Regarding tuples and paths:
	\begin{enumerate}[label=-]
		\item For any integer $n \geq 1$ we write $n' \dfn n - 1$ for its predecessor.
		\item Given ``composable'' (possibly empty) paths $\ul i = (i_1, \dotsc, i_n)$ and $\ul j = (j_1, \dotsc, j_m)$ we write $\ul i \conc \ul j \dfn (i_1, \dotsc, i_n, j_1, \dotsc, j_m)$ for their concatenation.
		\item For a tuple $\ul i = (i_1, \dotsc, i_n)$ of objects of a category we write $\id_{\ul i} \dfn (\id_{i_1}, \dotsc, \id_{i_n})$ for the corresponding tuple of identity morphisms.
	\end{enumerate}
	Regarding sizes:
	\begin{enumerate}[label=-]
		\item We assume fixed a hierarchy of sub"/$1$"/categories $\Set \subseteq \Set' \subseteq \Cls$ of small sets, large sets, and classes, such that the collections of morphisms of $\Set$ and $\Set'$ form objects in $\Set'$ and $\Cls$, respectively.
		\item For each $\E = \Set$, $\Set'$ or $\Cls$ we denote by $\inCat{\E}$ the $2$"/category of $1$"/categories, functors and $1$"/natural transformations internal in $\E$, and by $\und{\inCat{\E}}$ its underlying $1$"/category of internal categories and functors. Thus $\Cat \dfn \inCat{\Set}$ and $\Cat' \dfn \inCat{\Set'}$ are the $2$"/categories of small and large $1$"/categories, respectively, and $\inCat{\Cls}$ is the $2$"/category of \emph{class-sized} $1$"/categories whose collections of objects and morphisms form classes.
		\item Following \cite{Lambert24}, given any ``$n$"/dimensional categorical structure'' $\mathcal K$ we denote by $\und{\mathcal K}$ the underlying $(n-1)$"/dimensional structure obtained by removing the top"/dimensional cells of $\mathcal K$. For example $\und{\und{\TwoCat}}$ denotes the $1$"/category of small $2$"/categories and $2$"/functors obtained by removing the $2$"/cells and $3$"/cells from the $3$"/category $\TwoCat$ of small $2$"/categories.
	\end{enumerate}
	Regarding composition and duals:
	\begin{enumerate}[label=-]
		\item In $2$"/categories and sesquicategories composition of morphisms, horizontal composition of cells (i.e.\ along an object) and whiskerings of cells with morphisms are denoted by $\of$. Vertical composition of cells (i.e.\ along a morphism) is denoted by $\hc$. We use the ``classical'' order for composition of morphisms and horizontal composition of cells, and diagrammatic order for vertical composition of cells. That is, $g \of f$ denotes the composition of morphisms $\map fAC$ and $\map gCE$ and $\psi \of \phi$ denotes the horizontal composition of cells $\cell\phi f{f'}$ and $\cell\psi g{g'}$, while $\phi \hc \xi$ denotes the vertical composition of cells $\cell\phi fg$ and $\cell\xi gh$.
		\item Following the conventions of e.g. \cite{Shulman08} and \cite{Koudenburg20}, in (generalised) double categories composition of vertical (also called tight) morphisms and composition of cells along a boundary of horizontal (also called loose) morphisms is called \emph{vertical composition} and denoted by $\of$. Composition of cells along a vertical (tight) morphism is called \emph{horizontal composition}, denoted by $\hc$ and written in diagrammatic order.
		
		As explained in Example~1.5 of \cite{Koudenburg20}, any augmented virtual double category $\K$ contains a $2$"/category $V(\K)$ consisting of its objects, vertical morphisms and vertical cells, the latter as described in the \outlineref\ above. Our notations for composition in $\K$, when restricted to its vertical morphisms and vertical cells, coincide with those for composition in $V(\K)$. Unfortunately the ``horizontal/vertical'' terminology is reversed however: for vertical cells $\phi$, $\psi$ and $\xi$, the composite $\psi \of \phi$ (along an object) is vertical in $\K$ but horizontal in $V(\K)$, while the composite $\phi \hc \xi$ (along a vertical morphism) is horizontal in $\K$ but vertical in $V(\K)$.
		\item Given a $2$"/category $\K$ we denote by $\op \K$, $\co \K$ and $\coop \K$ the three dual $2$"/categories formed by reversing the direction either of the morphisms (`op') or cells (`co') of $\K$, or of both (`coop'). By the ``op"/dual'' of a result or definition concerning $2$"/categories we mean its application to the op"/duals of the involved $2$"/categories; ``co"/dual'' and ``coop"/dual'' are used similarly. A ``split op"/$2$-fibration'' (\defref{split op-2-fibration}) for example is a $2$"/functor $\map qJB$ whose op"/dual $\map{\op q}{\op J}{\op B}$ is a split $2$"/fibration.

	\end{enumerate}
	
	\section{Sesquicategories} \label{sesquicategory section}
	We start by recalling the notion of sesquicategory; see e.g.\ Section~3 of \cite{Stell94}, Section~2 of \cite{Street96} or Section~2 of \cite{Lack10}. In \secref{split two-sided 2-fibrations section} we will define a split two"/sided $2$"/fibration to be an internal split two"/sided fibration (\secref{internal split two-sided fibrations section}) in the sesquicategory of $2$"/categories, $2$"/functors and lax natural transformations (\exref{lax natural transformation}).
	\begin{definition} \label{sesquicategory}
		A \emph{sesquicategory} $\s$ consists of an underlying class"/sized $1$"/category $\und\s$, equipped with a lifting $\map{\s(\dash, \dash)}{\op{\und\s} \times \und\s}{\und{\inCat{\Cls}}}$ of its hom"/functor \mbox{$\map{\und\s(\dash, \dash)}{\op{\und\s} \times \und\s}{\Cls}$}. That is, the diagram below commutes, where the functor $\map {\lns{\dash}}{\und{\inCat{\Cls}}}{\Cls}$ on the right sends each category to its class of objects.
		\begin{displaymath}
			\begin{tikzpicture}
				\matrix(m)[math35, column sep={7em,between origins}]{ & \und{\inCat{\Cls}} \\ \op{\und\s} \times \und{\s} & \Cls \\};
				\path[map]	(m-1-2) edge node[right] {$\lns{\dash}$} (m-2-2)
										(m-2-1) edge node[above left] {$\s(\dash, \dash)$} (m-1-2)
														edge node[below] {$\und{\s}(\dash, \dash)$} (m-2-2);
			\end{tikzpicture}
		\end{displaymath}
		
		We call $\s$ \emph{small} if its underlying $1$"/category $\und\s$ is a small $1$"/category and its lifted hom"/functor $\s(\dash, \dash)$ factors through the inclusion $\und{\Cat} \subseteq \und{\inCat{\Cls}}$. \emph{Large} sesquicategories are defined analogously.
	\end{definition}
	
	Sesquicategories can equivalently be defined as class"/sized categories enriched over $\und{\inCat{\Cls}}$ equipped with the `funny tensor product' $\Box$; see e.g.\ Section~3 of \cite{Stell94}. See e.g.\ Section~2 of \cite{Weber13} for a description of $\Box$, whose corresponding closed structure $\inhom{\dash, \dash}$ is given by $1$"/categories of functors and `unnatural' transformations; see \defref{unnatural transformation definition} below. As explained in \appref{comma objects appendix} below, this viewpoint naturally leads to the correct (enriched) notions of (co)limit for sesquicategories.
	
	We unpack \defref{sesquicategory} in three steps. Firstly, underlying a sesquicategory $\s$ is a class"/sized $1$"/category $\und\s$ consisting of objects $A$, $B$, \dots, and morphisms $\map fAC$, \dots that are equipped with a composition denoted $\of$. Secondly, $\s$ comes equipped with, for each pair of its objects $A$ and $C$, a class"/sized $1$"/category $\s(A, C)$ whose objects are the morphisms $A \to C$ of $\und\s$ and whose morphisms are cells $\cell\phi f{\map gAC}$, \dots, that are equipped with a composition denoted $\hc$ that we write in diagrammatic order (for details see \conventionsref). Both compositions are associative and have units that we denote by $\id$. Thirdly, the lifting $\s(\dash, \dash)$ extends composition $\of$ to include \emph{whiskerings} of cells with morphisms on either side. For the cells and morphisms below, for example, the composites $\cell{h \of \phi \dfn \s(\id_A, h)(\phi)}{h \of f}{h \of g}$ and $\cell{\psi \of g}{h \of g}{k \of g}$ are specified. Whiskering too is associative and unital, and is functorial with respect to $\hc$. Section~3 of \cite{Stell94} lists all axioms satisfied by $\of$ and $\hc$. 
	\begin{displaymath}
		\begin{tikzpicture}
			\matrix(m)[math35]{A & C & E \\ A & C & E \\};
						\path[map]	(m-1-1) edge[bend right=45] node[below] {$g$} (m-1-2)
																edge[bend left=45] node[above] {$f$} (m-1-2)
												(m-1-2) edge node[above] {$h$} (m-1-3)
												(m-2-2) edge[bend right=45] node[below] {$k$} (m-2-3)
																edge[bend left=45] node[above] {$h$} (m-2-3)
												(m-2-1) edge node[below] {$g$} (m-2-2);
						\path				([yshift=8.5pt]$(m-1-1)!0.5!(m-1-2)$) edge[cell] node[right] {$\phi$} ([yshift=-9.5pt]$(m-1-1)!0.5!(m-1-2)$)
												([yshift=8.5pt]$(m-2-2)!0.5!(m-2-3)$) edge[cell] node[right] {$\psi$} ([yshift=-9.5pt]$(m-2-2)!0.5!(m-2-3)$);
		\end{tikzpicture}
	\end{displaymath}
	
	\subsubsection*{Strict and costrict cells}
	Notice that the definition above does \emph{not} require the \emph{middle"/four interchange law} for horizontally composable cells $\phi$ and $\psi$ such as those above, which states that
	\begin{displaymath}
		(h \of \phi) \hc (\psi \of g) = (\psi \of f) \hc (k \of \phi).
	\end{displaymath}
	Consequently it is not possible in general to extend $\of$ to horizontal composition of cells, i.e.\ $\cell{\psi \of \phi}{h \of f}{k \of g}$ in the example above, like in a $2$"/category. In fact, a sesquicategory $\s$ ``is'' a $2$"/category, that is, its horizontal composition (necessary uniquely) extends to cells, if and only if all horizontally composable cells $\phi$, $\psi \in \s$ satisfy the middle"/four interchange law. It is in this sense that the notion of sesquicategory generalises that of $2$"/category.
	
	Call a subsesquicategory $\R \subseteq \s$ \emph{wide} if it contains all objects and morphisms of $\s$. It follows from the discussion above that $\R$ is a sub"/$2$"/category of $\s$ if its cells satisfy the middle"/four interchange law. Every sesquicategory contains two maximal such sub"/$2$"/categories consisting of its `strict' and `costrict' cells, respectively, as follows; this is Definition~2.1 of \cite{Lack10}.
	\begin{definition} \label{strict cells}
		Consider the composable cells $\phi$ and $\psi$ above. We say that $\psi$ is \emph{$\phi$"/strict}, and likewise that $\phi$ is \emph{$\psi$"/costrict}, whenever $\phi$ and $\psi$ satisfy the middle"/four interchange law above. The cell $\psi$ is called \emph{strict} if it is $\phi$"/strict for all cells $\phi$ horizontally precomposable with $\psi$. Dually, $\phi$ is \emph{costrict} if it is $\psi$"/costrict for all $\psi$ horizontally postcomposable with $\phi$.
	\end{definition}
	
	Since (co)strict cells are closed under whiskering, we obtain the following.
	\begin{proposition} \label{sub-2-categories}
		Every sesquicategory $\s$ contains wide sub"/$2$"/categories $\s_{\str}$ and $\s_{\costr}$ consisting of its strict and costrict cells, respectively.
	\end{proposition}
	
	\subsubsection*{Examples of sesquicategories}
	Here we describe two examples of sesquicategories. Their morphisms and cells are, respectively, unnatural transformations between $1$"/functors and lax natural transformations between $2$"/functors, whose definitions are recalled below. While not part of either example, we also recall the notion of modification of natural transformations, which is used in \secref{density section}.
	\begin{definition} \label{unnatural transformation definition}
		An \emph{unnatural transformation} $\cell \phi fg$ of functors $f$, $\map gAC$ is a family $\map{\phi_x}{fx}{gx}$ of morphisms in $C$, one for each object $x \in A$, that is not required to satisfy any (naturality) conditions.
	\end{definition}
	\begin{example} \label{unnatural transformation}
		Large categories, functors and unnatural transformations form a sesquicategory $\Catun$. Strict unnatural transformations $\cell \psi h{\map kCE}$ in $\Catun$ are precisely the natural transformations; to see this precompose $\psi$ with $\phi$ with $A = 1$ the terminal category. The only costrict unnatural transformations are the identity transformations $\cell{\id_f}ff$. To see this take $\psi$ the universal unnatural transformation with $E = \2 \Box C$ given by the `funny tensor product'; see \appref{comma objects appendix} below or Section~5 of \cite{Lack10}.
	\end{example}
	
	\begin{definition} \label{lax natural transformation definition}
		Consider parallel $2$"/functors $f$, $\map g AC$ of $2$"/categories. A \emph{lax natural transformation} $\cell \phi fg$ consists of a \emph{component morphism} $\map{\phi_x}{fx}{gx}$ for each object $x \in A$ as well as a \emph{lax naturality cell}
		\begin{displaymath}
			\begin{tikzpicture}
				\matrix(m)[math35]{fx & gx \\ fy & gy \\};
				\path[map]	(m-1-1) edge node[above] {$\phi_x$} (m-1-2)
														edge node[left] {$fs$} (m-2-1)
										(m-1-2) edge node[right] {$gs$} (m-2-2)
										(m-2-1) edge node[below] {$\phi_y$} (m-2-2)
										(m-1-2) edge[cell, shorten >= 9pt, shorten <= 9pt] node[below right] {$\phi_s$} (m-2-1);
			\end{tikzpicture}
		\end{displaymath}
		for each morphism $\map sxy$ in $A$. These cells satisfy three axioms, ensuring naturality with respect to the cells of $A$ and compatibility both with the identity morphisms and with the composition of morphisms in $A$; see e.g.\ Section~I,2.4 of \cite{Gray74} (where they are called ``quasi"/natural transformations'') or Definition~4.2.1 of \cite{Johnson-Yau21}.
	\end{definition}
	
	\begin{definition} \label{modification}
		Consider parallel $2$"/functors $f$, $\map g AC$ and lax natural transformations $\phi$, $\cell\psi fg$. A \emph{modification} $M\colon \phi \Rrightarrow \psi$ consists of a family of \emph{component cells} $\cell{M_x}{\phi_x}{\psi_x}$, one for each object $x \in A$, that satisfy the naturality axiom
		\begin{displaymath}
			\begin{tikzpicture}[textbaseline]
				\matrix(m)[math35, column sep={3.7em,between origins}]{fx & gx \\ fy & gy \\};
				\path[map]	(m-1-1) edge[bend left=30] node[above] {$\phi_x$} (m-1-2)
														edge[bend right=30] node[below, inner sep=1pt] {$\psi_x$} (m-1-2)
														edge node[left] {$fs$} (m-2-1)
										(m-1-2) edge node[right] {$gs$} (m-2-2)
										(m-2-1) edge[bend right=30] node[below] {$\psi_y$} (m-2-2);
				\path[transform canvas={yshift=-0.75em}]	(m-1-2) edge[cell, shorten >= 9pt, shorten <= 9pt] node[below right, inner sep=1pt] {$\psi_s$} (m-2-1);
				\path				([yshift=8pt]$(m-1-1)!0.35!(m-1-2)$) edge[cell] node[right, inner sep=1.5pt] {$M_x$} ([yshift=-8pt]$(m-1-1)!0.35!(m-1-2)$);
			\end{tikzpicture} = \begin{tikzpicture}[textbaseline]
				\matrix(m)[math35, column sep={3.7em,between origins}]{fx & gx \\ fy & gy \\};
				\path[map]	(m-1-1) edge[bend left=30] node[above] {$\phi_x$} (m-1-2)
														edge node[left] {$fs$} (m-2-1)
										(m-1-2) edge node[right] {$gs$} (m-2-2)
										(m-2-1) edge[bend left=30] node[above, inner sep=1pt] {$\phi_y$} (m-2-2)
														edge[bend right=30] node[below] {$\psi_y$} (m-2-2);
				\path[transform canvas={yshift=0.75em}]	(m-1-2) edge[cell, shorten >= 9pt, shorten <= 9pt] node[above left, inner sep=1pt] {$\phi_s$} (m-2-1);
				\path				([yshift=8pt]$(m-2-1)!0.35!(m-2-2)$) edge[cell] node[right, inner sep=1.5pt] {$M_y$} ([yshift=-8pt]$(m-2-1)!0.35!(m-2-2)$);
			\end{tikzpicture}
		\end{displaymath}
		for each morphism $\map sxy$ in $A$.
	\end{definition}
	
	\begin{example} \label{lax natural transformation}
		Large $2$"/categories (that is, categories with a large set of objects and large hom"/categories), $2$"/functors and lax natural transformations form a sesquicategory that we denote by $\und{\TwoCatLax}$. The postwhiskering $\cell{h \of \phi}{h \of f}{h \of g}$ in $\und{\TwoCatLax}$ is given by $(h \of \phi)_x = h(\phi_x)$ and $(h \of \phi)_s = h(\phi_s)$ while the prewhiskering $\cell{\psi \of f}{h \of f}{k \of f}$ is given by $(\psi \of f)_x = \psi_{fx}$ and $(\psi \of f)_s = \psi_{fs}$. Similar to the situation in the previous example, in Section~5 of \cite{Lack10} it is shown that lax natural transformations that are strict in $\und{\TwoCatLax}$ are precisely the $2$"/natural transformations, i.e.\ those with identity cells as naturality cells, and that the costrict ones are precisely those with identity morphisms as components.
				
		$\und{\TwoCatLax}$ contains as a wide subsesquicategory the sesquicategory $\und{\TwoCatPs}$ of large $2$"/categories, $2$"/functors and \emph{pseudonatural transformations} $\phi$, whose naturality cells $\phi_s$ are invertible. Similar to $\und{\TwoCatLax}$ there is also a sesquicategory $\und{\TwoCatOpl}$ of \emph{oplax natural transformations}. For the latter see e.g.\ Section~3 of \cite{Lack10} (wherein \emph{icons} are defined to be oplax natural transformations with identity component morphisms). Notice however that it is not possible to form a sesquicategory of $2$"/categories, \emph{pseudofunctors} (see e.g.\ Definition~4.1.2 of \cite{Johnson-Yau21}) and any notion of natural transformation that allows non"/identity component morphisms. Indeed, postwhiskering such transformations with a pseudofunctor is not functorial.
	\end{example}
	
	\begin{remark} \label{lax 3-category}
		The reason that we use vertical bars `$\und{\dash}$' in the notation $\und{\TwoCatLax}$ above is that $\und{\TwoCatLax}$ forms the underlying sesquicategory of the ``lax $3$"/category'' $\TwoCatLax$ of large $2$"/categories, $2$"/functors, lax natural transformations and modifications. Section~4.3 of \cite{Lambert24} introduces the notion of lax $3$"/category as well as $\TwoCatLax$.
		
		However, it is not true that every lax $3$"/category contains an underlying sesquicategory, obtained by discarding the $3$"/cells. Indeed, as a counter example it is straightforward to construct a ``minimal'' lax $3$"/category in which horizontal precomposition of $2$"/cells with a morphism is merely laxly functorial, that is, not strictly functorial as required in a sesquicategory.
	\end{remark}
	
	\section{Comma objects in sesquicategories} \label{comma objects section}
	The notion of internal split fibration in a sesquicategory that will be introduced in \secref{internal split two-sided fibrations section} uses the notion of $\2$"/cotensor. Here we consider the notion of comma object in a sesquicategory, which generalises that of $\2$"/cotensor. The $1$"/dimensional universal property below is the same as the one satisfied by a comma object in a $2$"/category, see e.g.\ Section~1 of \cite{Street74b}. Since the latter does not involve horizontal compositions of cells, it applies verbatim to sesquicategories.
	\begin{definition} \label{comma object}
		Given two morphisms $\map fAC$ and $\map gBC$ in a sesquicategory $\s$, the \emph{$1$"/universal comma object} (shortly \emph{comma object}) of $f$ and $g$ consists of an object $f \slash g$ equipped with projections $\pi_A$ and $\pi_B$, as in the diagram below left, as well as a cell $\cell\pi{f \of \pi_A}{g \of \pi_B}$ that satisfies the following $1$"/dimensional universal property. Given any other cell $\phi$ in $\s$ as below in the middle there exists a unique morphism $\map{\phi'}X{f \slash g}$ such that $\pi_A \of \phi' = \phi_A$, \mbox{$\pi_B \of \phi' = \phi_B$} and $\pi \of \phi' = \phi$.

		\begin{displaymath}
			\begin{tikzpicture}[baseline]
				\matrix(m)[math35]{f \slash g & A \\ B & C \\};
				\path[map]	(m-1-1) edge node[above] {$\pi_A$} (m-1-2)
														edge node[left] {$\pi_B$} (m-2-1)
										(m-1-2) edge node[right] {$f$} (m-2-2)
										(m-2-1) edge node[below] {$g$} (m-2-2);
				\path				(m-1-2) edge[cell, shorten >= 9pt, shorten <= 9pt] node[below right] {$\pi$} (m-2-1);
			\end{tikzpicture} \qquad\qquad\qquad \begin{tikzpicture}[baseline]
				\matrix(m)[math35]{X & A \\ B & C \\};
				\path[map]	(m-1-1) edge node[above] {$\phi_A$} (m-1-2)
														edge node[left] {$\phi_B$} (m-2-1)
										(m-1-2) edge node[right] {$f$} (m-2-2)
										(m-2-1) edge node[below] {$g$} (m-2-2);
				\path				(m-1-2) edge[cell, shorten >= 9pt, shorten <= 9pt] node[below right] {$\phi$} (m-2-1);
			\end{tikzpicture} \qquad\qquad\qquad \begin{tikzpicture}[baseline]
				\matrix(m)[math35]{A^\2 \\ A \\};
				\path[map]	(m-1-1) edge[bend right=45] node[left] {$\src$} (m-2-1)
														edge[bend left=45] node[right] {$\tgt$} (m-2-1);
				\path				([xshift=-8pt]$(m-1-1)!0.55!(m-2-1)$) edge[cell] node[above] {$\delta$} ([xshift=9pt]$(m-1-1)!0.55!(m-2-1)$);
			\end{tikzpicture}
		\end{displaymath}
		
		The \emph{$\2$"/cotensor} $A^\2$ of an object $A$ in $\s$ is defined to be the comma object $A^\2 \dfn \id_A \slash \id_A$. To distinguish its two projections $A^\2 \rightrightarrows A$ we denote them by $\src$ and $\tgt$, respectively, while we denote its defining cell by $\delta$, as depicted above right.
		
		The $1$"/dimensional universal properties of $f \slash g$ and $A^\2$ are succinctly stated as the existence of natural bijections of classes
		\begin{displaymath}
			\und\s(X, f \slash g) \iso \und{\s(\id_X, f) \slash \s(\id_X, g)} \qquad \text{and} \qquad \und\s(X, A^\2) \iso \und{\s(X, A)^\2}
		\end{displaymath}
		where $\s(\id_X, f) \slash \s(\id_X, g)$ denotes the usual comma $1$"/category of the functors $\map{\s(\id_X, f)}{\s(X, A)}{\s(X, C)}$ and $\map{\s(\id_X, g)}{\s(X, B)}{\s(X, C)}$ and $\s(X, A)^\2$ denotes the usual arrow $1$"/category of $\s(X, A)$.
	\end{definition}
	
	\subsubsection*{Examples of comma objects}
	Besides the $1$"/dimensional universal property above, the two examples of comma objects below satisfy different $2$"/dimensional universal properties; see \appref{comma objects appendix} below for details.
	
	\begin{example} \label{comma objects for unnatural transformations}
		The comma object $f \slash g$ in the sesquicategory $\Catun$ of unnatural transformations (\exref{unnatural transformation}) has as objects triples $(a, u, b)$ with $a \in A$, $b \in B$ and $\map u{fa}{gb}$ in $C$. A morphism $(a_0, u_0, b_0) \to (a_1, u_1, b_1)$ is a pair $(s, t)$ of morphisms $\map s{a_0}{a_1}$ in $A$ and $\map t{b_0}{b_1}$ in $B$, without any condition; in particular, the square formed by $fs$, $gt$, $u_0$ and $u_1$ is not required to commute.
	\end{example}
	
	\begin{example} \label{lax comma 2-category}
		Let $f$, $\map gAC$ be $2$"/functors between large $2$"/categories. The comma object $f \slash g$ in the sesquicategory $\und{\TwoCatLax}$ of lax natural transformations (\exref{lax natural transformation}) is the following \emph{lax comma $2$"/category}; see e.g.\ Section I,2.5 of \cite{Gray74} or Construction~2.13 of \cite{Lambert24}. Its objects are triples $(a, u, b)$ with $a \in A$, $b \in B$ and $\map u{fa}{gb}$ in $C$. A morphism $\map{(s, \zeta, t)}{(a_0, u_0, b_0)}{(a_1, u_1, b_1)}$ in $f \slash g$ consists of morphisms $\map s{a_0}{a_1}$ in $A$ and $\map t{b_0}{b_1}$ in $B$ together with a cell $\zeta$ as on the left below. A cell $\cell{(\sigma, \tau)}{(s_0, \zeta_0, t_0)}{(s_1, \zeta_1, t_1)}$ consists of cells $\cell\sigma{s_0}{s_1}$ in $A$ and $\cell\tau{t_0}{t_1}$ in $B$ satisfying the identity on the right below. The $2$"/categorical structure of $f \slash g$ is induced by those of $A$, $B$ and $C$.
		
		The $\2$"/cotensor $A^\2 \dfn \id_A \slash \id_A$ of a $2$"/category $A$ is called the \emph{lax arrow $2$"/category} associated to $A$.
		\begin{displaymath}
			\begin{tikzpicture}[textbaseline]
				\matrix(m)[math35]{fa_0 & gb_0 \\ fa_1 & gb_1 \\};
				\path[map]	(m-1-1) edge node[above] {$u_0$} (m-1-2)
														edge node[left] {$fs$} (m-2-1)
										(m-1-2) edge node[right] {$gt$} (m-2-2)
										(m-2-1) edge node[below] {$u_1$} (m-2-2);
				\path				(m-1-2) edge[cell, shorten >= 9pt, shorten <= 9pt] node[below right] {$\zeta$} (m-2-1);
			\end{tikzpicture} \qquad\qquad\qquad \begin{tikzpicture}[textbaseline]
				\matrix(m)[math35, column sep={3.7em,between origins}]{fa_0 & gb_0 \\ fa_1 & gb_1 \\};
				\path[map]	(m-1-1) edge node[above] {$u_0$} (m-1-2)
														edge[bend right=40] node[left] {$fs_1$} (m-2-1)
														edge[bend left=40] node[above right, inner sep=0pt] {$fs_0$} (m-2-1)
										(m-1-2) edge[bend left=40] node[right] {$gt_0$} (m-2-2)
										(m-2-1) edge node[below] {$u_1$} (m-2-2);
				\path				([xshift=9pt]$(m-1-1)!0.55!(m-2-1)$) edge[cell] node[above] {$f\sigma$} ([xshift=-7.5pt]$(m-1-1)!0.55!(m-2-1)$)
										(m-1-2) edge[cell, shorten >= 10pt, shorten <= 10pt, transform canvas={xshift=1.2em}] node[below right, inner sep=0.5pt] {$\zeta_0$} (m-2-1);
			\end{tikzpicture} = \begin{tikzpicture}[textbaseline]
				\matrix(m)[math35, column sep={3.7em,between origins}]{fa_0 & gb_0 \\ fa_1 & gb_1 \\};
				\path[map]	(m-1-1) edge node[above] {$u_0$} (m-1-2)
														edge[bend right=40] node[left] {$fs_1$} (m-2-1)
										(m-1-2) edge[bend right=40] node[below left, inner sep=0pt] {$gt_1$} (m-2-2)
														edge[bend left=40] node[right] {$gt_0$} (m-2-2)
										(m-2-1) edge node[below] {$u_1$} (m-2-2);
				\path				([xshift=9pt]$(m-1-2)!0.55!(m-2-2)$) edge[cell] node[above] {$g\tau$} ([xshift=-7.5pt]$(m-1-2)!0.55!(m-2-2)$)
										(m-1-2) edge[cell, shorten >= 10pt, shorten <= 10pt, transform canvas={xshift=-1.1em}] node[above left, inner sep=1pt] {$\zeta_1$} (m-2-1);
			\end{tikzpicture}
		\end{displaymath}
	\end{example}
	
	\subsubsection*{Conical limits}
	By `conical limits' in sesquicategories $\s$ we shall mean limits of $1$"/dimensional diagrams that satisfy the same $1$"/ and $2$"/dimensional universal properties as satisfied by conical limits in $2$"/categories (see e.g.\ Section~3.8 of \cite{Kelly82}). In detail, given an indexing $1$"/category $\catvar I$, by a \emph{$1$"/diagram} $D$ in $\s$ we mean a $1$"/functor $\map D{\catvar I}{\und\s}$ into the $1$"/category $\und\s$ underlying $\s$. A \emph{$1$"/cone} $(C, \xi)$ over $D$ in $\und\s$ is a natural transformation $\cell\xi{\Delta C}D$ as usual, where $\Delta C$ denotes the constant diagram in $\und\s$ at the vertex $C \in \s$. Given another $1$"/cone $(C, \zeta)$ over $D$ with the same vertex $C$, by a \emph{modification} \mbox{$\map M{(C,\xi)}{(C,\zeta)}$} we mean a family of cells $\cell{M_I}{\xi_I}{\zeta_I}$ in $\s$ between the components $\xi_I$ and \mbox{$\map{\zeta_I}C{DI}$}, indexed by the objects $I \in \catvar I$, that is natural in the sense that $Du \of M_I = M_J$ for all morphisms $\map uIJ$ in $\catvar I$. In the latter equation $\of$ is the whiskering of $\s$.
	\begin{definition} \label{conical limit}
		A \emph{conical limit} of a $1$"/diagram $\map D{\catvar I}{\und \s}$ is a $1$"/cone $(L, \pi)$ over $D$ satisfying the following two universal properties. Any $1$"/cone $(C, \xi)$ over $D$ factors uniquely through $\pi$ as a morphism $\map{\xi'}CL$, that is, $\pi_I \of \xi' = \xi_I$ for all $I \in \catvar I$. Any modification $\map M{(C, \xi)}{(C, \zeta)}$ of $1$"/cones, in the sense above, factors uniquely through $\pi$ as a cell $\cell{M'}{\xi'}{\zeta'}$ such that $\pi_I \of M' = M_I$ for all $I \in \catvar I$.
	\end{definition}
	
	That conical limits in the sense above are in fact conical enriched limits in the usual sense becomes clear in \exref{conical limit as enriched limit} below.
	
	\begin{convention} \label{conical limits convention}
		All $1$"/dimensional limits that we consider in sesquicategories (terminal objects, products, pullbacks, etc.) we shall understand to be conical limits in the sense above.
	\end{convention}
	
	Let $A \xrar f C \xlar g B$ be a cospan in a sesquicategory and assume that the $\2$"/cotensor $C^\2$ exists. We denote by $A \times_C C^\2 \times_C B$ the iterated pullback of the diagram $A \xrar f C \xlar{\src} C^\2 \xrar{\tgt} C \xlar g B$. The following result, where $\pi_{C^\2}$ denotes the projection onto $C^\2$, is easily checked.
	\begin{proposition} \label{comma objects from 2-cotensors}
		The iterated pullback $A \times_C C^\2 \times_C B$ forms the comma object $f \slash g$, being defined as such by the universal cell $\delta \of \pi_{C^\2}$.
	\end{proposition}
	
	\section{Profunctors in categories}\label{internal profunctors section}
	Here we recall from Example~2.10 of \cite{Cruttwell-Shulman10} (or Example~2.9 of \cite{Koudenburg20}) the notion of internal category in a $1$"/category $\E$ with pullbacks, as well as that of internal profunctor $A \brar B$ between such categories $A$ and $B$. In the next section we will see that $\2$"/cotensors $A^\2$ in a sesquicategory $\s$ (\defref{comma object}) form internal categories in the underlying $1$"/category $\und\s$, and define split fibrations in $\s$ to be internal profunctors $A^\2 \brar B^\2$ in $\und\s$.
	
	We will regard internal profunctors as being horizontal morphisms in an `augmented virtual double category' in the sense of \cite{Koudenburg20} as follows. First recall from e.g.\ Definition~2.1 of \cite{Cruttwell-Shulman10} the notion of a \emph{virtual double category}, consisting of objects $A$, $B$, $C,\dotsc$, \emph{vertical} morphisms $\map fAC$ as well as \emph{horizontal} morphisms $\hmap JAB$, and \emph{multicells} (or, shortly, \emph{cells}) $\phi$ of the form as on the left below, each with a single morphism $\hmap KCD$ as horizontal target and a (potentially empty) path $\ul J = (A_0 \xbrar{J_1} A_1 \dotsb A_{n-1} \xbrar{J_n} A_n)$ of morphisms as horizontal source. Vertical morphisms are equipped with a composition that is strictly unital and associative, while no composition is specified for horizontal morphisms. Cells are equipped with a unital and associative composition reminiscent of that of operads; see \cite{Cruttwell-Shulman10} or \cite{Koudenburg20} for details.
	
	The notion of \emph{augmented virtual double category} (Definition~1.2 of \cite{Koudenburg20}) extends that of virtual double category by including cells $\psi$ with empty horizontal targets, as in the middle below. We call cells of the form $\cell\phi{\ul J}K$ \emph{unary} and cells of the form $\cell\psi{\ul J}C$ \emph{nullary}. Nullary cells $\cell\psi AC$ with empty horizontal source, like on the right below, are called \emph{vertical} and denoted by $\cell\psi fg$. A \emph{horizontal} cell is one whose vertical morphisms $f$ and $g$ are identity morphisms. The sense in which cells compose becomes clear in \propref{profunctors form an augmented virtual double category} below, or see Definition~1.2 of \cite{Koudenburg20} for details.
	\begin{displaymath}
		\begin{tikzpicture}[baseline]
			\matrix(m)[math35]{A_0 & A_1 & A_{n'} & A_n \\ C & & & D \\};
			\path[map]	(m-1-1) edge[barred] node[above] {$J_1$} (m-1-2)
													edge node[left] {$f$} (m-2-1)
									(m-1-3) edge[barred] node[above] {$J_n$} (m-1-4)
									(m-1-4) edge node[right] {$g$} (m-2-4)
									(m-2-1) edge[barred] node[below] {$K$} (m-2-4);
			\path[transform canvas={xshift=1.75em}]	(m-1-2) edge[cell] node[right] {$\phi$} (m-2-2);
			\draw				($(m-1-2)!0.5!(m-1-3)$) node {$\dotsb$};
		\end{tikzpicture} \quad\quad\hspace{-0.2em} \begin{tikzpicture}[baseline]
			\matrix(m)[math35, column sep={1.75em,between origins}]
				{A_0 & & A_1 & & A_{n'} & & A_n \\ & & & C & & & \\};
			\path[map]	(m-1-1) edge[barred] node[above] {$J_1$} (m-1-3)
													edge node[below left] {$f$} (m-2-4)
									(m-1-5) edge[barred] node[above] {$J_n$} (m-1-7)
									(m-1-7) edge node[below right] {$g$} (m-2-4);
			\path				(m-1-4) edge[cell] node[right] {$\psi$} (m-2-4);
			\draw				(m-1-4) node {$\dotsb$};
		\end{tikzpicture} \quad\quad\hspace{-0.2em} \begin{tikzpicture}[baseline]
			\matrix(m)[math35]{A \\ C \\};
			\path[map]	(m-1-1) edge[bend right=45] node[left] {$f$} (m-2-1)
													edge[bend left=45] node[right] {$g$} (m-2-1);
			\path				(m-1-1) edge[cell] node[right] {$\psi$} (m-2-1);
		\end{tikzpicture}
	\end{displaymath}
	
		Let $\E$ be a $1$"/category with iterated pullbacks. The following definition lists the data of an augmented virtual double category $\inProf{\E}$ of \emph{profunctors in $\E$}, whose remaining structure is supplied in \propref{profunctors form an augmented virtual double category} below. The objects, vertical morphisms and vertical cells of $\inProf{\E}$ are categories in $\E$, functors in $\E$ and their transformations, respectively, in the classical sense; see e.g.\ Section~1 of \cite{Street74b}. The virtual double category of profunctors in a $1$"/category is described in Example~2.10 of \cite{Cruttwell-Shulman10}. Recall the notation $n' \dfn n - 1$ for any integer $n \geq 1$.
	\begin{definition} \label{internal profunctor}
		Let $\E$ be a $1$"/category with iterated pullbacks $J_1 \times_{A_1} \dotsb \times_{A_{n'}} J_n$ chosen for all finite iterated spans $(A_0 \leftarrow J_1 \rightarrow A_1 \leftarrow J_2 \rightarrow \dotsb \leftarrow J_n \rightarrow A_n)$.
		\begin{enumerate}[label =-]
			\item A \emph{category $A = (A, \alpha, \bar\alpha, \tilde\alpha)$ in} $\E$ is a span $A \xleftarrow{\src} \alpha \xrightarrow{\tgt} A$ in $\E$ equipped with \emph{composition} and \emph{unit} morphisms of spans $\map{\bar\alpha}{\alpha \times_A \alpha}\alpha$ and $\map{\tilde\alpha}A\alpha$, that satisfy the associativity axiom $\bar\alpha \of (\bar\alpha \times_A \id_\alpha) = \bar\alpha \of (\id_\alpha \times_A \bar\alpha)$ and the unit axioms $\bar\alpha \of (\tilde\alpha \of \src, \id_\alpha) = \id_\alpha = \bar\alpha \of (\id_\alpha, \tilde\alpha \of \tgt)$.
			\item	A \emph{functor} $\map fAC$ of categories in $\E$ is a morphism $\map fAC$ equipped with a morphism of spans $\map{\bar f}\alpha\gamma$ that satisfies the associativity and unit axioms $\bar\gamma \of (\bar f \times_f \bar f) = \bar f \of \bar \alpha$ and $\tilde \gamma \of f = \bar f \of \tilde \alpha$.
			\item A \emph{vertical cell} $\cell\psi fg$ as on the right above is a \emph{transformation} $f \Rar g$ in $\E$, that is, a morphism of spans $\map\psi A\gamma$ satisfying the external equivariance axiom $\bar\gamma \of (\bar f, \psi \of \tgt) = \bar\gamma \of (\psi \of \src, \bar g)$.
			\item A \emph{profunctor} $\hmap JAB$ between categories in $\E$ is a span $A \xleftarrow{j_A} J \xrightarrow{j_B} B$ that is equipped with left and right \emph{action} morphisms of spans $\map\lambda{\alpha \times_A J}J$ and $\map\rho{J \times_B \beta}J$ satisfying the usual associativity, unit and compatibility axioms for bimodules:
			\begin{align*}
				\lambda \of (\bar\alpha \times_A \id_J) & = \lambda \of (\id_\alpha \times_A \lambda); & \rho \of (\id_J \times_B \bar\beta) &= \rho \of (\rho \times_B \id_\beta); \\
				\lambda \of (\tilde\alpha \of j_A, \id_J) & = \id_J = \rho \of (\id_J, \tilde\beta \of j_B); &\rho \of (\lambda \times_B \id_\beta) &= \lambda \of (\id_\alpha \times_A \rho).
			\end{align*}
			\item	A \emph{unary cell} $\cell\phi{(J_1, \dotsc, J_n)}K$ of functors and profunctors in $\E$, as on the left above and where $n \geq 1$, is a morphism of spans $\map\phi{J_1 \times_{A_1} \dotsb \times_{A_{n'}} J_n}K$ satisfying the external left and right equivariance axioms
			\begin{align*}
				\phi \of (\lambda_{J_1} \times_{\id_{A_1}} \id_{J_2} \times_{\id_{A_2}} \dotsb \times_{\id_{A_{n'}}} \id_{J_n}) &= \lambda_K \of (\bar f \times_f \phi) \\
				\phi \of (\id_{J_1} \times_{\id_{A_1}} \dotsb \times_{\id_{A_{n''}}} \id_{J_{n'}} \times_{\id_{A_{n'}}} \rho_{J_n}) &= \rho_K \of (\phi \times_g \bar g)
			\end{align*}
			and the internal equivariance axioms
			\begin{multline*}
				\phi \of (\id_{J_1} \times_{\id_{A_1}} \dotsb \times_{\id_{A_{i'''}}} \id_{J_{i''}} \times_{\id_{A_{i''}}} \rho_{J_{i'}} \times_{\id_{A_{i'}}} \id_{J_i} \times_{\id_{A_i}} \dotsb \times_{\id_{A_{n'}}} \id_{J_n}) \\
					= \phi \of (\id_{J_1} \times_{\id_{A_1}} \dotsb \times_{\id_{A_{i''}}} \id_{J_{i'}} \times_{\id_{A_{i'}}} \lambda_{J_i} \times_{\id_{A_i}} \id_{J_{i+1}} \times_{\id_{A_{i+1}}} \dotsb \times_{\id_{A_{n'}}} \id_{J_n})
			\end{multline*}
			for $2 \leq i \leq n$.
			\item	A \emph{nullary cell} $\cell\psi{(J_1, \dotsc, J_n)}C$, as in the middle above and where $n \geq 1$, is a morphism of spans $\map\phi{J_1 \times_{A_1} \dotsb \times_{A_{n'}} J_n}\gamma$ satisfying equivariance axioms analogous to those above with $\bar\gamma$ replacing $\lambda_K$ and $\rho_K$.
			\item A \emph{$(0,1)$"/ary cell} $\cell\phi AK$ below is a morphism of spans $\map\phi AK$ satisfying the external equivariance axiom $\lambda_K \of (\bar f, \phi \of \tgt) = \rho_K \of (\phi \of \src, \bar g)$.
			\begin{displaymath}
				\begin{tikzpicture}[baseline]
					\matrix(m)[math35, column sep={1.75em,between origins}]{& A & \\ C & & D \\};
					\path[map]	(m-1-2) edge[transform canvas={xshift=-1pt}] node[left] {$f$} (m-2-1)
															edge[transform canvas={xshift=1pt}] node[right] {$g$} (m-2-3)
											(m-2-1) edge[barred] node[below] {$K$} (m-2-3);
					\path				(m-1-2) edge[cell, transform canvas={yshift=-0.25em}] node[right, inner sep=2.5pt] {$\phi$} (m-2-2);
				\end{tikzpicture}
			\end{displaymath}
		\end{enumerate}
	\end{definition}

	Notice that for any category $A$ in $\E$ the span $I_A \dfn (A \xlar{\src} \alpha \xrar{\tgt} A)$ forms a profunctor in $\E$ when equipped with the composition morphism $\bar\alpha$ as its actions; $I_A$ is called the \emph{unit} profunctor of $A$. In terms of unit profunctors the nullary cells $\cell\psi{\ul J}C$ above can be alternatively defined as being unary cells $\cell\psi{\ul J}{I_C}$.
	
	The following is well"/known; see e.g. Section~1 of \cite{Street74b}.
	\begin{proposition} \label{categories form a 2-category}
		Let $\E$ be a $1$"/category with chosen iterated pullbacks. Categories in $\E$, their functors and the transformations between them, as defined above, form a $2$"/category $\inCat\E$.
	\end{proposition}
	
	\subsubsection*{Augmented virtual double categories of internal profunctors}
	Here we define a composition for the cells of profunctors in a $1$"/category $\E$, in the sense of \defref{internal profunctor}. The proposition below (which is Example~2.9 of \cite{Koudenburg20}) asserts that this composition equips categories, functors, profunctors and their cells in $\E$ with the structure of an augmented virtual double category, which we denote $\inProf\E$. We will use the following notation for iterated pullbacks.
	
	Throughout this paper we consider the notion of iterated span in $\E$ to include objects $A$ of $\E$, which we think of as \emph{empty} iterated spans $(A)$. Given a non"/empty iterated span $\ul J = (A_0 \leftarrow J_1 \rightarrow \dotsb \leftarrow J_n \rightarrow A_n)$ we write $\prod_{\dash} \ul J \dfn (A_0 \leftarrow J_1 \times_{A_1} \dotsb \times_{A_{n'}} J_n \rightarrow A_n)$ for the span induced by its iterated pullback. Extending this definition to empty iterated spans $\prod_{\dash} (A_0) \dfn (A_0 \xleftarrow{\id} A_0 \xrightarrow{\id} A_0)$, notice that the universal property of iterated pullbacks induces, for every path $\ull J = (\ul J_1, \dotsc, \ul J_n)$ of (possibly empty) iterated spans $\ul J_i$ of length $m_i \geq 0$, a canonical isomorphism of spans
	\begin{displaymath}
		\eta_{\ull J}\colon (\prod\nolimits_{\dash} \ul J_1) \times_{A_{1m_1}} \dotsb \times_{A_{n'm_{n'}}} (\prod\nolimits_{\dash} \ul J_n) \xrightarrow{\iso} \prod\nolimits_{\dash} (\ul J_1 \conc \dotsb \conc \ul J_n).
	\end{displaymath}
	Here $\ul J_1 \conc \dotsb \conc \ul J_n$ denotes the concatenation of $\ull J$ into a single iterated span, in which empty iterated spans are discarded.
	
	To define the composition of cells in $\inProf\E$, consider a path of cells of profunctors in $\E$ in the sense of \defref{internal profunctor} above. Each horizontal source $\ul J_i$ is a (possibly empty) path of profunctors in $\E$. Each horizontal target $\ul K_i$ is either a singleton profunctor $\ul K_i = (\hmap{K_i}{B_{i'}}{B_i})$, in the case that $\phi_i$ is a unary cell, or empty, that is, $\ul K_i = (B_{i'})$ is a category $B_{i'} = (B_{i'}, \beta_{i'}, \bar\beta_{i'}, \tilde\beta_{i'})$ in $\E$ (\defref{internal profunctor}), in the case that $\phi_i$ is a nullary cell. Treating both cases simultaneously, recall from the definition that underlying each $\phi_i$ is a morphism of spans $\map{\phi_i}{\prod_{\dash} \ul J_i}{K_i^\textup u}$, where $K_i^\textup u \dfn K_i$ if $\ul K_i = (K_i)$ is a singleton and $K_i^\textup u \dfn I_{B_{i'}}$ is the unit profunctor (\defref{internal profunctor}) if $\ul K_i = (B_{i'})$ is empty.
	\begin{displaymath}
		\begin{tikzpicture}
			\matrix(m)[math35, column sep={4.25em,between origins}]
				{	A_{10} & A_{1m_1} & A_{2m_2} &[1em] A_{n'm_{n'}} & A_{nm_n} \\
					B_0 & B_1 & B_2 & B_{n'} & B_n \\ };
			\path[map]	(m-1-1) edge[barred] node[above] {$\ul J_1$} (m-1-2)
													edge node[left] {$f_0$} (m-2-1)
									(m-1-2) edge[barred] node[above] {$\ul J_2$} (m-1-3)
													edge node[right] {$f_1$} (m-2-2)
									(m-1-3) edge node[right] {$f_2$} (m-2-3)
									(m-1-4) edge[barred] node[above] {$\ul J_n$} (m-1-5)
													edge node[left] {$f_{n'}$} (m-2-4)
									(m-1-5) edge node[right] {$f_n$} (m-2-5)
									(m-2-1) edge[barred] node[below] {$\ul K_1$} (m-2-2)
									(m-2-2) edge[barred] node[below] {$\ul K_2$} (m-2-3)
									(m-2-4) edge[barred] node[below] {$\ul K_n$} (m-2-5);
			\path[transform canvas={xshift=2.125em}]	(m-1-1) edge[cell] node[right] {$\phi_1$} (m-2-1)
									(m-1-2)	edge[cell] node[right] {$\phi_2$} (m-2-2)
									(m-1-4) edge[cell] node[right] {$\phi_n$} (m-2-4);
			\draw				($(m-1-3)!0.5!(m-2-4)$) node {$\dotsb$};
		\end{tikzpicture}
	\end{displaymath}
	
	According to Definition~1.2 of \cite{Koudenburg20}, the structure of the augmented virtual double category $\inProf\E$ is required to supply a composite $\psi \of \ul\phi = \psi \of (\phi_1, \dotsc, \phi_n)$ of the form below right, for each cell $\psi$ of profunctors in $\E$ below left, whose horizontal source is the concatenation $\ul K \dfn \ul K_1 \conc \ul K_2 \conc \dotsb \conc \ul K_n$ of the horizontal targets $\ul K_i$ of the cells $\phi_i$.
	\begin{displaymath}
		 \begin{tikzpicture}[baseline]
			\matrix(m)[math35, column sep={10em,between origins}]{B_0 & B_n \\ C & D \\};
			\path[map]	(m-1-1) edge[barred] node[above] {$\ul K_1 \conc \ul K_2 \conc \dotsb \conc \ul K_n$} (m-1-2)
													edge node[left] {$h$} (m-2-1)
									(m-1-2) edge node[right] {$k$} (m-2-2)
									(m-2-1) edge[barred] node[below] {$\ul L$} (m-2-2);
			\path[transform canvas={xshift=5em}]	(m-1-1) edge[cell] node[right] {$\psi$} (m-2-1);
		\end{tikzpicture} \qquad \qquad \qquad \begin{tikzpicture}[textbaseline]
						\matrix(m)[math35, column sep={9em,between origins}]{A_{10} & A_{nm_n} \\ C & D \\};
						\path[map]	(m-1-1) edge[barred] node[above] {$\ul J_1 \conc \ul J_2 \conc \dotsb \conc \ul J_n$} (m-1-2)
																edge node[left] {$h \of f_0$} (m-2-1)
												(m-1-2) edge node[right] {$k \of f_n$} (m-2-2)
												(m-2-1) edge[barred] node[below] {$\ul L$} (m-2-2);
						\path[transform canvas={xshift=1.35em}]	(m-1-1) edge[cell] node[right] {$\psi \of (\phi_1, \dotsc, \phi_n)$} (m-2-1);
		\end{tikzpicture}
	\end{displaymath}
	
	We define the composite to be the morphism of spans
	\begin{displaymath}
		\psi \of (\phi_1, \dotsc, \phi_n) \dfn \brks{\psi^\textup u \of (\phi_1 \times_{f_1} \dotsb \times_{f_{n'}} \phi_n) \of \inv\eta_{\ull J}},
	\end{displaymath}
	with $\eta_{\ull J}$ as defined above. Here $\map{\psi^\textup u}{\prod_{\dash} (K_1^\textup u, \dotsc, K_n^\textup u)}{L^\textup u}$ is the morphism of spans that is obtained from $\map\psi{\prod_{\dash} \ul K}{L^\textup u}$ by adding, to the iterated span $\ul K = (\ul K_1\conc \dotsb \conc \ul K_n)$ that is the horizontal source of $\psi$, the unit profunctors $K_j^\textup u = I_{B_{j'}} = \brks{B_{j'} \xlar{\src} \beta_{j'} \xrar{\tgt} B_{j'}}$, for each $1 \leq j \leq n$ with $\ul K_j$ empty, as follows. Compose $\psi$ with any suitable choice of actions among the actions $\map\lambda{\beta_{i'} \times_{B_{i'}} K_i}{K_i}$ and $\map\rho{K_i \times_{B_i} \beta_i}{K_i}$ on the singleton profunctors $\ul K_i = (K_i)$, and the actions $\map\lambda{\gamma \times_C L^\textup u}L^\textup u$ and $\map\rho{L^\textup u \times_D \delta}L^\textup u$ on the target profunctor $L^\textup u$. In the latter cases the actions must also be composed with the appropriate morphism of spans $\map{\bar h}{\beta_0}\gamma$ or $\map{\bar k}{\beta_n}\delta$. In the case that $\psi$ is a vertical cell or a $(0, 1)$"/ary cell (\defref{internal profunctor}), the only available actions are those of $L^\textup u$. It might be necessary to iterate a single action multiple times, for instance in the case that consecutive paths $\ul K_j$ and $\ul K_{j+1}$ are empty. It follows from the equivariance axioms for $\psi$ (\defref{internal profunctor}) that this construction of $\psi^\textup u$ is well"/defined: composing $\psi$ with any two choices of actions giving the correct source $\prod_{\dash}(K_1^\textup u, \dotsc, K_n^\textup u)$ results in the same morphism of spans $\psi^\textup u$. 
	
	Recall from our \conventionsref\ that we regard the composite $\psi \of (\phi_1, \dotsc, \phi_n)$, as defined above, as a \emph{vertical} composite. Vertical composition of cells in augmented virtual double categories induces a `horizontal composition' for nullary cells, which we will use in \secref{density section} below. Remember that every augmented virtual double category $\K$ contains a \emph{vertical $2$"/category} $V(\K)$ consisting of its objects, vertical morphisms and vertical cells; see Example~1.5 of \cite{Koudenburg20}. As mentioned above, the following result is Example~2.9 of op.\ cit.
	\begin{proposition} \label{profunctors form an augmented virtual double category}
		Let $\E$ be a $1$"/category with chosen iterated pullbacks. Categories in $\E$, their functors and profunctors, as well as the cells between them, as in \defref{internal profunctor} above, combine into an augmented virtual double category $\inProf\E$ as follows, with $V(\inProf\E) = \inCat\E$.
		\begin{enumerate}[label=-]
			\item Composition $\of$ of functors in $\E$ is induced by composition of morphisms in $\E$.
			\item	Vertical composition $\psi \of (\phi_1, \dotsc, \phi_n)$ of cells is defined as above.
			\item The horizontal identity cell $\id_J$, as on the left below, is the identity morphism of the span $J$. The vertical identity cell $\id_f$, as on the right, is the morphism of spans $\tilde\gamma \of f$ with $\tilde\gamma$ the unit morphism of $C$.
				\begin{displaymath}
					\begin{tikzpicture}[baseline]
						\matrix(m)[math35]{A & B \\ A & B \\};
						\path[map]	(m-1-1) edge[barred] node[above] {$J$} (m-1-2)
																edge node[left] {$\id_A$} (m-2-1)
												(m-1-2) edge node[right] {$\id_B$} (m-2-2)
												(m-2-1) edge[barred] node[below] {$J$} (m-2-2);
						\path[transform canvas={xshift=1.25em}]	(m-1-1) edge[cell] node[right] {$\id_J$} (m-2-1);
					\end{tikzpicture} \qquad\qquad\qquad\qquad \begin{tikzpicture}[baseline]
			\matrix(m)[math35]{A \\ C \\};
			\path[map]	(m-1-1) edge[bend right=45] node[left] {$f$} (m-2-1)
													edge[bend left=45] node[right] {$f$} (m-2-1);
			\path[transform canvas={xshift=-0.6em}]	(m-1-1) edge[cell] node[right, inner sep=3pt] {$\id_f$} (m-2-1);
		\end{tikzpicture}
				\end{displaymath}
		\end{enumerate}
	\end{proposition}
	
	\section{Split two-sided fibrations in sesquicategories} \label{internal split two-sided fibrations section}
	Let $\C$ be a $2$"/category with all pullbacks and $\2$"/cotensors. Street shows in Proposition~2 of \cite{Street74b} that $\2$"/cotensors $A^\2$ in $\C$ form categories in the underlying $1$"/category $\und{\C}$, and that the assignment $A \mapsto A^\2$ extends to a faithful $2$"/functor $\map{\dash^\2}\C{\inCat{\und{\C}}}$ (\propref{categories form a 2-category}). Section~2 of op.\ cit.\ then proceeds by defining \emph{split two"/sided fibrations} $A \brar B$ in $\C$ (called ``split bifibrations'' in the reference) to be (in our terms) profunctors in $\und{\C}$ of the form $A^\2 \brar B^\2$.
	
	In this section we apply Street's strategy, almost verbatim, to a sesquicategory $\s$ that has all pullbacks and $\2$"/cotensors. We show that the assignment \mbox{$A \mapsto A^\2$} in $\s$ extends to a faithful $2$"/functor $\map{\dash^\2}{\s_{\str}}{\inCat{\und\s} = V(\inProf{\und\s}})$, where $\s_{\str} \subseteq \s$ is the wide sub"/$2$"/category of strict cells (\propref{sub-2-categories}). We then define \emph{split two"/sided fibrations} $A \brar B$ in $\s$ to be profunctors in $\und{\s}$ of the form $A^\2 \brar B^\2$.
	
	Let $\s$ be a sesquicategory that has all pullbacks (Convention~\ref{conical limits convention}) and $\2$"/cotensors (\defref{comma object}). Let $A$ be an object with $\2$"/cotensor $A^\2$. The universal cell $\delta_A \dfn \delta$ defining $A^\2$, in the right"/hand side below, induces for the pair $(A, A^\2)$ a composition morphism $\map{\bar\alpha}{A^\2 \times_A A^\2}{A^\2}$ and a unit morphism $\map{\tilde\alpha}A{A^\2}$, in the sense of \defref{internal profunctor}. The composition $\bar\alpha$ is the unique factorisation of the left"/hand side below through $\delta$, as shown, while the unit $\tilde\alpha$ is the unique factorisation of the identity cell $\id_{\id_A}$ through $\delta$.
	\begin{displaymath}
		\begin{tikzpicture}[textbaseline]
				\matrix(m)[math35, column sep={4.5em,between origins}, row sep={4em,between origins}]{A^\2 \times_A A^\2 & A^\2 \\ A^\2 & A \\};
				\path[map]	(m-1-1) edge node[above] {$\pi_2$} (m-1-2)
														edge node[left] {$\pi_1$} (m-2-1)
										(m-1-2) edge[bend right=26] node[above left, inner sep=1pt] {$\src$} (m-2-2)
														edge[bend left=26] node[right] {$\tgt$} (m-2-2)
										(m-2-1) edge[bend right=26] node[below] {$\src$} (m-2-2)
														edge[bend left=26] node[above, inner sep=1pt] {$\tgt$} (m-2-2);
				\path				([xshift=-7.5pt]$(m-1-2)!0.5!(m-2-2)$) edge[cell] node[above] {$\delta$} ([xshift=8.5pt]$(m-1-2)!0.5!(m-2-2)$)
										([yshift=-7.5pt]$(m-2-1)!0.5!(m-2-2)$) edge[cell] node[right] {$\delta$} ([yshift=8.5pt]$(m-2-1)!0.5!(m-2-2)$);
				\draw				(-0.2,0.2) -- (0,0.2) -- (0,0.4);
			\end{tikzpicture} = \begin{tikzpicture}[textbaseline]
			\matrix(m)[math35]{A^\2 \times_A A^\2 \\ A^\2 \\ A \\};
						\path[map]	(m-2-1) edge[bend right=45] node[left] {$\src$} (m-3-1)
																edge[bend left=45] node[right] {$\tgt$} (m-3-1)
												(m-1-1) edge[dashed] node[right] {$\bar\alpha$} (m-2-1);
						\path				([xshift=-8pt]$(m-2-1)!0.55!(m-3-1)$) edge[cell] node[above] {$\delta$} ([xshift=9pt]$(m-2-1)!0.55!(m-3-1)$);
		\end{tikzpicture}
	\end{displaymath}
	
	 The following theorem, which is a variation on Proposition~2 of \cite{Street74b}, shows that $\bar\alpha$ and $\tilde\alpha$ make $(A, A^\2)$ into a category in $\und\s$.
	 
	\begin{theorem} \label{embedding of strict cells}
		Let $\s$ be a sesquicategory with chosen iterated pullbacks (Convention~\ref{conical limits convention}) and $\2$"/cotensors (\defref{comma object}). For each object $A$ the quadruple $A^\2 \dfn (A, A^\2, \bar\alpha, \tilde\alpha)$, as defined above, forms a category in the underlying $1$"/category $\und\s$ (\defref{internal profunctor}). The assignment $A \mapsto A^\2$ extends to a faithful locally full $2$"/functor (\propsref{sub-2-categories}{profunctors form an augmented virtual double category})
		\begin{displaymath}
			\map{\dash^\2}{\s_{\str}}{\inCat{\und\s} = V(\inProf{\und\s})}
		\end{displaymath}
		as follows (\defref{internal profunctor}):
		\begin{enumerate}[label=-]
			\item assign to a morphism $\map fAC$ of $\s$ the functor $f^\2 \dfn (f, \bar f)$ in $\und\s$, where $\map{\bar f}{A^\2}{C^2}$ denotes the unique factorisation of $f \of \delta_A$ through the cell $\delta_C$ that defines $C^\2$;
			\item assign to a strict cell $\cell\phi fg$ of $\s$ the transformation $\cell{\phi^\2}{f^\2}{g^\2}$ in $\und\s$ that is given by the unique factorisation $\map{\phi^\2}A{C^\2}$ of $\phi$ through $\delta_C$.
		\end{enumerate}
	\end{theorem}
	\begin{proof}
		Following Section~1 of \cite{Street74b}, we use that the usual Yoneda embedding \mbox{$\und\s \hookrightarrow \Cls^{\op{\und\s}}$} extends to a $2$"/functor $\emb Y{\inCat{\und\s}} \inCat{\Cls}^{\op{\und\s}}$ that is full and faithful as follows. The image $YA$ of a category $A = (A, \alpha, \bar\alpha, \tilde\alpha)$ in $\und\s$ is the contravariant functor that assigns to each object $X \in \und\s$ the category $(YA)(X)$ in $\Cls$ whose underlying span is \mbox{$\und\s(\id_X, \src), \und\s(\id_X, \tgt)\colon \und\s(X, \alpha) \rightrightarrows \und\s(X, A)$} and whose composition and unit morphisms are induced by $\und\s(X, \bar\alpha)$ and $\und\s(X, \tilde\alpha)$. See also Section~1 of \cite{Street17} for a more detailed description of a closely related construction. Conversely, a functor \mbox{$\op{\und{\s}} \to \inCat{\Cls}$} is in the image of $Y$ if and only if its composite with the forgetful functor $\inCat{\Cls} \to \Cls \times \Cls$, that sends a category in $\Cls$ to its classes of objects and morphisms, is a pair of contravariant hom"/functors $(\und\s(\dash, A), \und\s(\dash, \alpha))$ for some objects $A$, $\alpha \in \s$.
		
		To see that $A^\2 \dfn (A, A^\2, \bar\alpha, \tilde\alpha)$ forms a category in $\und\s$ we can apply the proof of Proposition~2 of \cite{Street74b} verbatim, as follows. First notice that the classes of objects of the arrow categories $\s(X,A)^\2$ form categories in $\Cls$, for each object $X \in \und\s$, with the projections $\und{\s(X, A)^\2} \rightrightarrows \und\s(X, A)$ as underlying spans, and that these combine into a functor $\map{\und\s(\dash, A)}{\op{\und\s}}{\inCat\Cls}$. Composing these spans with the natural bijection $\und\s(\dash, A^\2) \iso \und{\s(\dash, A)^\2}$, which is the universal property of $A^\2$ (\defref{comma object}), we obtain a functor in the image of $Y$, which shows that the pair of projections $\src, \tgt\colon A^\2 \rightrightarrows A$ carries the structure of a category in $\und\s$. That the latter coincides with $(\bar\alpha, \tilde\alpha)$ as defined above is readily checked. Similarly, $f^\2 \dfn (f, \bar f)$ corresponds to the natural family of functors $\map{\und\s(X, f)}{\und\s(X, A)}{\und\s(X, C)}$ in $\Cls$ given by the functions $\und\s(\id_X, f)$ and $\und{\s(\id_X, f)^\2}$.
		
		Next, consider a cell $\cell\phi f{\map gAC}$ in $\s$. Post"/composition with $\phi$ gives a natural family of morphisms of spans $\und\s(X, A) \to \und{\s(X, C)^\2}$ in $\Cls$. These morphisms form transformations $\cell{\und\s(X, \phi)}{\und\s(X, f)}{\und\s(X, g)}$ in $\Cls$, that is, they satisfy the external equivariance axiom (\defref{internal profunctor}), if and only if $\phi$ is strict (\defref{strict cells}). Under the natural bijections $\und\s(\dash, C^\2) \iso \und{\s(\dash, C)^\2}$ natural families of transformations $\und\s(X, f) \Rar \und\s(X, g)$ in $\Cls$ bijectively correspond to transformations $f^\2 \Rar g^\2$ in $\inCat{\und\s}$, with $\und\s(X, \phi)$ corresponding to $\cell{\phi^\2}{f^\2}{g^\2}$ as defined in the theorem.
	\end{proof}
	
	\begin{example}	\label{lax arrow 2-category as internal category}
		Consider the lax arrow $2$"/category $A^\2$ associated to a $2$"/category $A$ (\exref{lax comma 2-category}). The unit and composition $2$"/functors $\map{\tilde\alpha} A{A^\2}$ and $\map{\bar\alpha}{A^\2 \times_A A^\2}{A^\2}$, as supplied by the theorem, that make $A^\2$ into an internal category in $\und{\TwoCatLax}$ (\exref{lax natural transformation}), are defined as follows. In obtaining $\tilde\alpha$ and $\bar\alpha$, use the universal property of $A^\2$ (\defref{comma object}) and the vertical composition of lax natural transformations; see e.g.\ Definition~4.2.15 of \cite{Johnson-Yau21}.
		\begin{align*}
			\tilde\alpha(a) &= (a, \id_a, a) &&\bar\alpha(a, u, b, v, c) = (a, v \of u, c)\\
			\tilde\alpha(s) &= (s, \id_s, s) &&\bar\alpha(r, \zeta, s, \eta, t) = \bigpars{r, tv_0u_0 \xRar{\eta u_0} v_1su_0 \xRar{v_1\zeta} v_1u_1r, t} \\
			\tilde\alpha(\sigma) &= (\sigma, \sigma) &&\bar\alpha(\rho, \sigma, \tau) = (\rho, \tau) 
		\end{align*}
	\end{example}
		
	As promised, our notion of split two"/sided fibration in a sesquicategory is analogous to Street's notion of ``split bifibration'' in a $2$"/category, introduced in Section~2 of \cite{Street74b}.
	\begin{definition} \label{internal split two-sided fibration}
		Let $\s$ be a sesquicategory with chosen pullbacks and $\2$"/cotensors. A \emph{split two"/sided fibration $\hmap JAB$ from $A$ to $B$} in $\s$ is a profunctor $\hmap J{A^\2}{B^\2}$ in $\und\s$ (\defref{internal profunctor} and \thmref{embedding of strict cells}).
	\end{definition}
	
	We denote by $\inSpTwoFib\s \subset \inProf{\und\s}$ the locally full sub"/augmented virtual double category that is generated by categories in $\und\s$ of the form $A^\2 = (A, A^\2, \bar\alpha, \tilde\alpha)$, with $A \in \und\s$, and vertical morphisms of the form $f^\2 = (f, \bar f)$, with $f \in \und\s$, both as in the theorem above. When referring to $\inSpTwoFib\s$, we will identify the objects $A$ and morphisms $f$ of $\s$ with the categories $A^\2$ and functors $f^\2$ in $\und\s$ that they induce. Under this identification the horizontal morphisms $A \brar B$ of $\inSpTwoFib\s$ are precisely the split two"/sided fibrations $A \brar B$ in $\s$.
	
	\begin{definition} \label{internal split (op)fibration}
		If $\s$ has a terminal object $1$, then by a \emph{split fibration over $A$} in $\s$ we mean a profunctor of the form $A^\2 \brar 1$ in $\und\s$ and by a \emph{split opfibration over $B$} in $\s$ we mean a profunctor $1 \brar B^\2$ in $\und\s$.
		
		We denote by $\inSpFib\s$ the $1$"/category of split fibrations in $\s$, whose morphisms are cells in $\inSpTwoFib\s$ of the form below. The $1$"/category $\inSpOpFib\s$ of split opfibrations in $\s$ is defined analogously.
		\begin{displaymath}
			\begin{tikzpicture}
				\matrix(m)[math35]{A & 1 \\ C & 1 \\};
					\path[map]	(m-1-1) edge[barred] node[above] {$J$} (m-1-2)
															edge node[left] {$f$} (m-2-1)
											(m-2-1) edge[barred] node[below] {$K$} (m-2-2);
					\path				(m-1-2) edge[eq] (m-2-2)
											(m-1-1) edge[cell, transform canvas={xshift=1.75em}] node[right] {$\phi$} (m-2-1);
			\end{tikzpicture}
		\end{displaymath}
	\end{definition}
	
	Since the span underlying a split fibration $J$ over $A$ is necessarily of the form $A \xlar p J \xrar ! 1$, we will often denote $J$ by $\map pJA$. Analogously, split opfibrations $\hmap{(!, q)}1B$ will be denoted $\map qJB$. In terms of Subsection~4.23 of \cite{Koudenburg24} the $1$"/category $\inSpOpFib\s$ of split opfibrations in $\s$ is the \emph{horizontal slice category} $1 \hs \inSpTwoFib\s$ over $1$; see \cororef{equivalence from yoneda embedding} below.
	\begin{lemma}
		Let $\s$ be a sesquicategory with chosen iterated pullbacks, $\2$"/cotensors and a terminal object $1$. If $\hmap JAB$ is a split two"/sided fibration in $\s$ with underlying span $A \xlar p J \xrar q B$, then $p$ forms a split fibration over $A$ in $\s$ and $q$ forms a split opfibration over $B$ in $\s$.
	\end{lemma}
	\begin{proof}[Sketch of the proof.]
		The terminal morphism $\map !J1$ admits a unique right action $\map\rho{J \times_1 1^\2}J$ that, as a consequence of the unit axiom for $\rho$ (\defref{internal profunctor}), necessarily equals the projection onto $J$. It is easily checked that, together with the original left action $\map\lambda{A^\2 \times_A J}J$, this projection furnishes the span $A \xlar p J \xrar ! 1$ into a split fibration in $\s$.
	\end{proof}
	
	\subsubsection*{Restricting internal split two"/sided fibrations}
	A universal construction in $\inSpTwoFib\s$ that we will use is that of ``restricting a horizontal morphism along vertical morphisms''. Such restrictions are defined by universal cells as follows, which is Definition~4.1 of \cite{Koudenburg20}.
	\begin{definition} \label{cartesian cells}
		A cell $\cell\chi J{\ul K}$, as in the right-hand side below, is called \emph{cartesian} if any cell $\phi$, as on the left-hand side, factors uniquely through $\chi$ as a cell $\phi'$ as shown.
		\begin{displaymath}
			\begin{tikzpicture}[textbaseline]
				\matrix(m)[math35]{X_0 & X_1 & X_{n'} & X_n \\ A & & & B \\ C & & & D \\};
				\path[map]	(m-1-1) edge[barred] node[above] {$H_1$} (m-1-2)
														edge node[left] {$h$} (m-2-1)
										(m-1-3) edge[barred] node[above] {$H_n$} (m-1-4)
										(m-1-4) edge node[right] {$k$} (m-2-4)
										(m-2-1) edge node[left] {$f$} (m-3-1)
										(m-2-4) edge node[right] {$g$} (m-3-4)
										(m-3-1) edge[barred] node[below] {$\ul K$} (m-3-4);
				\draw				($(m-1-2)!0.5!(m-1-3)$) node {$\dotsc$};
				\path[transform canvas={yshift=-1.75em}]	($(m-1-1.south)!0.5!(m-1-4.south)$) edge[cell] node[right] {$\phi$} ($(m-2-1.north)!0.5!(m-2-4.north)$);
			\end{tikzpicture} \quad = \quad \begin{tikzpicture}[textbaseline]
				\matrix(m)[math35]{X_0 & X_1 & X_{n'} & X_n \\ A & & & B \\ C & & & D \\};
				\path[map]	(m-1-1) edge[barred] node[above] {$H_1$} (m-1-2)
														edge node[left] {$h$} (m-2-1)
										(m-1-3) edge[barred] node[above] {$H_n$} (m-1-4)
										(m-1-4) edge node[right] {$k$} (m-2-4)
										(m-2-1) edge[barred] node[below] {$J$} (m-2-4)
														edge node[left] {$f$} (m-3-1)
										(m-2-4) edge node[right] {$g$} (m-3-4)
										(m-3-1) edge[barred] node[below] {$\ul K$} (m-3-4);
				\draw				($(m-1-2)!0.5!(m-1-3)$) node {$\dotsc$};
				\path				($(m-1-1.south)!0.5!(m-1-4.south)$) edge[cell] node[right] {$\phi'$} ($(m-2-1.north)!0.5!(m-2-4.north)$)
										($(m-2-1.south)!0.5!(m-2-4.south)$) edge[cell, transform canvas={yshift=-2pt}] node[right] {$\chi$} ($(m-3-1.north)!0.5!(m-3-4.north)$);
			\end{tikzpicture}
		\end{displaymath}
	\end{definition}
	
	If the cartesian cell $\chi$ of the form above exists, then its horizontal source \mbox{$\hmap JAB$} is called the \emph{restriction} of $\hmap{\ul K}CD$ along $f$ and $g$, and denoted $\ul K(f, g) \dfn J$. If $\ul K = (C \xbrar K D)$ is a singleton  path, then we call $K(f, g)$ \emph{unary}; in the case that $\ul K = (C)$ is empty we call $C(f, g)$ \emph{nullary}. The nullary restriction \mbox{$\hmap{C(\id, \id)}CC$} is called the \emph{(horizontal) unit} of the object $C$ and denoted $I_C \dfn C(\id, \id)$; if $I_C$ exists, then we call the object $C$ \emph{unital}.
	
	In Section~10 of \cite{Koudenburg20} it is shown that the notion of an augmented virtual double category with all horizontal units is equivalent to that of a virtual double category with all horizontal units. This explains the first of the following definitions.
	\begin{definition} \label{augmented virtual equipment}
		An augmented virtual double category that has all horizontal units is called a \emph{unital virtual double category}. An augmented virtual double category that has all unary restrictions $K(f, g)$ is called an \emph{augmented virtual equipment}. Combining the latter notions, a \emph{unital virtual equipment} is an augmented virtual equipment that has all restrictions $\ul K(f, g)$.
	\end{definition}
	
	The following result is Example~4.9 of \cite{Koudenburg20}.
	\begin{proposition}
		The augmented virtual double category $\inProf\E$ of \propref{profunctors form an augmented virtual double category} is a unital virtual equipment.
	\end{proposition}
	
	Using that the embedding $\inSpTwoFib\s \hookrightarrow \inProf{\und\s}$ is full on horizontal morphisms and cells, together with \propref{comma objects from 2-cotensors}, the (proof of the) previous proposition implies the following.
	\begin{corollary} \label{internal split two-sided fibrations form a unital virtual equipment}
		The augmented virtual double category $\inSpTwoFib\s$ of \defref{internal split two-sided fibration} is a unital virtual equipment as follows.
		\begin{enumerate}[label=\textup{(\alph*)}]
			\item The unary restriction $K(f, g)$ is the iterated pullback $A \times_C K \times_D B$ equipped with the unique actions making the projection $\map{\pi_K}{K(f, g)}K$ into a unary cell \mbox{$K(f, g) \Rar K$} (\defref{internal profunctor}), which then forms the defining cartesian cell.
			\item The horizontal unit $I_C$ is the unit split two"/sided fibration $C \xlar{\src} C^\2 \xrar{\tgt} C$ (\defref{internal profunctor}) of $C$, defined as such by the cartesian nullary cell $C^\2 \Rar C$ given by $\id_{C^\2}$.
			\item The nullary restriction $C(f, g)$ is the comma object $A \xlar{\pi_A} f \slash g \xrar{\pi_B} B$ (\defref{comma object}) equipped with the unique actions making the factorisation $f \slash g \to C^\2$, of the cell $\pi$ defining $f \slash g$ through the cell $\delta$ defining $C^\2$, into a nullary cell $f \slash g \Rar C$, which then forms the defining cartesian cell.
		\end{enumerate}
	\end{corollary}
	
	It is shown in Lemma~5.9 of \cite{Koudenburg20} that the nullary cartesian cell $C^\2 \Rar C$ of part (b) above admits a horizontal $(0,1)$"/ary inverse $C \Rar C^\2$; whence the following.
	\begin{corollary} \label{correspondences induced by horizontal units}
		Vertical pre"/ and post"/composition, respectively, with the cartesian cell $C^\2 \Rar C$ defining $C^\2$ as the horizontal unit, induces bijective correspondences between the following types of cells of $\inSpTwoFib\s$, which leave the vertical morphisms unchanged:
		\begin{enumerate}[label=\textup{(\alph*)}]
			\item unary cells $(J_1, \dotsc, J_n) \Rar C^\2$ and nullary cells $(J_1, \dotsc, J_n) \Rar C$;
			\item $(1,n)$"/ary cells $C^\2 \Rar \ul K$ and $(0, n)$"/ary cells $C \Rar \ul K$ (here $n= 0$ or $1$).
		\end{enumerate}
		
		In particular, $(1, 1)$"/ary cells $A^\2 \Rar C^\2$ correspond bijectively to the vertical cells of $\inSpTwoFib\s$ which, by \defref{internal split two-sided fibration} and \thmref{embedding of strict cells}, are the strict cells of $\s$ (\defref{strict cells}).
	\end{corollary}

	\section[Split op-2-fibrations]{Split op"/$2$"/fibrations} \label{split op-2-fibrations section}
	Recall from \exref{lax natural transformation} the sesquicategory $\und{\TwoCatLax}$ of large $2$"/categories, $2$"/functors and lax natural transformations. In this section we recall a theorem of \cite{Lambert24} which shows, in our terms, that split opfibrations in the sesquicategory $\und{\TwoCatLax}$ are the split op"/$2$"/fibrations introduced in \cite{Buckley14}. In the next section we will use this to describe split two"/sided fibrations in $\und{\TwoCatLax}$, as well as the morphisms (i.e.\ the cells of $\inSpTwoFib{\und{\TwoCatLax}} \subset \inProf{\und{\TwoCatLax}}$ as in \defref{internal profunctor}) between them.
	
	We start by recalling the notion of split op"/$2$"/fibration from Definition~2.1 of \cite{Lucyshyn-Wright16}. It is obtained by reversing the direction of the morphisms in the notion of split $2$"/fibration, which was introduced in Definitions~2.1.6 and 2.1.10 of \cite{Buckley14}.
	\begin{definition} \label{split op-2-fibration}
		A \emph{split op"/2"/fibration} is a $2$"/functor $\map qJB$ equipped with a \emph{cleavage} $\rho(\dash, \dash)$, which supplies
		\begin{enumerate}
			\item[\textup{(ml)}] for each morphism $\map u{qi}b$ in $B$ a lift $\map{\rho(i, u)}i{u_!i}$ in $J$, that is, $q(\rho(i, u)) = u$;
			\item[\textup{(cl)}] for each cell $\beta$ in $B$ of the form as below left a lift $\cell{\rho(\beta, s)}{\beta^* s}s$ in $J$ as below in the middle, that is, $q(\rho(\beta, s)) = \beta$.
			\begin{displaymath}
				\begin{tikzpicture}[baseline]
					\matrix(m)[math35]{ qi \\ qj \\};
						\path[map]	(m-1-1) edge[bend right=45] node[left] {$u$} (m-2-1)
																edge[bend left=45] node[right] {$qs$} (m-2-1);
						\path				([xshift=-8pt]$(m-1-1)!0.55!(m-2-1)$) edge[cell] node[above] {$\beta$} ([xshift=9pt]$(m-1-1)!0.55!(m-2-1)$);
				\end{tikzpicture} \qquad\qquad\qquad \begin{tikzpicture}[baseline]
					\matrix(m)[math35]{ i \\ j \\};
						\path[map]	(m-1-1) edge[bend right=70] node[left] {$\beta^* s$} (m-2-1)
																edge[bend left=70] node[right] {$s$} (m-2-1);
						\path				([xshift=-8pt, yshift=-6pt]$(m-1-1)!0.5!(m-2-1)$) edge[cell] node[above] {$\rho(\beta, s)$} ([xshift=9pt, yshift=-6pt]$(m-1-1)!0.5!(m-2-1)$);
				\end{tikzpicture} \qquad\qquad\qquad \begin{tikzpicture}[baseline]
					\matrix(m)[math35]{ qj \\ qk \\};
						\path[map]	(m-1-1) edge[bend right=45] node[left] {$v$} (m-2-1)
																edge[bend left=45] node[right] {$qt$} (m-2-1);
						\path				([xshift=-8pt]$(m-1-1)!0.55!(m-2-1)$) edge[cell] node[above] {$\delta$} ([xshift=9pt]$(m-1-1)!0.55!(m-2-1)$);
				\end{tikzpicture}
			\end{displaymath}
		\end{enumerate}
		
		These lifts are required to satisfy two properties. Firstly they should satisfy the following \emph{splitting equations}, which ensure that they are strictly coherent with respect to the $2$"/categorical structures of $J$ and $B$:
		\begin{enumerate}
			\item[\textup{(sm$\of$)}] $\rho(i, v \of u) = \rho(u_! i, v) \of \rho(i, u)$ for any $qi \xrar u b \xrar v b'$ in $B$;
			\item[\textup{(sm1)}] $\rho(i, \id_{qi}) = \id_i$ for any $i \in J$;
			\item[\textup{(sc$\hc$)}] $\rho(\gamma \hc \beta, s) = \rho(\gamma, \beta^*s) \hc \rho(\beta, s)$ for any $u' \xRar\gamma u \xRar\beta qs$ in $B$ with $\beta$ as above;
			\item[\textup{(sc$\of$)}] $\rho(\delta \of \beta, t \of s) = \rho(\delta, t) \of \rho(\beta, s)$ for any cells $\beta$ and $\delta$ in $B$ of the form above;
			\item[\textup{(sc1)}] $\rho(\id_{qs}, s) = \id_s$ for any morphism $s \in J$.
		\end{enumerate}
		
		Secondly the lifts of morphisms $\rho(i, u)$ are required to be \emph{opcartesian} and the lifts of cells $\rho(\beta, s)$ are required to be \emph{cartesian}, in the sense below. Informally, this ensures that ``factorisations through $\rho(i, u)$ and $\rho(\beta, s)$ are created by $q$''.
		\begin{enumerate}
			\item[\textup{(om1)}] If $\map sij$ in $J$ is such that $\map{qs}{qi}{qj}$ factors through $\map u{qi}b$ as $\map vb{qj}$ in $B$, then there exists a unique lift $\map{\tilde v}{u_!i}j$ in $J$ such that $q\tilde v = v$ and $\tilde v \of \rho(i, u) = s$.
			\begin{displaymath}
				\begin{tikzpicture}
					\matrix(m)[math35, xshift=-3cm]{i & u_!i \\ & j \\};
						\path[map]	(m-1-1) edge node[above] {$\rho(i, u)$} (m-1-2)
																edge node[below left] {$s$} (m-2-2)
												(m-1-2) edge[dashed] node[right] {$\tilde v$} (m-2-2);
					
					\matrix(m)[math35, xshift=3cm]{qi & b \\ & qj \\};
						\path[map]	(m-1-1) edge node[above] {$u$} (m-1-2)
																edge node[below left] {$qs$} (m-2-2)
												(m-1-2) edge node[right] {$v$} (m-2-2);
					\draw	(-0.5cm, 0) edge[mapsto] node[above] {$q$} (0.5cm, 0);
				\end{tikzpicture}
			\end{displaymath}
			\item[\textup{(om2)}] Consider parallel morphisms $s$, $\map tij$ in $J$. If $\cell\sigma st$ is such that $q\sigma$ factors through $\map u{qi}b$ as $\cell\beta vw$ in $B$, as shown below right, then there exists a unique lift $\cell{\tilde\beta}{\tilde v}{\tilde w}$ in $J$ as below left such that $q\tilde\beta = \beta$ and $\tilde\beta \of \rho(i, u) = \sigma$.
				\begin{displaymath}
					\begin{tikzpicture}
						\matrix(m)[math35, column sep={5.5em,between origins}, xshift=-3cm]{i & u_!i \\ & j \\};
							\path[map]	(m-1-1) edge node[above] {$\rho(i, u)$} (m-1-2)
													(m-1-2) edge[bend right=30, dashed] node[above left, inner sep=0.5pt] {$\tilde v$} (m-2-2)
																	edge[bend left=30, dashed] node[right] {$\tilde w$} (m-2-2);
							\path[map, transform canvas={xshift=-2pt, yshift=-2pt}]	(m-1-1)	edge[bend right=20] node[below left] {$s$} (m-2-2)
																	edge[bend left=20] node[above right, inner sep=0.5pt] {$t$} (m-2-2);
							\path				([xshift=-7.5pt,yshift=-7.5pt]$(m-1-1)!0.5!(m-2-2)$) edge[cell] node[above left, inner sep=1pt] {$\sigma$} ([xshift=4.5pt,yshift=4.5pt]$(m-1-1)!0.5!(m-2-2)$)
													([xshift=-8pt, yshift=-3pt]$(m-1-2)!0.5!(m-2-2)$) edge[cell, dashed] node[above, inner sep=2pt] {$\tilde\beta$} ([xshift=9pt, yshift=-3pt]$(m-1-2)!0.5!(m-2-2)$);
						
						\matrix(m)[math35, column sep={5.5em,between origins}, xshift=3cm]{qi & b \\ & qj \\};
							\path[map]	(m-1-1) edge node[above] {$u$} (m-1-2)
													(m-1-2) edge[bend right=30] node[above left, inner sep=0.5pt] {$v$} (m-2-2)
																	edge[bend left=30] node[right] {$w$} (m-2-2);
							\path[map, transform canvas={xshift=-2pt, yshift=-2pt}]	(m-1-1)	edge[bend right=20] node[below left] {$qs$} (m-2-2)
																	edge[bend left=20] node[above right, inner sep=0pt] {$qt$} (m-2-2);
							\path				([xshift=-7.5pt,yshift=-7.5pt]$(m-1-1)!0.5!(m-2-2)$) edge[cell] node[above left, inner sep=1pt] {$q\sigma$} ([xshift=4.5pt,yshift=4.5pt]$(m-1-1)!0.5!(m-2-2)$)
													([xshift=-8pt, yshift=-3pt]$(m-1-2)!0.5!(m-2-2)$) edge[cell] node[above, inner sep=2pt] {$\beta$} ([xshift=9pt, yshift=-3pt]$(m-1-2)!0.5!(m-2-2)$);
													
						\draw	(-0.5cm, 0) edge[mapsto] node[above] {$q$} (0.5cm, 0);
					\end{tikzpicture}
				\end{displaymath}
			\item[\textup{(cc)}] Consider parallel morphisms $r$, $\map sij$ in $J$. If $\cell\sigma rs$ is such that $\cell{q\sigma}{qr}{qs}$ factors through $\cell\beta u{qs}$ as $\cell\alpha {qr}u$ in $B$, then there exists a unique lift $\cell{\tilde\alpha}r{\beta^*s}$ in $J$ such that $q\tilde\alpha = \alpha$ and $\tilde\alpha \hc \rho(\beta, s) = \sigma$.
			\begin{displaymath}
				\begin{tikzpicture}
					\matrix(m)[math35, xshift=-3cm]{r & \\ \beta^*s & s \\};
						\path	(m-1-1) edge[cellmap, dashed] node[left] {$\tilde \alpha$} (m-2-1)
													edge[cellmap] node[above right] {$\sigma$} (m-2-2)
									(m-2-1)	edge[cellmap] node[below] {$\rho(\beta, s)$} (m-2-2);
					
					\matrix(m)[math35, xshift=3cm]{qr & \\ u & qs \\};
						\path	(m-1-1) edge[cellmap] node[left] {$\alpha$} (m-2-1)
													edge[cellmap] node[above right] {$q\sigma$} (m-2-2)
									(m-2-1)	edge[cellmap] node[below] {$\beta$} (m-2-2);
									
					\draw	(-0.5cm, 0) edge[mapsto] node[above] {$q$} (0.5cm, 0);
				\end{tikzpicture}
			\end{displaymath}
		\end{enumerate}
	\end{definition}
	
	\begin{remark} \label{relation to 1-fibration}
		Notice that the $1$"/dimensional opcartesian property satisfied by $\rho(i, u)$, (om1) above, precisely means that $\rho(i, u)$ is an opcartesian morphism in $J$ with respect to the $1$"/functor $\map{\und q}{\und J}{\und B}$ underlying $q$, in the usual $1$"/categorical sense. In particular, $\und q$ is a split $1$"/opfibration whenever $q$ is a split op"/$2$"/fibration. Similarly, the cartesian property (cc) can be restated as $\cell{\rho(\beta, s)}{\beta^*s}s$ being a cartesian morphism in the hom"/category $J(i, j)$, with respect to the hom"/functor $\map{q_{i, j}}{J(i, j)}{B(qi, qj)}$. In particular, (cl) and (cc), together with the splitting equations (sc$\hc$)--(sc$1$), can be thought of as requiring that $q$ is ``locally a split $1$"/fibration''; see e.g.\ Definition~4.7 of \cite{Bakovic10}.
	\end{remark}
	
	\begin{remark}
		A (not necessarily split) \emph{op"/$2$"/fibration} (Definition~2.1 of \cite{Lucyshyn-Wright16}) is a $2$"/functor $\map qJB$ such that any morphism $\map u{qi}b$ in $B$ admits an opcartesian lift satisfying (om1) and (om2) above, any cell $\cell \beta u{qs}$ in $B$ admits a cartesian lift satisfying (cc), and cartesian cells in $J$ are preserved under horizontal composition. In particular, $q$ need not be equipped with a coherent choice of such lifts; indeed, a cleavage $\rho(\dash, \dash)$ satisfying (sm$\of$)--(sc1) above might not exists for $q$.
		
		In Section~2 of \cite{Hermida99} (the ``op"/dual'' of) a weaker notion of op"/$2$"/fibration is considered, by requiring the existence of (op)cartesian lifts but not the preservation property.
	\end{remark}
	
	\begin{example} \label{lax comma 2-category is split op-2-fibered}
		Consider the lax comma $2$"/category $f \slash g$ of a cospan of $2$"/functors $A \xrar f C \xlar g B$ (\exref{lax comma 2-category}). The projection $2$"/functor $\map{\pi_B}{f \slash g}B$ is a split op"/$2$"/fibration as follows. The opcartesian lift of $\bigbrks{\pi_B(a, fa \xrar i gb, b) = b} \xrar u b'$ is the morphism
		\begin{displaymath}
			(a, i, b) \xrar{(\id_a, \id_{gu \of i}, u)}(a, gu \of i, b') \nfd u_!(a, i, b).
		\end{displaymath}
		Let $\map{(s, \zeta, t)}{(a_0, i_0, b_0)}{(a_1, i_1, b_1)}$ be a morphism in $f \slash g$, with $\cell\zeta{gt \of i_0}{i_1 \of fs}$. The cartesian lift of a cell $t' \xRar \beta \bigbrks{\pi_B(s, \zeta, t) = t}$ is the cell
		\begin{displaymath}
			\beta^*(s, \zeta, t) \dfn \bigpars{s, (g\beta \of i_0) \hc \zeta, t'}  \xRar{(\id_s, \beta)} (s, \zeta, t).
		\end{displaymath}
		
		Next, consider a lax arrow $2$"/category $A^\2 \dfn \id_A \slash \id_A$ (\defref{comma object}). By the previous the target $2$"/functor $\map{\tgt}{A^\2}A$ is a split $2$"/opfibration.
	\end{example}
	
	Recall that reversing the directions of the morphisms or cells of a $2$"/category gives two types of duality. Thus associated to a $2$"/category $A$ are three $2$"/categories $\op A$, $\co A$ and $\coop A$, formed by reversing the direction either of the morphisms (`op') or cells (`co') of $A$, or of both (`coop'). These three constructions induce three kinds of $2$"/fibration as follows.
	\begin{definition} \label{split co-2-fibration}
		A $2$"/functor $\map pJA$ is a \emph{split co"/$2$"/fibration} if \mbox{$\map{\coop p}{\coop J}{\coop A}$} is a split op"/$2$"/fibration. Thus $p$ comes equipped with a \emph{cleavage} $\lambda(\dash, \dash)$ that supplies lifts $\map{\lambda(u,i)}{u^* i}i$ in $J$ of morphisms $\map ua{pi}$ in $A$ and lifts \mbox{$\cell{\lambda(s, \alpha)}s{\alpha_! s}$} in $J$ of cells $\cell\alpha{ps}v$ in $A$. The lifts $\lambda(u, i)$  are required to be \emph{cartesian} and the lifts $\lambda(s, \alpha)$ \emph{opcartesian}, that is, they satisfy properties dual to (om1), (om2) and (cc) above, while as a family they satisfy splitting equations analogous to those of \defref{split op-2-fibration}.
		
		A split \emph{2"/fibration} is a $2$"/functor $f$ such that $\op f$ is a split op"/$2$"/fibration while a \emph{split coop"/2"/fibration} is a $2$"/functor $g$ such that $\co g$ is a split op"/$2$"/fibration. Thus $f$ is equipped with a coherent family of cartesian lifts of both morphisms and cells, while in the case of $g$ the lifts of morphisms and cells are opcartesian.
	\end{definition}
	Section~3.3 of \cite{Lambert24} lists some examples of (split) $2$"/fibrations.
	
	\begin{definition} \label{3-category of split op-2-fibrations}
		Let $\map qJB$ and $\map rKD$ be split op"/$2$"/fibrations. A morphism $\map{(\phi, g)}qr$ consists of a $2$"/functor $\map\phi JK$ over a $2$"/functor $\map gBD$, that is, $r \of \phi = g \of q$, that preserve the cleavages as follows:
		\begin{displaymath}
			\rho_r(\phi i, gu) = \phi(\rho_q(i, u)) \qquad\qquad \text{and} \qquad\qquad \rho_r(g\beta, \phi s) = \phi(\rho_q(\beta, s)).
		\end{displaymath}
		A $2$"/natural transformation $\cell{(\zeta, \xi)}{(\phi, g)}{(\psi, k)}$ of morphisms of split op"/$2$"/fibrations consists of $2$"/natural transformations $\cell\zeta\phi\psi$ over $\cell\xi gk$, that is, $r \of \zeta = q \of \xi$. Similarly, modifications of such $2$"/natural transformations are pairs of modifications (\defref{modification}), one over the other.
		
		Let $\SpOpTwoFib$ denote the $3$"/category of split op"/$2$"/fibrations, their morphisms, $2$"/natural transformations and modifications. Analogously, there are $3$"/categories $\SpCoTwoFib$, $\SpTwoFib$ and $\SpCoopTwoFib$ of split co"/$2$"/fibrations, split $2$"/fibrations and split coop"/$2$"/fibrations.
	\end{definition}
	
	\subsubsection*{Split op"/$2$"/fibrations as internal split opfibrations}
	We shall now recall (the op"/dual variant of) Theorem~4.13 of \cite{Lambert24} in detail, which shows, in our terms, that internal split opfibrations $\map qJB$ (\defref{internal split (op)fibration}) in the sesquicategory $\und{\TwoCatLax}$ (\exref{lax natural transformation}) are precisely split op"/$2$"/fibrations in the sense of \defref{split op-2-fibration}. Recall from \defsref{internal split two-sided fibration}{internal profunctor} that $q$ is a $2$"/functor equipped with a right action $\map\rho{J \times_B B^\2}J$ that is associative and unital. Using \propref{comma objects from 2-cotensors}, we identify the $2$"/pullback $J \times_B B^\2$ with the lax comma $2$"/category $q \slash B \dfn q \slash \id_B$ (\exref{lax comma 2-category}). Thus regarding $\rho$ as a $2$"/functor $q \slash B \to J$, it sends
	\begin{enumerate}[label=-]
		\item objects $(i, qi \xrar u b, b)$ in $q \slash B$ to objects $\rho(i, u, b)$ in $J$;
		\item morphisms $\map{(i \xrar s j, \beta, b \xrar t c)}{(i, u, b)}{(j, v, c)}$ in $q \slash B$, with $\cell\beta{t \of u}{v \of qs}$, to morphisms $\map{\rho(s, \beta, t)}{\rho(i, u, b)}{\rho(j, v, c)}$ in $J$;
		\item cells $\cell{(s_1 \xRar\sigma s_2, t_1 \xRar\tau t_2)}{(s_1, \beta_1, t_1)}{(s_2, \beta_2, t_2)}$ in $q \slash B$, satisfying the equation described in \exref{lax comma 2-category}, to cells \mbox{$\cell{\rho(\sigma, \tau)}{\rho(s_1, \beta_1, t_1)}{\rho(s_2, \beta_2, t_2)}$} in $J$.
	\end{enumerate}
	
	In \thmref{opcleavages are right actions} below we will see that right actions $\map\rho{q \slash B}J$ can be identified with cleavages $\rho(\dash, \dash)$ making $q$ into a split op"/$2$"/fibration (\defref{split op-2-fibration}). Throughout the remainder of this article we will identify such right actions and cleavages, which is the reason for us denoting both these structures by $\rho$. If necessary they can be distinguished by the types of their inputs: for cleavages the inputs are of the form $(i, qi \xrar u b)$ and $(s, qs \xRar\beta u)$ (both are binary); different from the inputs of the right actions above.
	
	The following construction is described in the proof of (the op"/dual of) Lemma~4.10 of \cite{Lambert24}.
	\begin{construction}[Cleavage from a right action] \label{cleavage from a right action} 
		We claim that any right action \mbox{$\map\rho{q \slash B}J$} as above induces a well"/defined cleavage $\rho(\dash, \dash)$ making $\map qJB$ into a split op"/$2$"/fibration (\defref{split op-2-fibration}) as follows.
		\begin{enumerate}
			\item[\textup{(ml)}] Given $\map u{qi}b$ in $B$ set $u_!i \dfn \rho(i, u, b)$ and take the opcartesian lift to be
			\begin{displaymath}
				\rho(i, u) \dfn \brks{i = \rho(i, \id_{qi}, qi) \xrar{\rho(\id_i, \id_u, u)} \rho(i, u, b) = u_!i},
			\end{displaymath}
			where $(\id_i, \id_u, u)$ is the morphism $(i, \id_{qi}, qi) \to (i, u, b)$ in $q \slash B$. Here $\rho(i, \id_{qi}, qi) = i$ by the unit axiom for $\rho$ (\defref{internal profunctor}).
			\item[\textup{(om1)}] Given $\map sij$ in $J$ such that $qs = v \of u$, as in \defref{split op-2-fibration}(om1), take the lift $\tilde v$ of $\map vb{qj}$ to be
			\begin{displaymath}
				\tilde v \dfn \brks{u_! i = \rho(i, u, b) \xrar{\rho(s, \id_{qs}, v)} \rho(j, \id_{qj}, qj) = j}.
			\end{displaymath}
			\item[\textup{(om2)}] Given $\cell\sigma st$ in $J$ such that $q\sigma = \beta \of u$, as in \defref{split op-2-fibration}(om2), take the lift $\tilde \beta$ of $\cell\beta vw$ to be
			\begin{displaymath}
				\tilde\beta \dfn \brks{\tilde v = \rho(s, \id_{qs}, v) \xRar{\rho(\sigma, \beta)} \rho(t, \id_{qt}, w) = \tilde w}.
			\end{displaymath}
			\item[\textup{(cl)}] Given $\cell\beta u{qs}$ in $B$ as in \defref{split op-2-fibration}(cl) set
			\begin{displaymath}
				\beta^*s \dfn \brks{i = \rho(i, \id_{qi}, qi) \xrar{\rho(s, \beta, u)} \rho(j, \id_{qj}, qj) = j}
			\end{displaymath}
			and take the cartesian lift to be
			\begin{displaymath}
				\rho(\beta, s) \dfn \brks{\beta^*s = \rho(s, \beta, u) \xRar{\rho(\id_s, \beta)} \rho(s, \id_{qs}, qs) = s},
			\end{displaymath}
			where $\rho(s, \id_{qs}, qs) = s$ by the unit axiom for $\rho$.
			\item[\textup{(cc)}] Given $\cell\sigma rs$ in $J$ such that $q\sigma = \alpha \hc \beta$, as in \defref{split op-2-fibration}(cc), take the lift $\tilde\alpha$ of $\cell\alpha {qr}u$ to be
			\begin{displaymath}
				\tilde\alpha \dfn \brks{r = \rho(r, \id_{qr}, qr) \xRar{\rho(\sigma , \alpha)} \rho(s, \beta, u) = \beta^* s}.
			\end{displaymath}
		\end{enumerate}
	\end{construction}
	
	The following is (the op"/dual of) Construction~4.8 of \cite{Lambert24}.
	\begin{construction}[Right action from a cleavage] \label{right action from a cleavage}
		Let $\map qJB$ be a $2$"/functor. We claim that any cleavage $\rho(\dash, \dash)$ making $q$ into a split op"/$2$"/fibration (\defref{split op-2-fibration}) induces a well"/defined right action $\map\rho{q \slash B}J$ (\defref{internal profunctor}) as follows.
		\begin{enumerate}
			\item[(o)] Send the object $(i, qi \xrar u b, b)$ of $q \slash B$ to $\rho(i, u, b) \dfn u_!i$ in $J$.
			\item[(m)] Send the morphism $\map{(i \xrar s j, \beta, b \xrar t c)}{(i, u, b)}{(j, v, c)}$ of $q \slash B$, where \mbox{$\cell\beta{t \of u}{v \of qs}$}, to the unique lift $\map{\rho(s, \beta, t) \dfn \tilde t}{u_!i}{v_!j}$ of $t$ in the factorisation $q\bigpars{\beta^*(\rho(j, v) \of s)} = t \of u$ (\defref{split op-2-fibration}(cl) and (om1)):
				\begin{displaymath}
					\begin{tikzpicture}
						\matrix(m)[math35, column sep={8.5em,between origins}, row sep={4em,between origins}, xshift=-4.5cm]{i & u_!i \\ j & v_!j \\};
							\path[map]	(m-1-1) edge node[above] {$\rho(i, u)$} (m-1-2)
																	edge node[left] {$s$} (m-2-1)
													(m-1-2) edge[dashed] node[right] {$\tilde t \nfd \rho(s, \beta, t)$} (m-2-2)
													(m-2-1) edge node[below] {$\rho(j, v)$} (m-2-2);
							\path[transform canvas={xshift=2pt}, map]	(m-1-1) edge node[above, xshift=16pt] {$\beta^*(\rho(j,v ) s)$} (m-2-2);
							\path				($(m-1-1)!0.33!(m-2-2)$) edge[cell, shorten >=9pt, shorten <=11pt] node[yshift=-10pt, xshift=20pt] {$\rho(\beta, \rho(j, v) s)$} (m-2-1);
						
						\matrix(m)[math35, column sep={4.5em,between origins}, row sep={4em,between origins}, xshift=2.5cm]{qi & b \\ qj & c \\};
							\path[map]	(m-1-1) edge node[above] {$u$} (m-1-2)
																	edge node[left] {$qs$} (m-2-1)
																	edge node[above, xshift=6pt] {$t \of u$} (m-2-2)
													(m-1-2) edge node[right] {$t$} (m-2-2)
													(m-2-1) edge node[below] {$v$} (m-2-2);
							\path				($(m-1-1)!0.5!(m-2-2)$) edge[cell, shorten >=4pt, shorten <=6pt] node[above left, inner sep=2pt] {$\beta$} (m-2-1);
						\draw	(-0.5cm, 0) edge[mapsto] node[above] {$q$} (0.5cm, 0);
					\end{tikzpicture}
				\end{displaymath}
			\item[(c)] Send the cell $\cell{(s_1 \xRar\sigma s_2, t_1 \xRar\tau t_2)}{(s_1, \beta_1, t_1)}{(s_2, \beta_2, t_2)}$ of $q \slash B$, making the square below right commute (see \exref{lax comma 2-category}), to the unique lift \mbox{$\cell{\rho(\sigma, \tau) \dfn \tilde\tau}{\tilde t_1}{\tilde t_2}$} of $\tau$ in the factorisation $q(\widetilde{\tau \of u}) = \tau \of u$ (\defref{split op-2-fibration}(om2)). Here $\widetilde{\tau \of u}$	is the unique lift of $\tau \of u$ in the factorisation $q\bigpars{\rho(\beta_1, \rho(j, v) \of s_1) \hc (\rho(j, v) \of \sigma)} = (\tau \of u) \hc \beta_2$ (\defref{split op-2-fibration}(cc)):
				\begin{displaymath}
					\begin{tikzpicture}
						\matrix(m)[math35, column sep={10em,between origins}, xshift=-3.6cm]{\tilde t_1 \of \rho(i, u) & \rho(j, v) \of s_1 \\ \tilde t_2 \of \rho(i, u) & \rho(j, v) \of s_2 \\};
						\path[map]	(m-1-1) edge[cellmap] node[above, xshift=2pt] {$\rho(\beta_1, \rho(j, v)s_1)$} (m-1-2)
																edge[cellmap, dashed] node[left] {$\widetilde{\tau \of u}$} (m-2-1)
												(m-1-2) edge[cellmap] node[right] {$\rho(j, v) \of \sigma$} (m-2-2)
												(m-2-1) edge[cellmap] node[below, xshift=-2pt] {$\rho(\beta_2, \rho(j, v)s_2)$} (m-2-2);
												
						\matrix(m)[math35, column sep={5em,between origins}, xshift=3.5cm]{t_1 \of u & v \of qs_1 \\ t_2 \of u & v \of qs_2 \\};
						\path[map]	(m-1-1) edge[cellmap] node[above] {$\beta_1$} (m-1-2)
																edge[cellmap] node[left] {$\tau \of u$} (m-2-1)
												(m-1-2) edge[cellmap] node[right] {$v \of q\sigma$} (m-2-2)
												(m-2-1) edge[cellmap] node[below] {$\beta_2$} (m-2-2);
						
						\draw (0.3cm, 0) edge[mapsto] node[above] {$q$} (1.1cm, 0);
					\end{tikzpicture}
				\end{displaymath}
		\end{enumerate}
	\end{construction}
	
	As the main result of this section we partly recall (the op"/dual of) Theorem~4.13 of \cite{Lambert24}.
	\begin{theorem}[Lambert] \label{opcleavages are right actions}
		Let $\map qJB$ be a $2$"/functor. The cleavage $\rho(\dash, \dash)$ for $q$ constructed in \conref{cleavage from a right action}  and the right action $\map\rho{q \slash B}J$ on $q$ constructed in \conref{right action from a cleavage} are both well"/defined. These constructions combine to form a bijective correspondence between cleavages for $q$ (making $q$ into a split op"/$2$"/fibration; see \defref{split op-2-fibration}) and right actions on $q$ (making $q$ into an internal split opfibration in $\und{\TwoCatLax}$; see \defref{internal split (op)fibration} and \exref{lax natural transformation}).
	\end{theorem}
	
	Recall from \defref{split co-2-fibration} the notion of split co"/$2$"/fibration.
	\begin{corollary}[Lambert] \label{cleavages are left actions}
		Let $\map pJA$ be a $2$"/functor. There is a bijective correspondence between cleavages $\lambda(\dash, \dash)$ for $p$ (making $p$ into a split co"/$2$"/fibration) and left actions $\map\lambda{A \slash p}J$ on $p$ (making $p$ into an internal split fibration in $\und{\TwoCatLax}$; see \defref{internal split (op)fibration}).
	\end{corollary}
	
	\section[Split two-sided 2-fibrations]{Split two"/sided $2$"/fibrations} \label{split two-sided 2-fibrations section}
	Here we characterise internal split two"/sided fibrations (\defref{internal split two-sided fibration}) in the sesquicategory $\und{\TwoCatLax}$ of lax natural transformations (\exref{lax natural transformation}), as well as the unary multicells in the unital virtual equipment $\inSpTwoFib{\und{\TwoCatLax}}$ (\defsref{internal split two-sided fibration}{internal profunctor}) between them. In doing so we use the characterisations of internal split (op)fibrations in $\und{\TwoCatLax}$ as split op"/$2$"/fibrations and co"/$2$"/fibrations, given in \thmref{opcleavages are right actions} and \cororef{cleavages are left actions} above.
	
	\begin{theorem} \label{split two-sided 2-fibration description}
		A span $A \xlar p J \xrar q B$ of $2$"/functors forms an internal split two"/sided fibration in $\und{\TwoCatLax}$ if and only if the following hold:
		\begin{enumerate}
			\item[\textup{(c)}] $p$ is a split co"/$2$"/fibration (\defref{split co-2-fibration}), with cleavage say $\lambda(\dash, \dash)$;
			\item[\textup{(o)}] $q$ is a split op"/$2$"/fibration (\defref{split op-2-fibration}), with cleavage say $\rho(\dash, \dash)$;
			\item[\textup{(v)}] the cleavages are ``vertical'' with respect to the other leg:
			\begin{align*}
				&q(\lambda(u, i)) = \id_{qi} &\text{and}& \qquad\qquad\quad q(\lambda(s, \alpha)) = \id_{qs} \\
				&p(\rho(j, v)) = \id_{pj} &\text{and}& \qquad\qquad\quad p(\rho(\beta, t)) = \id_{pt};
			\end{align*}
			\item[\textup{(cml)}] for $\map ua{pi}$ in $A$ and $\map v{qi}b$ in $B$ taking cartesian and opcartesian lifts ``commutes'' in $J$: $v_!(u^* i) = u^*(v_! i)$ and the diagram below left commutes;
				\begin{displaymath}
					\begin{tikzpicture}[baseline]
						\matrix(m)[math35, column sep={3.5em,between origins}]
							{& i & \\ u^* i & & v_!i \\ & v_!(u^* i) = u^*(v_! i) & \\};
						\path[map]	(m-1-2) edge node[above right] {$\rho(i, v)$} (m-2-3)
												(m-2-1) edge node[above left] {$\lambda(u, i)$} (m-1-2)
																edge node[left] {$\rho(u^*i, v)$} ($([xshift=-22pt, yshift=8pt]m-3-2)$)
												($([xshift=22pt, yshift=8pt]m-3-2)$) edge node[right] {$\lambda(u, v_!i)$} (m-2-3); 	
					\end{tikzpicture} \qquad\qquad \begin{tikzpicture}[baseline]
						\matrix(m)[math35, column sep={3.5em,between origins}]
							{& s & \\ \beta^* s & & \alpha_!s \\ & \alpha_!(\beta^* s) = \beta^*(\alpha_! s) & \\};
						\path				(m-1-2) edge[cellmap] node[above right] {$\lambda(s, \alpha)$} (m-2-3)
												(m-2-1) edge[cellmap] node[above left] {$\rho(\beta, s)$} (m-1-2)
																edge[cellmap] node[left] {$\lambda(\beta^*s, \alpha)$} ($([xshift=-22pt, yshift=8pt]m-3-2)$)
												($([xshift=22pt, yshift=8pt]m-3-2)$) edge[cellmap] node[right] {$\rho(\beta, \alpha_!s)$} (m-2-3); 	
					\end{tikzpicture}
				\end{displaymath}
			\item[\textup{(ccl)}] given $\map sij$ in $J$, for $\cell\alpha{ps}u$ in $A$ and $\cell\beta v{qs}$ in $B$ taking cartesian and opcartesian lifts ``commutes'' in $J$: $\alpha_!(\beta^* s) = \beta^*(\alpha_! s)$ and the diagram above right commutes.
		\end{enumerate}
	\end{theorem}
	
	\begin{remark} \label{relation to two-sided 1-fibration}
		Following \remref{relation to 1-fibration}, notice that the underlying span of $1$"/functors \mbox{$\und A \xlar{\und p} \und J \xrar{\und q} \und B$} forms a split two"/sided $1$"/fibration in the usual sense, that is, $\und p$ is a split $1$"/fibration and $\und q$ is a split $1$"/opfibration whose $1$"/dimensional cleavages satisfy the equations on the left of (v) above as well as condition (cml) (see e.g.\ Definition 2.3.4 of \cite{Loregian-Riehl20} for the notion of (not necessarily split) two"/sided $1$"/fibration). Likewise, the cartesian cells $\lambda(\alpha, s)$ and opcartesian cells $\lambda(t, \beta)$ above, satisfying the equations on the right of (v) as well as condition (ccl), make $A \xlar p J \xrar q B$ ``locally'' a ``reversed'' split two"/sided $1$"/fibration, that is, the reversed spans of hom"/functors $B(qi, qj) \leftarrow J(i, j) \rightarrow A(pi, pj)$ are split two"/sided $1$"/fibrations, for each pair of objects $i$, $j \in J$.
	\end{remark}
	
	\begin{remark} \label{factorisation of morphisms in J}
		Consider a split two"/sided $1$"/fibration $A \xlar p J \xrar q B$. The following factorisation property for morphisms of $J$ is well"/known, and we will use it in \thmref{Yoneda axiom} below. Notice that by the previous remark it also applies to spans \mbox{$A \brar B$} of $2$"/functors that satisfy the conditions of the theorem above.
		
		Every morphism $\map sij$ in $J$ factors uniquely as $s = \lambda(ps, j) \of \hat s \of \rho(i, qs)$, with $\map{\hat s}{(qs)_!i}{(ps)^*j}$ contained in the $p$"/fibre of $i$ and in the $q$"/fibre of $j$, that is, $p\hat s = \id_{pi}$ and $q\hat s = \id_{qj}$. To obtain $\hat s$, first consider the unique lift $\map{\widetilde{\id_{pi}}}i{(ps)^*j}$ of $\id_{pi}$ in $ps = ps \of \id_{pi}$. It follows from condition (v) above that $q\widetilde{\id_{pi}} = qs$, and we take $\hat s \dfn \widetilde{\id_{qj}}$ to be the unique lift of $\id_{qj}$ in $qs = \id_{qj} \of qs$. 
		\begin{displaymath}
			\begin{tikzpicture}
					\matrix(m)[math35, xshift=-4.5cm]{pi & pi \\ pj & pi \\};
					\path[map]	(m-1-1) edge node[left] {$ps$} (m-2-1)
											(m-2-2) edge node[below] {$ps$} (m-2-1);
					\path				(m-1-1) edge[eq] (m-1-2)
															edge[eq] (m-2-2)
											(m-1-2) edge[eq] (m-2-2);
										
					\matrix(m)[math35, column sep={5.5em,between origins}]{i & (qs)_!i \\ j & (ps)^*j \\};
					\path[map]	(m-1-1) edge node[above] {$\rho(i, qs)$} (m-1-2)
															edge[dashed] node[below left, inner sep=1pt] {$\widetilde{\id_{pi}}$} (m-2-2)
															edge node[left] {$s$} (m-2-1)
											(m-1-2) edge[dashed] node[right] {$\hat s \dfn \widetilde{\id_{qj}}$} (m-2-2)
											(m-2-2) edge node[below] {$\lambda(ps, j)$} (m-2-1);
					
					\matrix(m)[math35, xshift=5.5cm]{qi & qj \\ qj & qj \\};
					\path[map]	(m-1-1) edge node[above] {$qs$} (m-1-2)
															edge node[below left, inner sep=1pt] {$qs$} (m-2-2)
															edge node[left] {$qs$} (m-2-1);
					\path				(m-1-2) edge[eq] (m-2-2)
											(m-2-1) edge[eq] (m-2-2);
					
					\path				(-2.1cm, 0) edge[mapsto] node[above] {$p$} (-3.1cm,0)
											(2.8cm, 0) edge[mapsto] node[above] {$q$} (3.8cm,0);
			\end{tikzpicture}
		\end{displaymath}
	\end{remark}
	
	\begin{proof}[Proof of \thmref{split two-sided 2-fibration description}.]
		By \defsref{internal split two-sided fibration}{internal profunctor} the span $\hmap{(p, q)}AB$ is an internal split two"/sided fibration in $\und{\TwoCatLax}$ (\exsref{lax natural transformation}{lax comma 2-category}) whenever it comes equipped with a left action $\map\lambda{A^\2 \times_A J}J$ and right action $\map\rho{J \times_B B^\2}J$ that: (1) are both associative and unital, (2) are span morphisms, and (3) satisfy a compatibility axiom. By \cororef{cleavages are left actions} any associative and unital left action $\lambda$ corresponds bijectively to a cleavage $\lambda(\dash, \dash)$ making $p$ into a split co"/$2$"/fibration, that is, (c) above holds, while by \thmref{opcleavages are right actions} a right such action $\rho$ corresponds to a cleavage $\rho(\dash, \dash)$ making $q$ into a split op"/$2$"/fibration, that is, (o) holds. Assuming that (c) and (o) hold, it is straightforward to show that $\rho$ and $\lambda$ are span morphisms if and only if condition (v) above holds.
		
		Assume that (c), (o) and (v) hold. We claim that $\rho$ and $\lambda$ are compatible, that is, the diagram below left commutes, precisely if both (cml) and (ccl) above hold, thus completing the proof. Using \conref{cleavage from a right action}(ml) and its coop"/dual, clearly, the diagram commutes at objects \mbox{$(\map ua{pi}, i, \map v{qi}b)$} precisely if $v_!(u^* i) = u^*(v_! i)$, the first condition of (cml). Assuming the latter, below we will show that the diagram commutes at morphisms of the form below right, that is, triples of morphisms $(r, \zeta, ps) \in A^\2$ (see \exref{lax comma 2-category}), $s \in J$ and $(qs, \xi, t) \in B^\2$, precisely if the second condition of (cml) and the first condition of (ccl) hold. Finally, the author leaves it to the reader to show that, under the assumption of all conditions above except for the second condition of (ccl), the latter condition holds precisely if the diagram below left commutes at the cells of $A^\2 \times_A J \times_B B^\2$ (see \exref{lax comma 2-category}).
		\begin{displaymath}
			\begin{tikzpicture}[textbaseline]
				\matrix(m)[math35, column sep={9.5em,between origins}]
					{A^\2 \times_A J \times_B B^\2 & A^\2 \times_A J \\ J \times_B B^\2 & J \\};
				\path[map]	(m-1-1) edge node[above, xshift=2pt] {$\id_{A^\2} \times_A \rho$} (m-1-2)
														edge node[left] {$\lambda \times_B \id_{B^\2}$} (m-2-1)
										(m-1-2) edge node[right] {$\lambda$} (m-2-2)
										(m-2-1) edge node[below] {$\rho$} (m-2-2);
			\end{tikzpicture} \qquad \Bigl(\begin{tikzpicture}[textbaseline]
				\matrix(m)[math35]{a & pi \\ a' & pi' \\};
				\path[map]	(m-1-1) edge node[above] {$u$} (m-1-2)
														edge node[left] {$r$} (m-2-1)
										(m-1-2) edge node[right] {$ps$} (m-2-2)
										(m-2-1) edge node[below] {$u$} (m-2-2);
				\path				(m-1-2) edge[cell, shorten >= 8pt, shorten <= 8pt] node[below right] {$\zeta$} (m-2-1);
			\end{tikzpicture}, \begin{tikzpicture}[textbaseline]
				\matrix(m)[math35]{i \\ i' \\};
				\path[map]	(m-1-1) edge node[right] {$s$} (m-2-1);
			\end{tikzpicture}, \begin{tikzpicture}[textbaseline]
				\matrix(m)[math35]{qi & b \\ qi' & b' \\};
				\path[map]	(m-1-1) edge node[above] {$v$} (m-1-2)
														edge node[left] {$qs$} (m-2-1)
										(m-1-2) edge node[right] {$t$} (m-2-2)
										(m-2-1) edge node[below] {$v'$} (m-2-2);
				\path				(m-1-2) edge[cell, shorten >= 8pt, shorten <= 10.5pt] node[below right] {$\xi$} (m-2-1);
			\end{tikzpicture}\Bigr)
		\end{displaymath}
		
		Assume (c), (o), (v) as well as the commuting of the diagram above left on objects and morphisms so that, in particular, the first condition of (cml) holds. That the first condition of (ccl) holds too follows immediately from the commutativity of the diagram at the morphism $\bigpars{(u, \alpha, ps), s, (qs, \beta, v)}$ and (the coop"/dual of) \conref{cleavage from a right action}(cml). To prove the second condition of (cml) consider the commuting of the diagram at the morphism
		\begin{displaymath}
			\bigpars{(u, \id_u, \id_{pi}), \id_i, (\id_{qi}, \id_v, \id_b)}.
		\end{displaymath}
		We will use \conref{cleavage from a right action}(ml) and its coop"/dual repeatedly. Since $(\id_{qi}, \id_v, \id_b) \in J \times_B B^\2$ corresponds to $\id_{(i, v, b)} \in q \slash B$ (\propref{comma objects from 2-cotensors}), which is mapped to $\id_{v_! i}$ by $\rho$, we find that applying the top leg \mbox{$\lambda \of (\id_{A^\2} \times_A \rho)$} results in $\lambda(u, \id_u, \id_{v_! i}) = \lambda(u, v_!i)$. Applying the other leg $\rho \of (\lambda \times_B \id_{B^\2})$ instead, we obtain $\rho(\lambda(u,i), \id_v, \id_b)$, where $q(\lambda(u, i)) = \id_{qi}$ by condition (v). Using \conref{right action from a cleavage}(m), we conclude that $\lambda(u, v_!i) = \rho(\lambda(u, i), \id_v, \id_b) = \widetilde{\id_b}$ with $\widetilde{\id_b}$ the unique lift of $\id_b$ such that $\widetilde{\id_b} \of \rho(u^*i, v) = (\id_v)^*\bigpars{\rho(i, v) \of \lambda(u, i)} = \rho(i, v) \of \lambda(u, i)$ (\defref{split op-2-fibration}(sc1)), which proves the second condition of (cml).
		 
		 Conversely, assume that all conditions (c)--(ccl) of the statement hold except for the second condition of (ccl) so that, in particular, the diagram above left commutes at objects. We will show that it commutes at morphisms $M \dfn \bigpars{(r, \zeta, ps), s, (qs, \xi, t)}$ above right too. By \conref{right action from a cleavage}(m) and its coop-dual the images of $M$ under the $2$"/functors of the diagram involve unique lifts $\tilde r$, $\hat r$, $\tilde t$ and $\hat t$ of $r$ and $t$ as follows (here too we apply \propref{comma objects from 2-cotensors} to the iterated pullbacks of the diagram):
		 \begin{enumerate}[label=-]
		 	\item $(\id_{A^\2} \times_A \rho)(M) = (r, \zeta, \tilde t) \in A \slash p$ such that $\tilde t \of \rho(i, v) = \xi^*\bigpars{\rho(i', v') \of s}$;
		 	\item $\lambda(r, \zeta, \tilde t) = \tilde r$ such that $\lambda(u', v'_!i') \of \tilde r = \zeta_! \bigpars{\tilde t \of \lambda(u, v_!i)}$;
		 	\item $(\lambda \times_B \id_{B^\2})(M) = (\hat r, \xi, t) \in q \slash B$ such that $\lambda(u', i') \of \hat r = \zeta_! \bigpars{s \of \lambda(u, i)}$;
		 	\item $\rho(\hat r, \xi, t) = \hat t$ such that $\hat t \of \rho(u^*i, v) = \xi^*\bigpars{\rho(u'^*i', v') \of \hat r}$.
		 \end{enumerate}
		 Showing that the diagram commutes at $M$ means showing that $\tilde r = \hat t$. Since factorisations through the (op)cartesian morphisms $\lambda(u', v'_!i')$ and $\rho(u^*i, v)$ below are unique, the required equality follows from the string of equalities below.
		 \begin{align*}
		 		\lambda(u', v'_!i'& ) \of \tilde r \of \rho(u^*i, v) \overset{\tilde r, \textup s}= \zeta_!\bigpars{\tilde t \of \lambda(u, v_!i)} \of \id_{\id_a!}\bigpars{\rho(u^*i, v)} \\
		 		&\overset{\textup s}= \zeta_!\bigpars{\tilde t \of \lambda(u, v_!i) \of \rho(u^*i, v)} \overset{\textup{cml2}}= \zeta_!\bigpars{\tilde t \of \rho(i, v) \of \lambda(u, i)} \\
		 		&\overset{\tilde t, \textup s}= \zeta_!\bigpars{\xi^*(\rho(i', v') \of s) \of \id_{\id_{qi}}^*(\lambda(u, i))} \overset{\textup s}= \zeta_!\xi^*\bigpars{\rho(i', v') \of s \of \lambda(u, i)} \\
		 		&\overset{\textup{ccl1}}= \xi^*\zeta_!\bigpars{\rho(i', v') \of s \of \lambda(u, i)} \overset{\textup s, \hat r}= \xi^*\bigpars{\rho(i', v') \of \lambda(u', i') \of \hat r} \\
		 		&\overset{\textup{cml2}}= \xi^*\bigpars{\lambda(u', v'_!i') \of \rho(u'^*i', v') \of \hat r} \overset{\textup s, \hat t}= \lambda(u', v'_!i') \of \hat t \of \rho(u^*i, v)
		 \end{align*}
		 In deriving the equalities labelled `s' above the splitting equations (\defref{split op-2-fibration}) are used; those labelled $\tilde r$, $\tilde t$, $\hat r$ or $\hat t$ use the defining equation of $\tilde r$, $\tilde t$, $\hat r$ or $\hat t$, respectively, as given above; those labelled cml2 or ccl1 use the second condition of (cml) or the first condition of (ccl), respectively.
	\end{proof}
	
	The previous theorem gives rise to the following definition.
	\begin{definition} \label{split two-sided 2-fibration}
		A \emph{split two"/sided $2$"/fibration} $\hmap JAB$ is an internal split two"/sided fibration in the sesquicategory $\und{\TwoCatLax}$ of lax natural transformations (\defref{internal split two-sided fibration} and \exref{lax natural transformation}), that is, a span $A \xlar p J \xrar q B$ of $2$"/functors satisfying all conditions \textup{(c)}--\textup{(ccl)} of the previous theorem.
		
		We write $\SpTwoTwoFib \dfn \inSpTwoFib{\und{\TwoCatLax}}$ for the unital virtual equipment of split two"/sided $2$"/fibrations between (possibly large) $2$"/categories (\defref{internal split two-sided fibration} and \cororef{internal split two-sided fibrations form a unital virtual equipment}).
	\end{definition}
	
	\begin{example} \label{lax comma 2-category is a split two-sided 2-fibration}
		Consider the lax comma $2$"/category $f \slash g$ of a cospan of $2$"/functors $A \xrar f C \xlar g B$ (\exref{lax comma 2-category}). The span $C(f, g) \dfn \brks{A \xlar{\pi_A} f \slash g \xrar{\pi_B} B}$ of its projections is a split two"/sided $2$"/fibration. \exref{lax comma 2-category is split op-2-fibered} describes the cleavage for $\pi_B$; the cleavage for $\pi_A$ is defined coop-dually.	In particular, the span $I_A \dfn \brks{A \xlar{\src} A^\2 \xrar{\tgt} A}$ associated to a lax arrow $2$"/category \mbox{$A^\2 \dfn \id_A \slash \id_A$} (\defref{lax comma 2-category}) is a split two"/sided $2$"/fibration.
		
		Applying parts (b) and (c) of \cororef{internal split two-sided fibrations form a unital virtual equipment} to $\SpTwoTwoFib \dfn \inSpTwoFib{\und{\TwoCatLax}}$, we see that $C(f, g)$ forms the nullary restriction of $C$ along $f$ and $g$ in the unital virtual equipment $\SpTwoTwoFib$ and that $I_A$ forms the horizontal unit of $A$, both in the sense of \defref{cartesian cells}. In \exref{lax comma 2-categories are locally discrete} below we will see that $C(f, g)$ and $I_A$ are in fact `locally discrete' split two-sided $2$"/fibrations.
	\end{example}
	
	\subsubsection*{Multimorphisms of split two"/sided $2$"/fibrations}
	The remainder of this section is devoted to the description of the multimorphisms of split two"/sided $2$"/fibrations that form the unary cells of $\SpTwoTwoFib$. The description of nullary cells we defer to \cororef{nullary cells description}. Remember that $\SpTwoTwoFib$ embeds locally"/fully into the unital virtual equipment $\inProf{\und{\TwoCatLax}}$ of internal profunctors in $\und{\TwoCatLax}$ (\defsref{internal split two-sided fibration}{internal profunctor}).
	
	Let $\ul J = (A_0 \xlar{p_1} J_1 \xrar{q_1} \dotsb \xlar{p_n} J_n \xrar{q_n} A_n)$ be an iterated span of split two"/sided $2$"/fibrations $J_m = (A_{m'} \xlar{p_m} J_m \xrar{q_m} A_m)$ of length $n \geq 1$. Like in \propref{profunctors form an augmented virtual double category} we write $\prod_{\dash} \ul J \dfn (A_0 \leftarrow J_1 \times_{A_1} \dotsb \times_{A_{n'}} J_n \rightarrow A_n)$ for the span induced by its iterated pullback.
	\begin{theorem} \label{unary cells of split two-sided 2-fibrations}
		Besides the iterated span $\ul J$ above, consider a split two"/sided $2$"/fibration $C \xlar p K \xrar q D$ as well as $2$"/functors $\map f{A_0}C$ and $\map g{A_n}D$, that together make up the boundary of the unary cells below.
		
		Unary cells
		\begin{displaymath}
			\begin{tikzpicture}
				\matrix(m)[math35]{A_0 & A_1 & A_{n'} & A_n \\ C & & & D \\};
				\path[map]	(m-1-1) edge[barred] node[above] {$J_1$} (m-1-2)
														edge node[left] {$f$} (m-2-1)
										(m-1-3) edge[barred] node[above] {$J_n$} (m-1-4)
										(m-1-4) edge node[right] {$g$} (m-2-4)
										(m-2-1) edge[barred] node[below] {$K$} (m-2-4);
				\path[transform canvas={xshift=1.75em}]	(m-1-2) edge[cell] node[right] {$\phi$} (m-2-2);
				\draw				($(m-1-2)!0.5!(m-1-3)$) node {$\dotsb$};
			\end{tikzpicture}
		\end{displaymath}
		in $\SpTwoTwoFib$ correspond precisely to $2$"/functors $\map\phi{J_1 \times_{A_1} \dotsb \times_{A_{n'}} J_n}K$ over $f$ and $g$ that satisfy the following conditions:
		\begin{enumerate}
			\item[\textup{(pcm)}] $\phi\bigpars{\lambda_{p_1}(u, i_1), \id_{i_2}, \dotsc, \id_{i_n}} = \lambda_p\bigpars{fu, \phi(i_1, \dotsc, i_n)}$ for any object $(i_1, \dotsc, i_n)$ of $\prod_{\dash} \ul J$ and $\map ua{p_1i_1}$ in $A_0$;
			\item[\textup{(poc)}] $\phi\bigpars{\lambda_{p_1}(s_1, \alpha), \id_{s_2}, \dotsc, \id_{s_n}} = \lambda_p\bigpars{\phi(s_1, \dotsc, s_n), f\alpha}$ for any morphism $(s_1, \dotsc, s_n)$ of $\prod_{\dash} \ul J$ and $\cell \alpha{p_1s_1}u$ in $A_0$;
			\item[\textup{(iem)}] for any $1 \leq m < n$, objects $(i_1, \dotsc, i_m)$ in $\prod_{\dash} (J_1, \dotsc, J_m)$ and $(i_{m+1}, \dotsc, i_n)$ in $\prod_{\dash} (J_{m+1}, \dotsc, J_n)$, as well as $\map u{q_mi_m}{p_{m+1}i_{m+1}}$ in $A_m$, the morphism
			\begin{displaymath}
				\phi\bigpars{\id_{i_1}, \dotsc, \id_{i_{m'}}, \rho_{q_m}(i_m, u), \lambda_{p_{m+1}}(u, i_{m+1}), \id_{i_{m+2}}, \dotsc, \id_{i_n}}
			\end{displaymath}
			is an identity morphism;
			\item[\textup{(iec)}] for any $1 \leq m < n$, morphisms $(s_1, \dotsc, s_m)$ in $\prod_{\dash} (J_1, \dotsc, J_m)$ and $(s_{m+1}, \dotsc, s_n)$ in $\prod_{\dash} (J_{m+1}, \dotsc, J_n)$, as well as $\cell \alpha{p_{m+1}s_{m+1}}{q_ms_m}$ in $A_m$, the cell
			\begin{displaymath}
				\phi\bigpars{\id_{s_1}, \dotsc, \id_{s_{m'}}, \rho_{q_m}(\alpha, s_m), \lambda_{p_{m+1}}(s_{m+1}, \alpha), \id_{s_{m+2}}, \dotsc, \id_{s_n}}
			\end{displaymath}
			is an identity cell;
			\item[\textup{(pom)}] $\phi\bigpars{\id_{i_1}, \dotsc, \id_{i_{n'}}, \rho_{q_n}(i_n, v)} = \rho_q\bigpars{\phi(i_1, \dotsc, i_n), gv}$ for any object $(i_1, \dotsc, i_n)$ of $\prod_{\dash} \ul J$ and $\map v{q_ni_n}b$ in $A_n$;
			\item[\textup{(pcc)}] $\phi\bigpars{\id_{s_1}, \dotsc, \id_{s_{n'}}, \rho_{q_n}(\beta, s_n)} = \rho_q\bigpars{g\beta, \phi(s_1, \dotsc, s_n)}$ for any morphism $(s_1, \dotsc, s_n)$ of $\prod_{\dash} \ul J$ and $\cell \beta u{q_ns_n}$ in $A_n$.
		\end{enumerate}
	\end{theorem}
	
	Comparing the sources of the morphisms on either side of the equality required by condition (pcm) above, we find that $\phi(u^*i_1, i_2, \dotsc, i_n) = (fu)^*\phi(i_1, \dotsc, i_n)$. Likewise, condition (iec) above implies the equality of morphisms
	\begin{displaymath}
		\phi(s_1, \dotsc, s_{m'}, \alpha^* s_m, s_{m+1}, s_{m+2}, \dotsc, s_n) \! = \! \phi(s_1, \dotsc, s_{m'}, s_m, \alpha_! s_{m+1}, s_{m+2}, \dotsc, s_n).
	\end{displaymath}
	Similarly, each of the other conditions above implies an equality of pairs of objects or morphisms, respectively.
	
	Taking $n = 1$ and $A_0 = 1 = C$ the terminal $2$"/category in the theorem above, so that $J_1$ and $K$ are split op"/$2$"/fibrations by \defref{internal split (op)fibration} and \thmref{opcleavages are right actions}, notice that in this case the conditions on the $2$"/functor $\map\phi{J_1}K$ are precisely those that make $\phi$ into a morphism of split $2$"/opfibrations, in the sense of \defref{3-category of split op-2-fibrations}. Thus we obtain the following corollary.
	\begin{corollary} \label{split 2-opfibrations as split two-sided 2-fibrations}
		The horizontal slice category $1 \hs \SpTwoTwoFib$ (\defref{internal split (op)fibration}) of split two"/sided $2$"/fibrations $1 \brar B$ is isomorphic to the $1$"/category $\und{\und{\SpOpTwoFib}}$ of split op"/$2$"/fibrations (\defref{3-category of split op-2-fibrations}). Likewise, $\SpTwoTwoFib \hs 1 \iso \und{\und{\SpCoTwoFib}}$ is the $1$"/category of split co"/$2$"/fibrations.
	\end{corollary}
	
	\begin{proof}[Proof of \thmref{unary cells of split two-sided 2-fibrations}.]
		By \defthreeref{split two-sided 2-fibration}{internal split two-sided fibration}{internal profunctor} the $2$"/functor $\phi$ forms a unary cell in $\SpTwoTwoFib$ whenever it satisfies the left and right external equivariance axioms, which ensure coherence of $\phi$ with the left"/action of $A_0^\2$ and the right action of $A_n^\2$, as well as the internal equivariance axioms, which ensure coherence with the actions of \mbox{$A_1^\2$, \dots, $A_{n'}^\2$}. We claim that the left external equivariance axiom is equivalent to conditions (pcm) and (poc) above, that right external equivariance is equivalent to (pom) and (pcc) and that internal equivariance is equivalent to (iem) and (iec). We will partly prove the latter two equivalences, leaving the remainder to the reader.
		\begin{displaymath}
			\hspace{-35pt}\begin{tikzpicture}[textbaseline]
				\matrix(m)[math35, column sep={9.5em,between origins}]
					{	J_1 \times_{A_1} \dotsb \times_{A_{n'}} J_n \times_{A_n} A_n^\2 & K \times_D D^\2 \\
						J_1 \times_{A_1} \dotsb \times_{A_{n'}} J_n & K \\};
				\path[map]	(m-1-1) edge node[above, inner sep=6pt] {$\phi \times_g g^\2$} (m-1-2)
										(m-1-1)	edge[transform canvas={xshift=35pt}] node[left, inner sep=0.5pt] {$\id_{J_1} \times_{A_1} \dotsb \times_{A_{n''}} \id_{J_{n'}}\times_{A_{n'}} \rho_{q_n}$} (m-2-1)
										(m-1-2) edge node[right] {$\rho_q$} (m-2-2)
										(m-2-1) edge node[below] {$\phi$} (m-2-2);
			\end{tikzpicture} \hspace{0.25cm} \Bigl(\!\begin{tikzpicture}[textbaseline]
				\matrix(m)[math35]{i_1 \\ i'_1 \\};
				\path[map]	(m-1-1) edge node[left] {$s_1$} (m-2-1);
			\end{tikzpicture}\hspace{-0.25cm}, \dotsc, \begin{tikzpicture}[textbaseline]
				\matrix(m)[math35]{i_n \\ i'_n \\};
				\path[map]	(m-1-1) edge node[left] {$s_n$} (m-2-1);
			\end{tikzpicture}\hspace{-0.25cm}, \!\begin{tikzpicture}[textbaseline]
				\matrix(m)[math35]{q_ni_n & b \\ q_ni'_n & b' \\};
				\path[map]	(m-1-1) edge node[above] {$v$} (m-1-2)
														edge node[left] {$q_ns_n$} (m-2-1)
										(m-1-2) edge node[right] {$t$} (m-2-2)
										(m-2-1) edge node[below] {$v'$} (m-2-2);
				\path				(m-1-2) edge[cell, shorten >= 9pt, shorten <= 10pt] node[below right] {$\xi$} (m-2-1);
			\end{tikzpicture}\!\Bigr)
		\end{displaymath}
		
		The right external equivariance axiom is the commuting of the diagram above left. Using \conref{cleavage from a right action}, firstly notice that it commutes at objects $(i_1, \dotsc, i_n, \map v{q_ni_n}b)$ if and only if $\phi(i_1, \dotsc, i_{n'}, v^*i_n) = (gv)^*\phi(i_1, \dotsc, i_n)$, that is, the targets of the two sides of the equality of the condition (pom) above coincide. Secondly it is easily checked that the actual equality follows from the diagram commuting at the morphism $\bigpars{\id_{i_1}, \dotsc, \id_{i_n}, (\id_{i_n}, \id_v, v)}$. That it likewise implies condition (pcc) is similar.
		
		Conversely, assume that the conditions (pom) and (pcc) hold. We will show that the diagram above commutes at every morphism $M \dfn \bigpars{s_1, \dotsc, s_n, (s_n, \xi, t)}$ of \mbox{$J_1 \times_{A_1} \dotsb \times_{A_{n'}} J_n \times_{A_n} A_n^\2$}, which are of the form as above right; leaving the commutativity at the cells to the reader. Regarding the last two coordinates of $M$ as a morphism $(s_n, \xi, t)$ of $q_n \slash A_n$ (\propref{comma objects from 2-cotensors}), recall from \conref{right action from a cleavage}(m) that they are mapped by $\rho_{q_n}$ to the unique lift $\tilde t$ of $t$ such that \mbox{$\tilde t \of \rho_{q_n}(i_n, v) = \xi^*\bigpars{\rho_{q_n}(i_n', v') \of s_n}$}. To show that the diagram commutes at $M$ is to show that $\rho_q\bigpars{\phi(s_1, \dotsc, s_n), g\xi, gt} = \phi(s_1, \dotsc, s_{n'}, \tilde t)$, which follows from the string of equalities below and the fact that factorisations through the opcartesian morphism $\rho_q\bigpars{\phi(i_1, \dotsc, i_n), gv}$ are unique. The equalities here follow from \conref{right action from a cleavage}(m); condition (pom) and the functoriality of $\phi$; the equality of the (sources of the) cells of (pcc); the defining equation of $\tilde t$ and the functoriality of $\phi$; condition (pom). 
		\begin{align*}
			\rho_q\bigpars{\phi(&s_1, \dotsc, s_n), g\xi, gt} \of \rho_q\bigpars{\phi(i_1, \dotsc, i_n), gv} \\
			&= (g\xi)^*\bigpars{\rho_q(\phi(i_1', \dotsc, i_n'), gv') \of \phi(s_1, \dotsc, s_n)} \\
			&= (g\xi)^*\phi\bigpars{s_1, \dotsc, s_{n'}, \rho_{q_n}(i'_n, v') \of s_n} = \phi\bigpars{s_1, \dotsc, s_{n'}, \xi^*(\rho_{q_n}(i'_n, v') \of s_n)} \\
			&= \phi(s_1, \dotsc, s_{n'}, \tilde t) \of \phi\bigpars{\id_{i_1}, \dotsc, \id_{i_{n'}}, \rho_{q_n}(i_n, v)} \\
			&= \phi(s_1, \dotsc, s_{n'}, \tilde t) \of \rho_q\bigpars{\phi(i_1, \dotsc, i_n), gv}
		\end{align*}
		
		The internal equivariance axiom concerning the actions of $A_m^\2$, where $1 \leq m < n$, is the commuting of the diagram below (\defref{internal profunctor}). We claim that it is equivalent to the $m$th case of the conditions (iem) and (iec) above, as we will partly prove. That the source and target of the morphism considered in (iem) are equal is equivalent to the commuting of the diagram at the object \mbox{$(i_1, \dotsc, i_m, \map u{q_mi_m}{p_{m+1}i_{m+1}}, i_{m+1}, \dotsc, i_n)$} follows immediately from \conref{cleavage from a right action}(ml) and its coop"/dual. To see that the morphism itself is the identity consider the morphism $\bigpars{(\id_{i_m}, \id_u, u), (u, \id_u, \id_{p_{m+1}i_{m+1}})}$ in \mbox{$q_m \slash A_m \times_{A_m} A_m^\2 \iso J_m \times_{A_m} A_m^\2 \times_{A_m} A_m^\2$} (\propref{comma objects from 2-cotensors}). Notice that composition $\bar\alpha$ of $A_m^\2$ (\exref{lax arrow 2-category as internal category}) maps this morphism to the identity of the object $(i_m, u, p_{m+1}i_{m+1})$ in $q_m \slash A_m$, which is mapped to the identity of $u_!i_m$ in $J_m$ by $\rho_{q_m}$; see \conref{cleavage from a right action}(ml). Also from (the coop"/dual of) the latter recall that $\rho_{q_m}(i_m, u) \dfn \rho_{q_m}(\id_{i_m}, \id_u, u)$ and $\lambda_{p_{m+1}}(u, i_{m+1}) \dfn \lambda_{p_{m+1}}(u, \id_u, \id_{i_{m+1}})$. It follows that
		\begin{align*}
			\id&_{\phi(i_1, \dotsc, i_{m'}, u_!i_m, i_{m+1}, \dotsc, i_n)} \\
			&= \phi\bigpars{\id_{i_1}, \dotsc, \id_{i_{m'}}, \rho_{q_m} \of \bar\alpha((\id_{i_m}, \id_u, u), (u, \id_u, \id_{p_{m+1}i_{m+1}})), \id_{i_{m+1}}, \dotsc, \id_{i_n}} \\
			&= \phi\bigpars{\id_{i_1}, \dotsc, \id_{i_{m'}}, \rho_{q_m}(\rho_{q_m}(i_m, u), \id_u, \id_{p_{m+1}i_{m+1}}), \id_{i_{m+1}}, \dotsc, \id_{i_n}} \\
			&= \phi\bigpars{\id_{i_1}, \dotsc, \id_{i_{m'}}, \rho_{q_m}(i_m, u), \lambda_{p_{m+1}}(u, \id_{m+1}), \id_{i_{m+2}}, \dotsc, \id_{i_n}},
		\end{align*}
		where the middle equality follows from associativity of $\rho_{q_m}$ (\defref{internal profunctor}) and the last equality follows from the commuting of the diagram below. That condition (iec) similarly follows from the latter is left to the reader.
		\begin{displaymath}
			\begin{tikzpicture}
				\matrix(m)[math35, column sep={18em,between origins}, row sep={4em,between origins}]
					{	\prod_{\dash} (J_1, \dotsc, J_m) \times_{A_m} A_m^\2 \times_{A_m} \prod_{\dash} (J_{m+1}, \dotsc, J_n) & \prod_{\dash} \ul J \\
						\prod_{\dash} \ul J &  K \\ };
				\path[map]	(m-1-1) edge node[fill=white] {$\id_{\prod_{\dash} (J_1, \dotsc, J_{m'})} \times_{A_{m'}} \rho_{q_m} \times_{A_m} \id_{\prod_{\dash} (J_{m+1}, \dotsc, J_n)}$} (m-2-1)
														edge node[above, inner sep=7pt, xshift=4pt] {$\id_{\prod_{\dash} (J_1, \dotsc, J_m)} \times_{A_m} \lambda_{p_{m+1}} \times_{A_{m+1}} \id_{\prod_{\dash} (J_{m+2}, \dotsc, J_n)}$} (m-1-2)
										(m-1-2) edge node[right] {$\phi$} (m-2-2)
										(m-2-1) edge node[below] {$\phi$} (m-2-2);
			\end{tikzpicture}
		\end{displaymath}
		
		Conversely, assume that the $m$th case of conditions (iem) and (iec) holds. We will show that the diagram above commutes at all morphisms of its top left corner
		\begin{displaymath}
			M \dfn \Bigl(\begin{tikzpicture}[textbaseline]
				\matrix(m)[math35]{i_1 \\ i'_1 \\};
				\path[map]	(m-1-1) edge node[left] {$s_1$} (m-2-1);
			\end{tikzpicture}, \dotsc, \begin{tikzpicture}[textbaseline]
				\matrix(m)[math35]{i_m \\ i'_m \\};
				\path[map]	(m-1-1) edge node[left] {$s_m$} (m-2-1);
			\end{tikzpicture}, \begin{tikzpicture}[textbaseline]
				\matrix(m)[math35, column sep={5em,between origins}]{q_mi_m & p_{m+1}i_{m+1} \\ q_mi'_m & p_{m+1}i_{m+1}' \\};
				\path[map]	(m-1-1) edge node[above] {$u$} (m-1-2)
														edge node[left] {$q_ms_m$} (m-2-1)
										(m-1-2) edge node[right] {$p_{m+1}s_{m+1}$} (m-2-2)
										(m-2-1) edge node[below] {$u'$} (m-2-2);
				\path				(m-1-2) edge[cell, shorten >= 11pt, shorten <= 12pt] node[below right] {$\xi$} (m-2-1);
			\end{tikzpicture}, \begin{tikzpicture}[textbaseline]
				\matrix(m)[math35]{i_{m+1} \\ i'_{m+1} \\};
				\path[map]	(m-1-1) edge node[left] {$s_{m+1}$} (m-2-1);
			\end{tikzpicture}, \dotsc, \begin{tikzpicture}[textbaseline]
				\matrix(m)[math35]{i_n \\ i'_n \\};
				\path[map]	(m-1-1) edge node[left] {$s_n$} (m-2-1);
			\end{tikzpicture}\Bigr),
		\end{displaymath}
		where each $s_k \in J_k$ and $(q_ms_m, \xi, p_{m+1}s_{m+1}) \in A^\2_m$, leaving the case of commutativity at the cells to the reader. Using \conref{right action from a cleavage}(m), applying the bottom leg of the diagram above to $M$ results in the morphism $\phi\bigpars{s_1, \dotsc, s_{m'}, \widetilde{p_{m+1} s_{m+1}}, s_{m+1}, \dotsc, s_n}$ where $\widetilde{p_{m+1} s_{m+1}}$ is the unique lift of $p_{m+1} s_{m+1}$ such that $\widetilde{p_{m+1} s_{m+1}} \of \rho_{q_m}(i_m, u) = \xi^*\bigpars{\rho_{q_m}(i_m', u') \of s_m}$. Analogously, the image of $M$ under the top leg involves the lift $\widetilde{q_ms_m}$ of $q_ms_m$ such that $\lambda_{p_{m+1}}(u', i'_{m+1}) \of \widetilde{q_ms_m} = \xi_!\bigpars{s_{m+1} \of \lambda_{p_{m+1}}(u, i_{m+1})}$. That the diagram commutes at $M$ is shown by
		\begin{align*}
			\phi\bigpars{s_1&, \dotsc, s_{m'}, \widetilde{p_{m+1} s_{m+1}}, s_{m+1}, s_{m+2}, \dotsc, s_n} \\
			&\overset{(1)}= \phi\bigpars{s_1, \dotsc, s_{m'}, \widetilde{p_{m+1} s_{m+1}}, s_{m+1}, s_{m+2}, \dotsc, s_n} \of \phantom\phi\\
			&\hspace{1.5cm} \phi\bigpars{\id_{i_1}, \dotsc, \id_{i_{m'}}, \rho_{q_m}(i_m, u), \lambda_{p_{m+1}}(u, i_{m+1}), \id_{i_{m+2}}, \dotsc, \id_{i_n}} \\
			&\overset{(2)}= \phi\bigpars{s_1, \dotsc, s_{m'}, \xi^*(\rho_{q_m}(i_m', u') \of s_m), s_{m+1} \of \lambda_{p_{m+1}}(u, i_{m+1}), s_{m+2}, \dotsc, s_n} \\
			&\overset{(3)}= \phi\bigpars{s_1, \dotsc, s_{m'}, \rho_{q_m}(i_m', u') \of s_m, \xi_!(s_{m+1} \of \lambda_{p_{m+1}}(u, i_{m+1})), s_{m+2}, \dotsc, s_n} \\
			&\overset{(4)}= \phi\bigpars{s_1, \dotsc, s_{m'}, \rho_{q_m}(i_m', u') \of s_m, \lambda_{p_{m+1}}(u', i_{m+1}') \of \widetilde{q_ms_m}, s_{m+2}, \dotsc, s_n} \\
			&\overset{(5)}= \phi\bigpars{s_1, \dotsc, s_{m'}, s_m, \widetilde{q_ms_m}, s_{m+2}, \dotsc, s_n},
		\end{align*}
		where the equalities (1) and (5) follow from (iem); the equalities (2) and (4) follow from the defining equations of $\widetilde{p_{m+1}s_{m+1}}$ and $\widetilde{q_ms_m}$ together with the functorialty of $\phi$; equality (3) follows from the source and target morphisms of the cell considered in (iec) being equal.
	\end{proof}
	
	Closing this section, the following result describes the $(0, 1)$"/ary cells of $\SpTwoTwoFib$ (\defthreeref{split two-sided 2-fibration}{internal split two-sided fibration}{internal profunctor}). Its proof, which involves arguments similar to those of the proof of \thmref{split two-sided 2-fibration description}, is left to the reader.
	\begin{proposition} \label{(0,1)-ary cells of split two-sided 2-fibrations}
		Let $C \xlar p K \xrar q D$ be a split two"/sided $2$"/fibration. The $(0,1)$"/ary cells
		\begin{displaymath}
			\begin{tikzpicture}
				\matrix(m)[math35, column sep={1.75em,between origins}]{& A & \\ C & & D \\};
				\path[map]	(m-1-2) edge[transform canvas={xshift=-1pt}] node[left] {$f$} (m-2-1)
														edge[transform canvas={xshift=1pt}] node[right] {$g$} (m-2-3)
										(m-2-1) edge[barred] node[below] {$K$} (m-2-3);
				\path				(m-1-2) edge[cell, transform canvas={yshift=-0.25em}] node[right, inner sep=2.5pt] {$\phi$} (m-2-2);
			\end{tikzpicture}
		\end{displaymath}
		of $\SpTwoTwoFib$ are precisely the $2$"/functors $\map\phi AK$ over $f$ and $g$ that satisfy the following conditions:
		\begin{enumerate}
			\item[\textup{(mf)}]	for each morphism $\map uab$ of $A$ the lifts $\map{\lambda(fu, \phi b)}{(fu)^*(\phi b)}{\phi b}$ and $\map{\rho(\phi a, gu)}{\phi a}{(gu)_!(\phi a)}$ in $K$ satisfy $(fu)^*(\phi b) = (gu)_!(\phi a)$ and \mbox{$\lambda(fu, gb) \of \rho(\phi a, gu) = \phi u$};
			\item[\textup{(cf)}]	for each cell $\cell\alpha uv$ of $A$ the lifts $\cell{\lambda(\phi u, f\alpha)}{\phi u}{(f\alpha)_!(\phi u)}$ and $\cell{\rho(g\alpha, \phi v)}{(g\alpha)^*(\phi v)}{\phi v}$ in $K$ satisfy $(f\alpha)_!(\phi u) = (g\alpha)^*(\phi v)$ and \mbox{$\rho(g\alpha, \phi v) \of \lambda(\phi u, f\alpha) = \phi\alpha$}.
		\end{enumerate}
	\end{proposition}
		
	\section[Locally discrete split two-sided 2-fibrations]{Locally discrete split two"/sided $2$"/fibrations} \label{locally discrete split two-sided 2-fibrations section}
	Under the ``two"/sided $2$"/categorical Grothendieck correspondence'' of \cororef{equivalence from yoneda embedding} below $2$"/functors $B \to \Cat^{\op A}$ correspond to split two"/sided $2$"/fibrations $A \brar B$ that are `locally discrete' in the sense of \defref{locally discrete split two-sided 2-fibration} below. Locally discrete split two"/sided $2$"/fibrations generate a sub"/unital virtual equipment of $\SpTwoTwoFib$ (\defref{split two-sided 2-fibration} and \propref{LdSpTwoTwoFib is a unital virtual equipment}), whose cells are described in \propref{cells of locally discrete split two-sided 2-fibrations} and \cororef{nullary cells description}. In \defref{small fibres} we consider locally discrete split two"/sided $2$"/fibrations whose fibres are small.
	
	We start by recalling the $1$"/categorical notion of discrete two"/sided fibration; see e.g.\ Definition~2.3.1 of \cite{Loregian-Riehl20}. Given a morphism $\map sij$ in the source category $J$ of a $1$"/functor $\map fJX$, we say that $s$ is \emph{$f$"/vertical} if $fs$ is an identity morphism.
	\begin{definition} \label{discrete two-sided 1-fibration}
		A \emph{reversed discrete two"/sided $1$"/fibration} is a span $A \xlar p J \xrar q B$ of $1$"/functors satisfying
		\begin{itemize}
			\item[\textup{(l)}] each $\map u{pi}a$ in $A$ has a unique $q$"/vertical lift $\map{\lambda(i, u)}i{u_!i}$ in $J$;
			\item[\textup{(r)}] each $\map vb{qj}$ in $B$ has a unique $p$"/vertical lift $\map{\rho(v, j)}{v^*j}j$ in $J$;
			\item[\textup{(f)}] for each $\map sij$ in $J$ the lifts $\map{\lambda(i,ps)}i{(ps)_!i}$ and $\map{\rho(qs, j)}{(qs)^*j}j$ satisfy $(ps)_!i = (qs)^*j$ and $s = \rho(qs, j) \of \lambda(i, ps)$.
		\end{itemize}
	\end{definition}
	
	A span $A \xlar p J \xrar q B$ of $1$"/functors is said to \emph{jointly reflect identity morphisms} if for any morphism $\map sij$ in $J$, the equalities $ps = \id_{pi}$ and $qs = \id_{qj}$ imply that $i = j$ and $s = \id_i$. The following folklore result shows that, as the notation suggests, the lifts $\lambda(i, u)$ and $\rho(v, j)$ above are opcartesian and cartesian, respectively.
	\begin{proposition} \label{equivalent description of discrete two-sided 1-fibration}
		For a span $J = \brks{A \xlar p J \xrar q B}$ of $1$"/functors the following are equivalent:
		\begin{enumerate}[label=\textup{(\alph*)}]
			\item $J$ is a reversed split two"/sided $1$"/fibration (\remref{relation to two-sided 1-fibration}) that jointly reflects identity morphisms;
			\item $J$ is a reversed discrete two"/sided $1$"/fibration.
		\end{enumerate}
		In either case the cleavages for $p$ and $q$ are unique.
	\end{proposition}
	\begin{proof}
		(a) $\Rar$ (b). In \defref{discrete two-sided 1-fibration}(l) and (r) we can take $\lambda(i, u)$ and $\rho(v, j)$ to be the $q$"/vertical cartesian lift of $u$ and the $p$"/vertical opcartesian lift of $v$, respectively (see \thmref{split two-sided 2-fibration description}); that they are unique follows from their universal properties and the assumption that $J$ jointly reflects identity morphisms. To show \defref{discrete two-sided 1-fibration}(f), factorise $s = \rho(qs, j) \of \hat s \of \lambda(i, ps)$ as in \remref{factorisation of morphisms in J}. Applying the same assumption to $\hat s$, we find that it is an identity morphism.
		
		(b) $\Rar$ (a). We claim that, for each $\map u{pi}a$ in $A$, the unique $q$"/vertical lift $\map{\lambda(i, u)}i{u_!i}$ is opcartesian. To show this, consider morphisms $\map sij$ in $J$ and $\map va{pj}$ in $A$ making the diagram below right commute. To obtain the unique lift $\tilde v$ of $v$, first factorise $s = \rho(qs, j) \of \lambda(i, ps)$ as in the diagram below left, using \defref{discrete two-sided 1-fibration}(f). Consider the unique $q$"/vertical lift $\map{\lambda(u_!i, v)}{u_!i}{v_!(u_!i)}$ of $v$. Uniqueness of $q$"/vertical lifts means that $\lambda(u_!i, v) \of \lambda(i, u) = \lambda(i, ps)$; in particular $v_!(u_! i) = h$.
		\begin{displaymath}
			\begin{tikzpicture}
				\matrix(m)[math35, column sep={5.5em,between origins}, row sep={4.25em,between origins}, xshift=-3cm]{i & u_!i \\ & j \\};
					\draw				($(m-1-1)!0.5!(m-2-2)$) node[text height=1.5ex, text depth=0.25ex] (h) {$h$};
					\path[map]	(m-1-1) edge node[above] {$\lambda(i, u)$} (m-1-2)
															edge[bend right=80] node[left] {$s$} (m-2-2)
											(m-1-2) edge[dashed] node[right] {$\tilde v$} (m-2-2);
					\path[map, transform canvas={xshift=2pt, yshift=-1pt}]	(m-1-1) edge[dashed] node[above right, inner sep=-1pt] {$\lambda(i, ps)$} (h)
											(h)			edge[dashed] node[below left, inner sep=-1.25pt, yshift=-1pt] {$\rho(qs, j)$} (m-2-2);
				
				\matrix(m)[math35, column sep={5.5em,between origins}, row sep={4.25em,between origins}, xshift=3cm]{pi & a \\ & pj \\};
					\path[map]	(m-1-1) edge node[above] {$u$} (m-1-2)
															edge node[below left] {$ps$} (m-2-2)
											(m-1-2) edge node[right] {$v$} (m-2-2);
				\draw	(-0.5cm, 0) edge[mapsto] node[above] {$p$} (0.5cm, 0);
			\end{tikzpicture}
		\end{displaymath}
		
		The required lift of $v$ is $\tilde v \dfn \rho(qs, j) \of \lambda(u_!i, v)$. Indeed, this satisfies $\tilde v \of \lambda(i, u) = s$; to show that it is unique, consider $\map w{u_!i}j$ satisfying $pw = v$ and $w \of \lambda(i, u) = s$ as well. Factorise $w = \rho(qw, j) \of \lambda(u_!i, v)$. Here $\rho(qw, j) = \rho(qs, j)$ by uniqueness of $p$"/vertical lifts since $q\rho(qw, j) = qw = q(w \of \lambda(i, u)) = qs$, and we conclude that $w = \tilde v$. This completes the proof of $\lambda(i, u)$ being opcartesian. That the $q$"/vertical opcartesian lifts $\lambda(i, u)$ form a cleavage making $p$ into a split $1$"/opfibration follows from their uniqueness. That $q$ is a split $1$"/fibration is dual. That the ``reverse'' of condition (cml) of \thmref{split two-sided 2-fibration description} holds (i.e.\ with $p$ and $q$ as well as $\lambda$ and $\rho$ interchanged) follows immediately from applying \defref{discrete two-sided 1-fibration}(f) to $s \dfn \lambda(v, i) \of \rho(u, i)$.
	\end{proof}
	
	A span $J = \brks{A \xlar p J \xrar q B}$ of $2$"/functors is said to \emph{jointly reflect identity cells} if for any cell $\cell\sigma st$ in $J$, the equalities $p\sigma = \id_{ps}$ and $q\sigma = \id_{qt}$ imply that $s = t$ and $\sigma = \id_s$. Assume that $J$ is a split two"/sided $2$"/fibration (\defref{split two-sided 2-fibration}). By the previous proposition and \remref{relation to two-sided 1-fibration} the two conditions on $J$ below are equivalent.
	\begin{definition} \label{locally discrete split two-sided 2-fibration}
		A split two"/sided 2"/fibration $J = \brks{A \xlar p J \xrar q B}$ is called \emph{locally discrete} if the following equivalent conditions hold:
		\begin{enumerate}
			\item[\textup{(jr)}] $J$ jointly reflects identity cells, in the above sense;
			\item[\textup{(ld)}] $J$ is locally a reversed discrete two"/sided $1$"/fibration, that is, each span of hom"/functors $A(pi, pj) \leftarrow J(i, j) \rightarrow B(qi, qj)$ is a reversed discrete two"/sided $1$"/fibration (\defref{discrete two-sided 1-fibration}).
		\end{enumerate}
		
		We denote by $\LdSpTwoTwoFib \subset \SpTwoTwoFib$ the locally full sub"/augmented virtual double category generated by locally discrete split two"/sided 2"/fibrations.
	\end{definition}
	
	Notice that a span $\hmap JAB$ of $2$"/functors is a locally discrete split two"/sided $2$"/fibration precisely if it satisfies conditions (c), (o), (v) and (cml) of \thmref{split two-sided 2-fibration description} as well as either (jr) or (ld) above. Condition (ccl) of the theorem is then implied by condition (ld), by \remref{relation to two-sided 1-fibration} and \propref{equivalent description of discrete two-sided 1-fibration}.
	
	\begin{example} \label{lax comma 2-categories are locally discrete}
		Consider the split two"/sided $2$"/fibration $A \xlar{\pi_A} f \slash g \xrar{\pi_B} B$ associated to the lax comma $2$"/category $f \slash g$ of a cospan of $2$"/functors $A \xrar f C \xlar g B$ (\exref{lax comma 2-category is a split two-sided 2-fibration}). Clearly, it jointly reflects identity cells; hence it is locally discrete. In particular, the split two"/sided unit $2$"/fibration $A \xlar{\src} A^\2 \xrar{\tgt} A$ associated to a lax arrow $2$"/category $A^\2 \dfn \id_A \slash \id_A$ is locally discrete.
	\end{example}
	
	It is easily checked that the restriction $K(f, g)$ (\cororef{internal split two-sided fibrations form a unital virtual equipment}(a)) of a locally discrete split two"/sided $2$"/fibration $K$ is again locally discrete. Together with the previous example we conclude the following.
	\begin{proposition} \label{LdSpTwoTwoFib is a unital virtual equipment}
		$\LdSpTwoTwoFib$ is a unital virtual equipment (\defref{augmented virtual equipment}).
	\end{proposition}
	
	\subsubsection*{Locally discrete split op"/$2$"/fibrations}
	Recall that a $1$"/functor $\map qJB$ is a \emph{discrete $1$"/fibration} if every $\map vb{qj}$ in $B$ has a unique lift $\map{\rho(v, j)}{v_*j}j$ in $J$. This is equivalent to either the span \mbox{$1 \xlar ! J \xrar q B$} being a discrete two"/sided $1$"/fibration (\defref{discrete two-sided 1-fibration}) or $q$ being a split $1$"/fibration that reflects identity morphisms (use \propref{equivalent description of discrete two-sided 1-fibration}).
	
	The following result is essentially the first part of Proposition~3.17 of \cite{Lambert24}.
	\begin{proposition}[Lambert] \label{locally discrete split op-2-fibration description}
		For a $2$"/functor $\map qJB$ the following are equivalent:
		\begin{enumerate}[label=\textup{(\alph*)}]
			\item $1 \leftarrow J \xrar q B$ is a locally discrete split two"/sided $2$"/fibration;
			\item $\map{\und q}{\und J}{\und B}$ is a split $1$"/opfibration and $q$ is locally a discrete $1$"/fibration.
		\end{enumerate}
	\end{proposition}
	\begin{proof}
		(a) $\Rar$ (b) follows from \remref{relation to two-sided 1-fibration} and \defref{locally discrete split two-sided 2-fibration}. (b) $\Rar$ (a) That the $1$"/dimensional cleavage making $\und q$ into a split $1$"/opfibration extends to a $2$"/dimensional cleavage making $q$ into a split op"/$2$"/fibration (\defref{split op-2-fibration}) is straightforward; see also the proof of Proposition~3.17 of \cite{Lambert24}. Hence $1 \leftarrow J \xrar q B$ is a split two"/sided $2$"/fibration by \cororef{split 2-opfibrations as split two-sided 2-fibrations}, which is locally discrete because it jointly reflects identity cells. 
	\end{proof}
	
	The following are variants on Definition~2.4 of \cite{Lambert24}, which introduces the notion of ``split discrete $2$"/fibration''.
	\begin{definition}[Lambert] \label{locally discrete split op-2-fibration}
		A $2$"/functor $\map qJB$ is a \emph{locally discrete split op"/$2$"/fibration} if it satisfies either of the equivalent conditions of \propref{locally discrete split op-2-fibration description} above. Coop"/dually a $2$"/functor $\map pJA$ is a \emph{locally discrete split co"/$2$"/fibration} if $A \xlar p J \rightarrow 1$ is a locally discrete split two"/sided $2$"/fibration or, equivalently, if $\und p$ is a split $1$"/fibration and $p$ is locally a discrete $1$"/opfibration. A \emph{locally discrete split $2$"/fibration} is a $2$"/functor $f$ such that $\op f$ is a locally discrete split op"/$2$"/fibration.
	\end{definition}
	
	\begin{remark}
		Consider a span $J = \brks{A \xlar p J \xrar q B}$ of $2$"/functors that is locally a reversed discrete two"/sided $1$"/fibration (\defref{discrete two-sided 1-fibration}) and whose underlying span of $1$"/functors is a split two"/sided $1$"/fibration (\remref{relation to two-sided 1-fibration}). In light of \propref{locally discrete split op-2-fibration description} one might hope to be able to show that $J$ is a split two"/sided $2$"/fibration (and hence a locally discrete such). That does not seem to be possible in general. In particular, it is not clear to the author how to obtain the $2$"/dimensional universal property (om2) of \defref{split op-2-fibration}, that is required of the cleavage of $q$.
	\end{remark}
	
	\begin{remark}
		Consider the $2$"/category $\TwoCat$ of small $2$"/categories, $2$"/functors and $2$"/natural transformations. Lemma~3.30 of \cite{Lambert24} shows that the notion of `discrete fibration $\map fEB$ internal to $\TwoCat$', in the sense of e.g.\ Section~2.2 of \cite{Weber07}, is more restrictive than the notion of locally discrete split $2$"/fibration above. In particular, the former notion implies that the underlying $1$"/functor $\und f$ is a discrete $1$"/fibration, while the latter notion's weaker requirement is that $\und f$ be a split $1$"/fibration.
	\end{remark}
	
	\subsubsection*{Multimorphisms of locally discrete split two"/sided $2$"/fibrations}
	The remainder of this section is largely devoted to describing the cells of $\LdSpTwoTwoFib$. The following is \thmref{unary cells of split two-sided 2-fibrations} restricted to cells $\ul J \Rar K$ with a locally discrete target $K$; here too we will use the notation $\prod_{\dash} \ul J$ of \propref{profunctors form an augmented virtual double category}.
	\begin{proposition} \label{cells of locally discrete split two-sided 2-fibrations}
		Consider an iterated span \mbox{$\ul J = (A_0 \xlar{p_1} J_1 \xrar{q_1} \dotsb \xlar{p_n} J_n \xrar{q_n} A_n)$} of (locally discrete) split two"/sided $2$"/fibrations (\defref{split two-sided 2-fibration}), a locally discrete split two"/sided $2$"/fibration $C \xlar p K \xrar q D$ and $2$"/functors $\map f{A_0}C$ and $\map g{A_n}D$, together making up the boundary of the unary cell below.
		
		Any $2$"/functor \mbox{$\map\phi{J_1 \times_{A_1} \dotsb \times_{A_{n'}} J_n}K$} over $f$ and $g$ satisfies the following conditions:
		\begin{enumerate}
			\item[\textup{(poc)}] $\phi\bigpars{\lambda_{p_1}(s_1, \alpha), \id_{s_2}, \dotsc, \id_{s_n}} = \lambda_p\bigpars{\phi(s_1, \dotsc, s_n), f\alpha}$ for any morphism $(s_1, \dotsc, s_n)$ of $\prod_{\dash} \ul J$ and $\cell \alpha{p_1s_1}u$ in $A_0$;
			\item[\textup{(iec)}] for any $1 \leq m < n$, morphisms $(s_1, \dotsc, s_m)$ in $\prod_{\dash} (J_1, \dotsc, J_m)$ and $(s_{m+1}, \dotsc, s_n)$ in $\prod_{\dash} (J_{m+1}, \dotsc, J_n)$, as well as $\cell \alpha{p_{m+1}s_{m+1}}{q_ms_m}$ in $A_m$, the cell
			\begin{displaymath}
				\phi\bigpars{\id_{s_1}, \dotsc, \id_{s_{m'}}, \rho_{q_m}(\alpha, s_m), \lambda_{p_{m+1}}(s_{m+1}, \alpha), \id_{s_{m+2}}, \dotsc, \id_{s_n}}
			\end{displaymath}
			is an identity cell;
			\item[\textup{(pcc)}] $\phi\bigpars{\id_{s_1}, \dotsc, \id_{s_{n'}}, \rho_{q_n}(\beta, s_n)} = \rho_q\bigpars{g\beta, \phi(s_1, \dotsc, s_n)}$ for any morphism $(s_1, \dotsc, s_n)$ of $\prod_{\dash} \ul J$ and $\cell \beta u{q_ns_n}$ in $A_n$.
		\end{enumerate}
		
		Moreover, unary cells
		\begin{displaymath}
			\begin{tikzpicture}
				\matrix(m)[math35]{A_0 & A_1 & A_{n'} & A_n \\ C & & & D \\};
				\path[map]	(m-1-1) edge[barred] node[above] {$J_1$} (m-1-2)
														edge node[left] {$f$} (m-2-1)
										(m-1-3) edge[barred] node[above] {$J_n$} (m-1-4)
										(m-1-4) edge node[right] {$g$} (m-2-4)
										(m-2-1) edge[barred] node[below] {$K$} (m-2-4);
				\path[transform canvas={xshift=1.75em}]	(m-1-2) edge[cell] node[right] {$\phi$} (m-2-2);
				\draw				($(m-1-2)!0.5!(m-1-3)$) node {$\dotsb$};
			\end{tikzpicture}
		\end{displaymath}
		of $\SpTwoTwoFib$ (respectively $\LdSpTwoTwoFib$) correspond precisely to those $2$"/functors \mbox{$\map\phi{J_1 \times_{A_1} \dotsb \times_{A_{n'}} J_n}K$} over $f$ and $g$ that satisfy the following additional conditions:
		\begin{enumerate}
			\item[\textup{(pcm)}] $\phi\bigpars{\lambda_{p_1}(u, i_1), \id_{i_2}, \dotsc, \id_{i_n}} = \lambda_p\bigpars{fu, \phi(i_1, \dotsc, i_n)}$ for any object $(i_1, \dotsc, i_n)$ of $\prod_{\dash} \ul J$ and $\map ua{p_1i_1}$ in $A_0$;
			\item[\textup{(iem)}] for any $1 \leq m < n$, objects $(i_1, \dotsc, i_m)$ in $\prod_{\dash} (J_1, \dotsc, J_m)$ and $(i_{m+1}, \dotsc, i_n)$ in $\prod_{\dash} (J_{m+1}, \dotsc, J_n)$, as well as $\map u{q_mi_m}{p_{m+1}i_{m+1}}$ in $A_m$, the morphism
			\begin{displaymath}
				\phi\bigpars{\id_{i_1}, \dotsc, \id_{i_{m'}}, \rho_{q_m}(i_m, u), \lambda_{p_{m+1}}(u, i_{m+1}), \id_{i_{m+2}}, \dotsc, \id_{i_n}}
			\end{displaymath}
			is an identity morphism;
			\item[\textup{(pom)}] $\phi\bigpars{\id_{i_1}, \dotsc, \id_{i_{n'}}, \rho_{q_n}(i_n, v)} = \rho_q\bigpars{\phi(i_1, \dotsc, i_n), gv}$ for any object $(i_1, \dotsc, i_n)$ of $\prod_{\dash} \ul J$ and $\map v{q_ni_n}b$ in $A_n$.
		\end{enumerate}
	\end{proposition}
	\begin{proof}
		The six conditions above are those of \thmref{unary cells of split two-sided 2-fibrations}. Assuming that $K$ is locally discrete, we will prove that the three conditions (poc), (iec) and (pcc) are satisfied by any $2$"/functor \mbox{$\map\phi{J_1 \times_{A_1} \dotsb \times_{A_{n'}} J_n}K$}.
		
		Using the fact that $\phi$ is a morphism of spans, notice that the left"/hand side of the equality asserted by (poc) is a $q$"/vertical lift of $\cell{f\alpha}{\brks{(p \of \phi)(s_1, \dotsc, s_n) = (f \of p_1)(s_1)}}{fu}$ in $K$. Since the right"/hand side is so too by definition, the equality follows by uniqueness of such lifts. Proving (pcc) is similar.
		
		Consider $s_1 \in J_1$, $(s_2, \dotsc, s_n) \in \prod_{\dash}(J_2, \dotsc, J_n)$ and $\cell\alpha{p_2s_2}{q_1s_1}$ in $A_1$. The case $m = 1$ of condition (iec) asserts that $\tau \dfn \phi\bigpars{\rho_{q_1}(\alpha, s_1), \lambda_{p_2}(s_2, \alpha), \id_{s_3}, \dotsc, \id_{s_n}}$ is an identity cell. Using that $\phi$ is a morphism of spans and that $\rho_{q_1}(\alpha, s_1)$ is $p_1$"/vertical, we find that the $p$ and $q$ images of $\tau$ are $\id_{(f \of p_1)(s_1)}$ and $\id_{(g \of q_n)(s_n)}$, respectively, so that the assertion follows from the fact that $p$ and $q$ jointly reflect identity cells. Proving the other cases of (iec) is similar.
	\end{proof}
	
	\subsubsection*{Marked lax natural transformations}
	Using the fact that nullary cells $\ul J \Rar C$ in $\SpTwoTwoFib$ correspond bijectively to unary cells $\ul J \Rar C^\2$ (\cororef{correspondences induced by horizontal units}), with $C^\2$ the lax arrow $2$"/category that forms the horizontal unit of $C$ (\exref{lax comma 2-categories are locally discrete}), the previous proposition implies the following description of such nullary cells too. The proof of the following corollary easily follows from the result above, by using the universal properties of $C^\2$ as a comma object in $\SpTwoTwoFib$ and its description as a split two"/sided $2$"/fibration (\exsref{lax comma 2-category}{lax comma 2-category is a split two-sided 2-fibration}).
	\begin{corollary} \label{nullary cells description}
		Consider a path $\hmap{(J_1, \dotsc, J_n)}{A_0}{A_n}$ of $n \geq 1$ (locally discrete) split two"/sided $2$"/fibrations as well as $2$"/functors $\map f{A_0}C$ and $\map g{A_n}C$, making up the boundary of the nullary cell below left.
		\begin{displaymath}
			\begin{tikzpicture}[baseline]
				\matrix(m)[math35, column sep={1.75em,between origins}]
					{A_0 & & A_1 & & A_{n'} & & A_n \\ & & & C & & & \\};
				\path[map]	(m-1-1) edge[barred] node[above] {$J_1$} (m-1-3)
														edge node[below left] {$f$} (m-2-4)
										(m-1-5) edge[barred] node[above] {$J_n$} (m-1-7)
										(m-1-7) edge node[below right] {$g$} (m-2-4);
				\path				(m-1-4) edge[cell] node[right] {$\phi$} (m-2-4);
				\draw				(m-1-4) node {$\dotsb$};
			\end{tikzpicture} \qquad\qquad\qquad \begin{tikzpicture}[baseline]
						\matrix(m)[math35]
							{& J_1 \times_{A_1} \dotsb \times_{A_{n'}} J_n & \\ A_0 & \phantom{X} & A_n \\ & C & \\};
						\path[map]	(m-1-2) edge node[left] {$p_1 \of \pi_{J_1}$} (m-2-1)
																edge node[right] {$q_n \of \pi_{J_n}$} (m-2-3)
												(m-2-1) edge node[below left] {$f$} (m-3-2)
												(m-2-3)	edge node[below right] {$g$} (m-3-2);
						\path[transform canvas={xshift=1.75em}]	(m-2-1) edge[cell] node[above] {$\phi$} (m-2-2);
					\end{tikzpicture}
		\end{displaymath}
		
		Any lax natural transformation (\defref{lax natural transformation definition}) above right satisfies the following conditions:
				\begin{enumerate}
			\item[\textup{(poc)}] $\phi_{\pars{\alpha_! s_1, s_2, \dotsc, s_n}} = \phi_{(s_1, \dotsc,s_n)} \hc (\phi_{(j_1, \dotsc, j_n)} \of f\alpha)$ for any morphism $(s_1, \dotsc, s_n)$ with target $(j_1, \dotsc, j_n)$ in $\prod_{\dash} \ul J$ and $\cell \alpha{p_1s_1}u$ in $A_0$;
			\item[\textup{(iec)}] for any $1 \leq m < n$, morphisms $(s_1, \dotsc, s_m)$ in $\prod_{\dash} (J_1, \dotsc, J_m)$ and $(s_{m+1}, \dotsc, s_n)$ in $\prod_{\dash} (J_{m+1}, \dotsc, J_n)$, as well as $\cell \alpha{p_{m+1}s_{m+1}}{q_ms_m}$ in $A_m$, the lax naturality cells $\phi_{(s_1, \dotsc, s_{m'}, \alpha^*s_m, s_{m+1}, s_{m+2}, \dotsc, s_n)}$ and $\phi_{(s_1, \dotsc, s_{m'}, s_m, \alpha_!s_{m+1}, s_{m+2}, \dotsc, s_n)}$ are equal;
			\item[\textup{(pcc)}] $\phi_{\pars{s_1, \dotsc, s_{n'}, \alpha^*s_n}} = (g\alpha \of \phi_{(i_1, \dotsc, i_n)}) \hc \phi_{(s_1, \dotsc, s_n)}$ for any morphism $(s_1, \dotsc, s_n)$ with source $(i_1, \dotsc, i_n)$ in $\prod_{\dash} \ul J$ and $\cell \alpha u{q_ns_n}$ in $A_n$.
		\end{enumerate}

		Moreover, nullary cells in $\SpTwoTwoFib$ (respectively $\LdSpTwoTwoFib$) of the form above left correspond precisely to lax natural transformations above right satisfying the following additional conditions:
		\begin{enumerate}
			\item[\textup{(cm)}] $\phi_{(u^*i_1, i_2, \dotsc, i_n)} = \phi_{(i_1, \dotsc, i_n)} \of fu$ for any object $(i_1, \dotsc, i_n)$ of $\prod_{\dash} \ul J$ and morphism $\map ua{p_1i_1}$ in $A_0$, and the lax naturality cell $\phi_{\pars{\lambda_{p_1}(u, i_1), \id_{i_2}, \dotsc, \id_{i_n}}}$ is an identity cell;
			\item[\textup{(im)}] for any $1 \leq m < n$, objects $(i_1, \dotsc, i_m)$ in $\prod_{\dash} (J_1, \dotsc, J_m)$ and $(i_{m+1}, \dotsc, i_n)$ in $\prod_{\dash} (J_{m+1}, \dotsc, J_n)$, as well as $\map u{q_mi_m}{p_{m+1}i_{m+1}}$ in $A_m$, the components $\phi_{(i_1, \dotsc, i_{m'}, i_m, u^*i_{m+1}, i_{m+2}, \dotsc, i_n)}$ and $\phi_{(i_1, \dotsc, i_{m'}, u_!i_m, i_{m+1}, i_{m+2}, \dotsc, i_n)}$ are equal and the lax naturality cell $\phi_{\pars{\id_{i_1}, \dotsc, \id_{i_{m'}}, \rho_{q_m}(i_m, u), \lambda_{p_{m+1}}(u, i_{m+1}), \id_{i_{m+2}}, \dotsc, \id_{i_n}}}$ is an identity cell;
			\item[\textup{(om)}] $\phi_{(i_1, \dotsc, i_{n'}, u_!i_n)} = gu \of \phi_{(i_1, \dotsc, i_n)}$ for any object $(i_1, \dotsc, i_n)$ of $\prod_{\dash} \ul J$ and morphism $\map u{q_ni_n}a$ in $A_n$, and the lax naturality cell $\phi_{\pars{\id_{i_1}, \dotsc, \id_{i_{n'}}, \rho_{q_n}(i_n, u)}}$ is an identity cell.
		\end{enumerate}
	\end{corollary}
	
	\begin{definition} \label{marked lax natural transformation}
		Let $\hmap{(J_1, \dotsc, J_n)}{A_0}{A_n}$ be a path of $n \geq 1$ split two"/sided $2$"/fibrations and $\map f{A_0}C$ and $\map g{A_n}C$ be $2$"/functors. A lax natural transformation $\cell\phi{f \of p_1 \of \pi_{J_1}}{g \of q_n \of \pi_{J_n}}$ above right is called \emph{marked} if it satisfies the conditions \textup{(cm)}, \textup{(im)} and \textup{(om)} above, that is, it forms a nullary cell $\ul J \Rar C$ above left in $\SpTwoTwoFib$.
	\end{definition}
	
	\begin{remark} \label{cartesian-marked}
		Here we use the terminology of \cite{Mesiti24} as follows. Given a $2$"/functor $\map gB\Cat$ into the $2$"/category $\Cat$ of small $1$"/categories, $1$"/functors and $1$"/natural transformations, consider its `$2$"/category of elements' split op"/$2$"/fibration $g^\yon$ as constructed in \conref{2-category of elements} below (take $A = 1$ the terminal $2$"/category). Given parallel $2$"/functors $h$, $\map k{g^\yon}C$, a lax natural transformation $\cell\phi hk$ is called ``cartesian"/marked'' in Definition~2.7 of op.\ cit.\ whenever its lax naturality cells $\phi_s$ are identity cells at each opcartesian morphism $s = \rho((*, i, b_0), v)$ in $g^\yon$ (see \conref{2-category of elements}). In Section I,2.4 of \cite{Gray74} a more general, unnamed notion of ``marked lax natural transformation'' is considered which subsumes Mesiti's notion.
	\end{remark}
	
	The following is \propref{(0,1)-ary cells of split two-sided 2-fibrations} restricted to $(0,1)$"/ary cells $A \Rar K$ with a locally discrete target $K$.
	\begin{corollary}
		Consider a locally discrete split two"/sided $2$"/fibration $C \xlar p K \xrar q D$ as well as $2$"/functors $\map fAC$ and $\map gAD$. Any $2$"/functor $\map\phi AK$ over $f$ and $g$ satisfies condition \textup{(cf)} of \propref{(0,1)-ary cells of split two-sided 2-fibrations}. The $(0,1)$"/ary cells
		\begin{displaymath}
			\begin{tikzpicture}
				\matrix(m)[math35, column sep={1.75em,between origins}]{& A & \\ C & & D \\};
				\path[map]	(m-1-2) edge[transform canvas={xshift=-1pt}] node[left] {$f$} (m-2-1)
														edge[transform canvas={xshift=1pt}] node[right] {$g$} (m-2-3)
										(m-2-1) edge[barred] node[below] {$K$} (m-2-3);
				\path				(m-1-2) edge[cell, transform canvas={yshift=-0.25em}] node[right, inner sep=2.5pt] {$\phi$} (m-2-2);
			\end{tikzpicture}
		\end{displaymath}
		of $\SpTwoTwoFib$ and $\LdSpTwoTwoFib$ are precisely the $2$"/functors $\map\phi AK$ over $f$ and $g$ that satisfy condition \textup{(mf)} of \propref{(0,1)-ary cells of split two-sided 2-fibrations}.
	\end{corollary}
	\begin{proof}
		Let $\cell\alpha uv$ be a cell of $A$ between morphisms $u$, $\map vab$. That condition (cf) of \propref{(0,1)-ary cells of split two-sided 2-fibrations} holds for $\alpha$ follows from applying \defref{discrete two-sided 1-fibration}(f) to the image \mbox{$\cell{\phi\alpha}{\phi u}{\phi v}$} that is contained in the apex of the reversed discrete two"/sided $1$"/fibration of hom"/categories $C(fa, fb) \xlar p K(\phi a, \phi b) \xrar q D(ga, gb)$ (\defref{locally discrete split two-sided 2-fibration}).
	\end{proof}
	
	To complete the description of cells in $\SpTwoTwoFib$ and $\LdSpTwoTwoFib$ we state the following corollary, which follows directly from \defsref{split two-sided 2-fibration}{internal split two-sided fibration}, \thmref{embedding of strict cells} and the fact that strict lax natural transformations are precisely the $2$"/natural ones (\exref{lax natural transformation}).
	\begin{corollary} \label{vertical cells description}
		Vertical cells in $\SpTwoTwoFib$ (\defref{split two-sided 2-fibration}) and $\LdSpTwoTwoFib$ (\defref{locally discrete split two-sided 2-fibration}), of the form below, are precisely $2$"/natural transformations \mbox{$\cell\psi fg$}.
		\begin{displaymath}	
			\begin{tikzpicture}[baseline]
				\matrix(m)[math35]{A_0 \\ C \\};
				\path[map]	(m-1-1) edge[bend right=45] node[left] {$f$} (m-2-1)
														edge[bend left=45] node[right] {$g$} (m-2-1);
				\path				(m-1-1) edge[cell] node[right] {$\psi$} (m-2-1);
			\end{tikzpicture}
		\end{displaymath}
	\end{corollary} 
	
	\subsubsection*{Locally discrete split two"/sided $2$"/fibrations with small fibres}
	We close this section by considering locally discrete split two"/sided $2$"/fibrations with small fibres, a notion that we will use in stating our main result (\cororef{main result}).
	\begin{definition} \label{small fibres}
		Given a locally discrete split two"/sided $2$"/fibration \mbox{$J = \brks{A \xlar p J \xrar q B}$}, consider a pair of objects $a \in A$ and $b \in B$.
		\begin{enumerate}[label=-]
			\item The \emph{fibre of $(a, b)$} is the sub"/$1$"/category $J_{(a, b)} \subset \und J$ whose objects $i \in J$ satisfy $pi = a$ and $qi = b$ and whose morphisms $\map sij$ in $J$ are such that $ps = \id_a$ and $qs = \id_b$.
			\item If all fibres $J_{(a, b)}$ are small categories, then $J$ is said to have \emph{small fibres}.
		\end{enumerate}
		We denote by $\SfLdSpTwoTwoFib \subset \LdSpTwoTwoFib$ the locally full sub"/augmented virtual equipment (\defref{augmented virtual equipment}) generated by those locally discrete split two"/sided $2$"/fibrations that have small fibres.
	\end{definition}
	
	That $\SfLdSpTwoTwoFib$ has all unary restrictions $K(f, g)$ (\cororef{internal split two-sided fibrations form a unital virtual equipment}(a) and \propref{LdSpTwoTwoFib is a unital virtual equipment}) follows from the isomorphism of fibres $K(f, g)_{(a, b)} \iso K_{(fa, fb)}$. Likewise, nullary restrictions $C(f, g) \dfn f \slash g$ (\exref{lax comma 2-category is a split two-sided 2-fibration}) have hom"/categories $C(f, g)_{(a, b)} \iso C(fa, gb)$ as fibres, so that $C(f, g)$ exists in $\SfLdSpTwoTwoFib$ if and only if $C(fa, gb)$ is small for all pairs $(a, b)$. In particular, a large $2$"/category $A$ is unital (\defref{cartesian cells}) in $\SfLdSpTwoTwoFib$ precisely if it is locally small.
	
	\section[The Yoneda 2-functor is dense]{The Yoneda $2$"/functor is dense} \label{density section}	
	In the final pair of sections we will show that the Yoneda $2$"/functors $\map\yon A{\Cat^{\op A}}$, whose definition is recalled below, form formal Yoneda embeddings, in the sense of Definition~4.5 of \cite{Koudenburg24}, in the augmented virtual equipment $\SfLdSpTwoTwoFib$ of locally discrete split two"/sided $2$"/fibrations with small fibres (\defref{small fibres}). This means proving that $\yon$ is formally dense, in the sense recalled in \defref{density} below, and that it satisfies a formal `Yoneda axiom', as recalled in \defref{Yoneda embedding} below. In \thmref{Yoneda 2-functor is dense} below we prove that $\yon$ is dense; in the next section (\thmref{Yoneda axiom}) we prove that it satisfies the Yoneda axiom.
	
	Recall that $\Cat$ denotes the $2$"/category of small $1$"/categories, functors and $1$"/natural transformations.
	\begin{construction}[Yoneda $2$"/functor] \label{Yoneda 2-functor}
		Let $A$ be a locally small $2$"/category, that is, each hom"/category $A(a_0, a_1)$ is small, and consider the $2$"/functor $2$"/category $\Cat^{\op A}$ of $2$"/functors $\op A \to \Cat$, $2$"/natural transformations and modifications (\defref{modification}). The \emph{Yoneda $2$"/functor} $\map\yon A{\Cat^{\op A}}$ maps each object $a \in A$ to the contravariant hom"/$2$"/functor $\map{A(\dash, a)}{\op A}\Cat$. It maps each morphism $\map ua{a'}$ to the $2$"/natural transformation given by postcomposition with $u$, and each cell $\cell\alpha uv$ to the modification given by postcomposition with $\alpha$.
	\end{construction}
	
	\subsubsection*{Formal density}
	The following definition combines Definitions~1.9 and 4.3 of \cite{Koudenburg24}. The juxtaposition of the nullary cells $\eta$ and $\phi'$ on the right"/hand side of the equation below denotes their horizontal composition $\eta \hc \phi' \dfn \id_{\id_M} \of (\eta, \phi')$, where $\id_{\id_M}$ denotes the identity cell at the object $M$; see Lemma~1.3 of \cite{Koudenburg20} for details. If $\eta$ and $\phi'$ are nullary cells of $\LdSpTwoTwoFib$ which, using \cororef{nullary cells description}, we view as lax natural transformations, then $\eta \hc \phi'$ is the vertical composite of lax natural transformations $(\eta \of \pi_J) \hc (\phi' \of \pi_{\prod_{\dash} \ul H})$ (see e.g.\ Definition~4.2.15 of \cite{Johnson-Yau21}), where $\pi_J$ and $\pi_{\prod_{\dash} \ul H}$ denote the projections $J \leftarrow \prod_{\dash} (J \conc \ul H) \to \prod_{\dash} \ul H$ (see \propref{profunctors form an augmented virtual double category}). To see this use the definition of the composition $2$"/functor $\map{\bar\mu}{M^\2 \times_M M^\2}{M^\2}$ (\exref{lax arrow 2-category as internal category}).
	\begin{definition} \label{density}
		In any augmented virtual double category consider a nullary cell $\eta$ as in the composite on the right"/hand side below. It is said to define $\map l BM$ as the \emph{left Kan extension} of $\map d AM$ along $\hmap JAB$ if any nullary cell $\phi$ as on the left"/hand side below, where $\ul H = (H_1, \dotsc, H_n)$ is any (possibly empty) path, factors uniquely through $\eta$ as a nullary cell $\phi'$, as shown. In that case the cell $\eta$ is called \emph{left Kan}.
		\begin{displaymath}
			\begin{tikzpicture}[textbaseline]
				\matrix(m)[math35]{A & B & B_1 & B_{n'} & B_n \\ & & M & & \\};
				\path[map]	(m-1-1) edge[barred] node[above] {$J$} (m-1-2)
														edge[transform canvas={yshift=-2pt}] node[below left] {$d$} (m-2-3)
										(m-1-2) edge[barred] node[above] {$H_1$} (m-1-3)
										(m-1-4) edge[barred] node[above] {$H_n$} (m-1-5)
										(m-1-5) edge[transform canvas={yshift=-2pt}] node[below right] {$k$} (m-2-3);
				\path				(m-1-3) edge[cell] node[right] {$\phi$} (m-2-3);
				\draw				($(m-1-3)!0.5!(m-1-4)$) node {$\dotsb$};
			\end{tikzpicture} = \begin{tikzpicture}[textbaseline]
				\matrix(m)[math35]{A & B & B_1 & B_{n'} & B_n \\ & M & & & \\};
				\path[map]	(m-1-1) edge[barred] node[above] {$J$} (m-1-2)
														edge[transform canvas={yshift=-2pt}] node[below left] {$d$} (m-2-2)
										(m-1-2) edge[barred] node[above] {$H_1$} (m-1-3)
														edge node[right] {$l$} (m-2-2)
										(m-1-4) edge[barred] node[above] {$H_n$} (m-1-5)
										(m-1-5) edge[transform canvas={yshift=-2pt}] node[below right] {$k$} (m-2-2);
				\path				(m-1-2) edge[cell, transform canvas={shift={(3em,0.333em)}}] node[right] {$\phi'$} (m-2-2)
										(m-1-2) edge[cell, transform canvas={shift={(-1.2em,0.333em)}}] node[right] {$\eta$} (m-2-2);
				\draw				($(m-1-3)!0.5!(m-1-4)$) node {$\dotsb$};
			\end{tikzpicture}
  	\end{displaymath}
  	
  	A morphism $\map dAM$ is called \emph{dense} if, for any morphism $\map gBM$ such that the nullary restriction $\hmap{M(d, g)}AB$ exists (\defref{cartesian cells}), the nullary cartesian cell $M(d, g) \Rar M$ defining $M(d, g)$ is left Kan.
  \end{definition}
  
  \begin{remark}
  	Mesiti in Definition~3.12 of \cite{Mesiti24} defines the notion of pointwise left Kan extension $\map lBM$ of a $2$"/functor $\map dAM$ along a (using our terminology) locally discrete (not necessarily split) op"/$2$"/fibration $\map qAB$. Such an extension is exhibited by a marked lax natural transformation $d \Rar l \of q$ (see \propref{nullary cells description} and \remref{cartesian-marked}). In general the relation between Mesiti's notion and the one above, when applied to the unital virtual equipment $\LdSpTwoTwoFib$ of locally discrete split two"/sided $2$"/fibrations (\defref{locally discrete split two-sided 2-fibration}), is presently unclear to the author. In particular, two"/sided $2$"/fibrations are not considered in \cite{Mesiti24}.
  	
  	However, in the specific case of the pointwise left Kan extensions that are relevant to the density of the Yoneda $2$"/functor $\map\yon 1\Cat$ (\conref{Yoneda 2-functor}), for the terminal $2$"/category $A = 1$, Mesiti's notion is an instance of the above notion restricted to the `$\ul H = (B)$ is empty'"/case; see \remref{comparison with Mesiti's notion of density} below.
  \end{remark}
  
  \begin{proposition} \label{n > 0}
  	In a unital virtual double category (\defref{augmented virtual equipment}) the `$\ul H = (B)$ is empty'"/case of the universal property above is subsumed by the `$\ul H = (H_1)$ is a singleton'"/case.
  \end{proposition}
  \begin{proof}
  	Let $\cell\chi{I_B}B$ be the cartesian cell defining the horizontal unit $I_B$ of $B$ (\defref{cartesian cells}). The `horizontal unit identities' of Lemma~5.9 of \cite{Koudenburg20} imply that vertically precomposing with $\chi$ gives bijective correspondences between nullary cells $\cell\phi JM$, of the form as in the $n = 0$ case of the universal property above, and nullary cells $\cell\psi{(J, I_B)}M$, as in the $n = 1$ case, as well as between vertical cells $\cell{\phi'}BM$, as in the $n =0$ case above, and $(1, 0)$"/ary cells $\cell{\psi'}{I_B}M$, as in the $n = 1$ case. Let $(\phi, \psi)$ and $(\phi', \psi')$ be pairs of corresponding cells. It follows from the interchange law (Lemma~1.3 of \cite{Koudenburg20}) that, under these correspondences, the $n = 0$ case of the universal property above for $\phi$ and $\phi'$ is equivalent to the $n = 1$ case of the universal property for $\psi$ and $\psi'$.
  \end{proof}
  
  \subsubsection*{$2$"/Category of elements}
  Consider any $2$"/functor $\map gB{\Cat^{\op A}}$ and recall from \exref{lax comma 2-category is a split two-sided 2-fibration} that the nullary restriction $\Cat^{\op A}(\yon, g)$ is the lax comma $2$"/category $\yon \slash g$ (\exref{lax comma 2-category}). In order to prove that $\map\yon A{\Cat^{\op A}}$ is dense in $\LdSpTwoTwoFib$ we need to prove that the nullary cartesian cell defining $\Cat^{\op A}(\yon, g) \iso \yon \slash g$ which, under the correspondence of \cororef{nullary cells description}, corresponds to the lax natural transformation defining $\yon \slash g$, is left Kan.
  
  We start by showing that the lax comma $2$"/category $\yon \slash g$ is a ``$2$"/category of elements'', in the sense of the following construction, which appears in Section~5 of \cite{Gray69}. Restricting it to $A = 1$, the terminal $2$"/category, recovers the ``$2$"/categorical Grothendieck construction'' described in Section~5 of \cite{Street76}.
  \begin{construction}[Gray; $2$"/category of elements] \label{2-category of elements}
  	Let $A$ and $B$ be $2$"/categories and $\map gB{\Cat^{\op A}}$ a $2$"/functor. Given a morphism $\map v{b_0}{b_1}$ in $B$ we denote all $1$"/functor components $\map{(gv)_a}{(gb_0)(a)}{(gb_1)(a)}$ of the natural transformation $gv$ by $v_! \dfn (gv)_a$. Similarly, the natural transformation components $\cell{(g\beta)_a}{(gv_0)_a}{(gv_1)_a}$, of the $g$"/image of a cell $\cell\beta{v_0}{v_1}$ of $B$, are denoted by $\cell{\beta_! \dfn (g\beta)_a}{v_{0!}}{v_{1!}}$. Likewise, given a morphism $\map u{a_0}{a_1}$ of $A$, we denote \mbox{$\map{u^* \dfn (gb)(u)}{(gb)(a_1)}{(gb)(a_0)}$} for any $b \in B$ and, given a cell $\cell\alpha{u_0}{u_1}$ of $A$, we denote $\cell{\alpha^* \dfn (gb)(\alpha)}{u_0^*}{u_1^*}$.
  	
  	The \emph{$2$"/category of elements} $g^\yon$ is the $2$"/category whose
  	\begin{enumerate}[label=-]
  		\item objects are triples $(a, i, b)$ with $a \in A$, $b \in B$ and $i \in (gb)(a)$;
  		\item morphisms $(a_0, i_0, b_0) \to (a_1, i_1, b_1)$ are triples $(a_0 \xrar ua_1, v_!i_0 \xrar s u^*i_1, b_0 \xrar v b_1)$ with $s \in (gb_1)(a_0)$;
  		\item cells $(u_0, s_0, v_0) \Rar (u_1, s_1, v_1)$ are pairs $(\alpha, \beta)$ of cells $\cell\alpha{u_0}{u_1}$ in $A$ and $\cell\beta{v_0}{v_1}$ in $B$ that make the diagram below commute.
  	\end{enumerate}
  	\begin{displaymath}
  		\begin{tikzpicture}
  			\matrix(m)[math35]{v_{0!}i_0 & u_0^* i_1 \\ v_{1!} i_0 & u_1^*i_1 \\ };
  			\path[map]	(m-1-1) edge node[above] {$s_0$} (m-1-2)
  													edge node[left] {$(\beta_!)_{i_0}$} (m-2-1)
  									(m-1-2) edge node[right] {$(\alpha^*)_{i_1}$} (m-2-2)
  									(m-2-1) edge node[below] {$s_1$} (m-2-2);
  		\end{tikzpicture}
  	\end{displaymath}
  	The composition of $\map{(u', s', v')}{(a_1, i_1, b_1)}{(a_2, i_2, b_2)}$ with the morphism $(u, s, v)$ above is defined as
  	\begin{displaymath}
  		(u' \of u, \; v'_! v_!i_0 \xrar{v'_!s} v'_!u^*i_1 = u^*v'_!i_1 \xrar{u^*s'} u^* u'^* i_2, \; v' \of v),
  	\end{displaymath}
  	while composition of cells in $g^\yon$ is induced by that of $A$ and $B$.
  	
  	It is straightforward to check that the span of projections $A \xlar{\pi_A} g^\yon \xrar{\pi_B} B$ forms a locally discrete split two"/sided $2$"/fibration (\defref{locally discrete split two-sided 2-fibration}). Indeed, the cartesian lift of the morphism $a_0 \xrar u a_1 = \pi_A(a_1, i, b)$ of $A$ is the morphism
  	\begin{displaymath}
  		\lambda\bigpars{u, (a_1, i, b)} \dfn \bigbrks{u^*(a_1, i, b) \dfn (a_0, u^*i, b) \xrar{(u, \id_{u^*i}, \id_b)} (a_1, i, b)}
  	\end{displaymath}
  	in $g^\yon$, while the opcartesian lift of the cell $\pi_A(u_0, s, v) = u_0 \xRar\alpha u_1$ of $A$ is the cell
  	\begin{displaymath}
  		\lambda\bigpars{(u_0, s, v), \alpha} \dfn \bigbrks{(u_0, s, v) \xRar{(\alpha, \id_v)} (u_1, (\alpha^*)_{i_1} \of s, v) \nfd \alpha_!(u_0, s, v)}
  	\end{displaymath}
  	in $g^\yon$; the cleavage for $\pi_B$ is defined coop"/dually.
  \end{construction}
	
	In showing that the $2$"/category of elements $g^\yon$ is the lax comma $2$"/category of $\map\yon A{\Cat^{\op A}}$ and $\map gB{\Cat^{\op A}}$ (\exref{lax comma 2-category}) we use the $2$"/categorical Yoneda lemma, which we recall here; see e.g.\ Section~2.4 of \cite{Kelly82}, or Section~6.9 of \cite{Hedman17} for a detailed proof.
	
	\begin{construction}[Evaluation and its inverse] \label{evaluation}
		Consider an object $a$ of a locally small $2$"/category $A$ and a $2$"/functor $\map f{\op A}\Cat$. The \emph{evaluation} $1$"/functor \mbox{$\map{\yev}{\Cat^{\op A}(\yon a, f)}{fa}$}, where $\Cat^{\op A}(\yon a, f)$ is the $1$"/category of $2$"/natural transformations $\cell\phi{\yon a}f$ and their modifications (\defref{modification}), is defined as follows. It maps natural transformations $\phi$ to the objects $\yev(\phi) \dfn \phi_a(\id_a)$, where $\map{\phi_a}{A(a, a)}{fa}$ is the $1$"/functor component at $a$. Similarly, each modification $M\colon \phi_0 \Rrightarrow \phi_1$ is mapped to the morphism $\map{\yev(M) \dfn (M_a)_{\id_a}}{\yev(\phi_0)}{\yev(\phi_1)}$, where $\cell{M_a}{(\phi_0)_a}{(\phi_1)_a}$ is the natural transformation component at $a$.
		
		Analogously to the notation used in \conref{2-category of elements} we denote by $\map{u^* \dfn fu}{fa_1}{fa_0}$ the $1$"/functor $f$"/image of a morphism $\map u{a_0}{a_1}$ of $A$, and by $\cell{\alpha^* \dfn f\alpha}{u_0^*}{u_1^*}$ the natural transformation $f$"/image of a cell $\cell\alpha{u_0}{u_1}$ of $A$. These assignments combine into the $1$"/functors $\map{\dash^*}{fa}{\Cat^{\op A}(\yon a, f)}$ that send $i \in fa$ to the $2$"/natural transformation $\cell{\dash^*i}{\yon a}f$ with $1$"/functor components $\map{(\dash^* i)_{a'}}{A(a', a)}{fa'}$ defined by $(\dash^*i)_{a'}(u) \dfn u^*i$ and $(\dash^*i)_{a'}(\alpha) \dfn (\alpha^*)_i$. Morphisms $\map sij$ of $fa$ are sent to the modifications $\dash^* s\colon \dash^* i \Rrightarrow \dash^* j$, each of whose natural transformation components $\cell{(\dash^* s)_{a'}}{(\dash^* i)_{a'}}{(\dash^* j)_{a'}}$ has component morphisms $\map{((\dash^*s)_{a'})_u \dfn u^*s}{u^*i}{u^*j}$ in $fa'$.
	\end{construction}
	
	\begin{lemma}[Yoneda lemma] \label{Yoneda lemma}
		The $1$"/functors
		\begin{displaymath}
			\begin{tikzpicture}
				\matrix(m)[math35, column sep=2.5em]{fa & \Cat^{\op A}(\yon a, f) \\ };
				\path[map, transform canvas={yshift=2pt}]	(m-1-1) edge node[above] {$\dash^*$} (m-1-2);
				\path[map, transform canvas={yshift=-2pt}]	(m-1-2) edge node[below] {$\yev$} (m-1-1);
			\end{tikzpicture}
		\end{displaymath}
		described above form isomorphisms $fa \iso \Cat^{\op A}(\yon a, f)$ of $1$"/categories that are $2$"/natural in $a$ and $f$.
	\end{lemma}
	
	Recall from \exref{lax comma 2-category is a split two-sided 2-fibration} that nullary restrictions $C(f, g)$ in $\SpTwoTwoFib$ (\defsref{cartesian cells}{split two-sided 2-fibration}) are constructed as the lax comma $2$"/categories $f \slash g$ (\defref{lax comma 2-category}). The restriction of the result below to $A = 1$, the terminal $2$"/category, is stated in Section~5 of \cite{Street76} and as Proposition~2.18 of \cite{Lambert24}.
	\begin{proposition} \label{2-category of elements is a lax comma 2-category}
		Let $A$ be a locally small $2$"/category and let $\map gB{\Cat^{\op A}}$ be a $2$"/functor. The $2$"/category of elements $g^\yon$ (\conref{2-category of elements}) forms the lax comma $2$"/category $\yon \slash g$ (\exref{lax comma 2-category}) of the Yoneda $2$"/functor $\map\yon A{\Cat^{\op A}}$ (\conref{Yoneda 2-functor}) and $g$.
		
		Under the correspondence of \cororef{nullary cells description} the nullary cartesian cell $\cell\chi{g^\yon}{\Cat^{\op A}}$, defining $g^\yon$ as the nullary restriction $\Cat^{\op A}(\yon, g)$ in $\SpTwoTwoFib$, corresponds to the marked lax natural transformation (\defref{marked lax natural transformation}) below left, whose components and lax naturality cells are defined as (take $f \dfn gb$ in \conref{evaluation})
		\begin{displaymath}
			\cell{\chi_{(a, i, b)} \dfn \dash^* i}{\yon a}{gb} \qquad \quad \text{and} \qquad \quad \chi_{(u, s, v)} \dfn \dash^* s\colon v_!\dash^* i_0 \Rrightarrow (u \of \dash)^*i_1;
		\end{displaymath}
		here $\map s{v_!i_0}{u^*i_1}$ and $\chi_{(u, s, v)}$ is of the form below right.
		\begin{displaymath}
			\begin{tikzpicture}[baseline]
				\matrix(m)[math35]{g^\yon & A \\ B & \Cat^{\op A} \\};
				\path[map]	(m-1-1) edge node[above] {$\pi_A$} (m-1-2)
														edge node[left] {$\pi_B$} (m-2-1)
										(m-1-2) edge node[right] {$\yon$} (m-2-2)
										(m-2-1) edge node[below] {$g$} (m-2-2);
				\path				(m-1-2) edge[cell, shorten >= 9pt, shorten <= 9pt] node[below right] {$\chi$} (m-2-1);
			\end{tikzpicture} \qquad \qquad \qquad \begin{tikzpicture}[baseline]
				\matrix(m)[math35, column sep={4.2em,between origins}]{\yon a_0 & gb_0 \\ \yon a_1 & gb_1 \\};
				\path[map]	(m-1-1) edge node[above] {$\dash^*i_0$} (m-1-2)
														edge node[left] {$\yon u \dfn (u \of \dash)$} (m-2-1)
										(m-1-2) edge node[right] {$v_! \dfn gv$} (m-2-2)
										(m-2-1) edge node[below] {$\dash^*i_1$} (m-2-2);
				\path[transform canvas={xshift=-1.4em}]	(m-1-2) edge[cell, shorten >= 11pt, shorten <= 9pt] node[below right, inner sep=2pt] {$\chi_{(u, s, v)}$} (m-2-1);
			\end{tikzpicture}
		\end{displaymath}
	\end{proposition}
	\begin{proof}
		Consider $A \xlar{\pi_A} \yon \slash g \xrar{\pi_B} B$ as constructed in \exref{lax comma 2-category}. We will define an invertible horizontal $(1, 1)$"/ary cell $\cell\phi{g^\yon}{\yon \slash g}$ in $\SpTwoTwoFib$ that, when composed with the cartesian cell defining $\yon \slash g$, results in the marked lax natural transformation $\chi$ described above. Define the underlying $2$"/functor $\map\phi{g^\yon}{\yon \slash g}$ as follows, with $\dash^* i$ and $\dash^* s$ as in \conref{evaluation}:
		\begin{enumerate}[label=-]
			\item on objects $\phi(a, i, b) \dfn (a, \yon a \xRar{\dash^* i} gb, b)$;
			\item on morphisms $(a_0 \xrar u a_1, v_!i_0 \xrar s u^*i_1, b_0 \xrar v b_1)$, using that $\cell{v_! \dfn gv}{gb_0}{gb_1}$ and $\cell{\yon u \dfn u \of \dash}{\yon a_0}{\yon a_1}$:
				\begin{displaymath}
					\phi(u, s, v) \dfn (u, v_! \dash^*i_0 = \dash^* v_!i_0 \xRrightarrow{\dash^* s} \dash^*u^*i_1 = (u \of \dash)^*i_1, v)
				\end{displaymath}
			\item on cells $(u_0 \xRar\alpha u_1, v_0, \xRar\beta v_1)$, satisfying $(\alpha^*)_{i_1} \of s_0 = s_1 \of (\beta_!)_{i_0}$, $\phi$ acts as the identity: $\phi(\alpha, \beta) \dfn (\alpha, \beta)$.
		\end{enumerate}
		
		It is straightforward to check that so defined $\phi$ is a $2$"/functor over $A$ and $B$ that preserves the cleavages of $g^\yon$ and $\yon \slash g$, and thus forms a horizontal $(1, 1)$"/ary cell \mbox{$\cell\phi{g^\yon}{\yon \slash g}$} in $\SpTwoTwoFib$. It follows from the Yoneda lemma above that $\phi$ as a $2$"/functor restricts to a bijection on objects and is locally invertible, that is, all hom"/functors $g^\yon\bigpars{(a_0, i_0, b_0), (a_1, i_1, b_1)} \to \yon \slash g \bigpars{(a_0, \dash^*i_0, b_0), (a_1, \dash^*i_1, b_1)}$ are invertible, so that $\phi$ is an invertible $2$"/functor. As a consequence, $\phi$ as a cell in $\SpTwoTwoFib$ is invertible; hence its composition with the cartesian cell defining \mbox{$\yon \slash g$} is again cartesian. That this composition corresponds to the marked lax natural transformation $\chi$ of the statement is easily checked.
	\end{proof}
	
	\subsubsection*{The Yoneda $2$"/functor is dense}
	We are now ready to prove that the Yoneda $2$"/functor is dense.
	\begin{theorem} \label{Yoneda 2-functor is dense}
		Let $A$ be a locally small $2$"/category. The Yoneda $2$"/functor \mbox{$\map\yon A{\Cat^{\op A}}$} (\conref{Yoneda 2-functor}) is dense (\defref{density}) in the unital virtual equipment $\LdSpTwoTwoFib$ of locally discrete split two"/sided $2$"/fibrations (\defref{locally discrete split two-sided 2-fibration}).
	\end{theorem}
	\begin{proof}
		The nullary restriction $\Cat^{\op A}(\yon, g)$ exists in $\LdSpTwoTwoFib$ for any $2$"/functor $\map g B{\Cat^{\op A}}$: indeed, it is the $2$"/category of elements $g^\yon$ of \conref{2-category of elements}, defined as such by the nullary cartesian cell $\chi$ in the right"/hand side below, as defined in the previous proposition. Thus, according to \defref{density}, we need to prove that for any such $2$"/functor $g$ and any iterated span \mbox{$\ul J = (B \xlar{p_1} J_1 \xrar{q_1} \dotsb \xlar{p_n} J_n \xrar{q_n} B_n)$} of locally discrete split two"/sided $2$"/fibrations, any marked lax natural transformation $\phi$ on the left"/hand side below factors uniquely through $\chi$ as a lax natural transformation $\phi'$, as shown. Recall that $\phi$ here is a transformation $\yon \of \pi_A \of \pi_{g^\yon} \Rar h \of q_n \of \pi_{J_n}$ of $2$"/functors $g^\yon \times_B \prod_{\dash} \ul J \to \Cat^{\op A}$; see \defref{marked lax natural transformation} and \cororef{nullary cells description}. Since $\LdSpTwoTwoFib$ is unital (\propref{LdSpTwoTwoFib is a unital virtual equipment}), it suffices to restrict to the case $n > 0$ by \propref{n > 0}. The proof consists of three steps: (1) the equation below uniquely determines $\phi'$ and, thus defined, $\phi'$ satisfies the equation, (2) $\phi'$ is a well"/defined lax natural transformation and (3) $\phi'$ is marked.
		\begin{displaymath}
			\begin{tikzpicture}[textbaseline]
				\matrix(m)[math35]{A & B & B_1 & B_{n'} & B_n \\ & & \Cat^{\op A} & & \\};
				\path[map]	(m-1-1) edge[barred] node[above] {$g^\yon$} (m-1-2)
														edge[transform canvas={yshift=-2pt}] node[below left] {$\yon$} (m-2-3)
										(m-1-2) edge[barred] node[above] {$J_1$} (m-1-3)
										(m-1-4) edge[barred] node[above] {$J_n$} (m-1-5)
										(m-1-5) edge[transform canvas={yshift=-2pt}] node[below right] {$h$} (m-2-3);
				\path				(m-1-3) edge[cell] node[right] {$\phi$} (m-2-3);
				\draw				($(m-1-3)!0.5!(m-1-4)$) node {$\dotsb$};
			\end{tikzpicture} \quad = \quad \begin{tikzpicture}[textbaseline]
				\matrix(m)[math35]{A & B & B_1 & B_{n'} & B_n \\ & \Cat^{\op A} & & & \\};
				\path[map]	(m-1-1) edge[barred] node[above] {$g^\yon$} (m-1-2)
														edge[transform canvas={yshift=-2pt}, ps] node[below left] {$\yon$} (m-2-2)
										(m-1-2) edge[barred] node[above] {$J_1$} (m-1-3)
														edge[ps] node[right] {$g$} (m-2-2)
										(m-1-4) edge[barred] node[above] {$J_n$} (m-1-5)
										(m-1-5) edge[transform canvas={yshift=-2pt}] node[below right] {$h$} (m-2-2);
				\path				(m-1-2) edge[cell, transform canvas={shift={(3em,0.333em)}}] node[right] {$\phi'$} (m-2-2)
										(m-1-2) edge[cell, transform canvas={shift={(-1.2em,0.333em)}}] node[right] {$\chi$} (m-2-2);
				\draw				($(m-1-3)!0.5!(m-1-4)$) node {$\dotsb$};
			\end{tikzpicture}
  	\end{displaymath}
  	
  	\emph{Step 1: $\phi = \chi \hc \phi'$ uniquely determines $\phi'$.} We claim that a marked lax natural transformation $\cell{\phi'}{g \of p_1 \of \pi_{J_1}}{h \of q_n \of \pi_{J_n}}$ of $2$"/functors $\prod_{\dash} \ul J \to \Cat^{\op A}$ (see \cororef{nullary cells description}) satisfies the equation $\phi = \chi \hc \phi'$ above if and only if the following two conditions hold. Firstly, for each $\ul i = (i_1, \dotsc, i_n) \in \prod_{\dash} \ul J$ the component morphism $\phi'_{\ul i}$ (\defref{lax natural transformation definition}), which is a $2$"/natural transformation $\cell{\phi'_{\ul i}}{gp_1i_1}{hq_ni_n}$ with $gp_1i_1$, $\map{hq_ni_n}{\op A}{\Cat}$, consists of $1$"/functor components $\map{(\phi'_{\ul i})_a}{(gp_1i_1)(a)}{(hq_ni_n)(a)}$ that are defined as in (a) and (b) below (here $\yev$ is the evaluation $1$"/functor of \conref{evaluation}). Secondly, for each $\map{\ul s}{\ul i}{\ul j}$ in $\prod_{\dash} \ul J$, the lax naturality cell $\phi'_{\ul s}$ (\defref{lax natural transformation definition}) is the modification (\defref{modification}) $\phi'_{\ul s}\colon hq_ns_n \of \phi'_{\ul i} \Rrightarrow \phi'_{\ul j} \of gp_1s_1$ whose component cells are the $1$"/natural transformations $(\phi'_{\ul s})_a$ whose component morphisms in $(hq_nj_n)(a)$ are defined as in (c) below.
  	\begin{enumerate} [label=(\alph*)]
  		\item objects $i \in (gp_1i_1)(a)$ are sent to $(\phi'_{\ul i})_a(i) \dfn \yev(\phi_{(a, i, p_1i_1) \conc \ul i})$, where \mbox{$(a, i, p_1i_1) \in g^\yon$} and $\cell{\phi_{(a, i, p_1i_1) \conc \ul i}}{\yon a}{hq_ni_n}$;
  		\item morphisms $\map sij$ in $(gp_1i_1)(a)$ are sent to $(\phi'_{\ul i})_a(s) \dfn \yev(\phi_{(\id_a, s, \id_{p_1i_1})\conc \id_{\ul i}})$, where $\map{(\id_a, s, \id_{p_1i_1})}{(a, i, p_1i_1)}{(a, j, p_1i_1)}$ is in $g^\yon$ and where $\phi_{(\id_a, s, \id_{p_1i_1}) \conc \id_{\ul i}}\colon \phi_{(a, i, p_1i_1) \conc \ul i} \Rrightarrow \phi_{(a, j, p_1i_1) \conc \ul i}$ is a lax naturality cell of $\phi$ (which is a modification);
  		\item at $i \in (gp_1i_1)(a)$ the component morphism $((\phi'_{\ul s})_a)_i$ is
  		\begin{align*}
  			(hq_ns_n \of \phi'_{\ul i})_a(i) &= \yev(hq_ns_n \of \phi_{(a, i, p_1i_1) \conc \ul i}) \\ 
  			&\xrar{\yev({\phi_{\rho_{\pi_B}((a, i, p_1i_1), p_1s_1) \conc \ul s}})} \yev(\phi_{(a, (p_1s_1)_!i, p_1j_1) \conc \ul j})\\
  			&= (\phi'_{\ul j})_a((p_1s_1)_!i) = (\phi'_{\ul j} \of gp_1s_1)_a(i),
  		\end{align*}
  		where the morphism is the evaluation of the lax naturality cell of $\phi$ at the concatenation of the opcartesian lift of $\pi_B(a, i, p_1i_1) = p_1i_1 \xrar{p_1s_1} p_1j_1$ in $g^\yon$ (\conref{2-category of elements}) and $\ul s \in \prod_{\dash} \ul J$.
  	\end{enumerate}
  	
  	To prove the if"/part of the claim assume that $\phi'$ is defined as above. To show that the component morphisms of $\chi \hc \phi'$ and $\phi$ coincide at the object $(a, i, p_1i_1) \conc \ul i$ of $g^\yon \times_B \prod_{\dash} \ul J$, by the Yoneda lemma (\lemref{Yoneda lemma}) it suffices to show that their images under evaluation coincide, which they do, by (a) above:
  	\begin{align*}
  		\yev \bigpars{\chi \hc \phi')_{(a, i, p_1i_1) \conc \ul i}} &= \yev \bigbrks{\yon a \xRar{\dash^*  i} gp_1i_1 \xRar{\phi'_{\ul i}} hq_ni_n} \\
  		&= (\phi'_{\ul i})_a(\id_a^* i) = (\phi'_{\ul i})_a(i) = \yev(\phi_{(a, i, p_1i_1) \conc \ul i}).
  	\end{align*}
  	That the lax naturality cells of $\chi \hc \phi'$ and $\phi$ coincide, too, that is, their component cells coincide at each morphism $\map{(u, s, p_1s_1) \conc \ul s}{(a_0, i, p_1i_1) \conc \ul i}{(a_1, j, p_1j_1) \conc \ul j}$ in $g^\yon \times_B \prod_{\dash} \ul J$, with $\map s{(p_1s_1)_!i}{u^*j}$ in $(gp_1j_1)(a_0)$, similarly follows from
  	\begin{align*}
  		\yev\bigpars{(\chi \hc \phi'&)_{(u, s, p_1s_1) \conc \ul s}} \overset{(1)}= \yev\bigpars{(\phi'_{\ul s} \of \chi_{(a_0, i, p_1i_1)}) \hc (\phi'_{\ul j} \of \chi_{(u, s, p_1s_1)})} \\
  		&\overset{(2)}= \yev\bigpars{(\phi'_{\ul s} \of \dash^* i) \hc (\phi'_{\ul j} \of \dash^* s)} = (\phi'_{\ul j})_{a_0}(s) \of ((\phi'_{\ul s})_{a_0})_i \\
  		&\overset{(3)}= \yev(\phi_{(\id_{a_0}, s, \id_{p_1j_1}) \conc \id_{\ul j}}) \of \yev({\phi_{\rho_{\pi_B}((a_0, i, p_1i_1), p_1s_1) \conc \ul s}}) \\
  		&\overset{(4)}= \yev(\phi_{\rho_{\pi_B}((a_0, i, p_1i_1), p_1s_1) \conc \ul s} \hc \phi_{(\id_{a_0}, s, \id_{p_1j_1}) \conc \id_{\ul j}} \hc \phi_{\lambda_{\pi_A}(u, (a_1, j, p_1j_1)) \conc \id_{\ul j}}) \\
  		&\overset{(5)}= \yev(\phi_{(u, s, p_1s_1) \conc \ul s}).
  	\end{align*}
  	Here the labelled identities follow from (1) the definition of vertical composition of lax natural transformations (see e.g.\ Definition~4.2.15 of \cite{Johnson-Yau21}); (2) the definition of $\chi$ (\propref{2-category of elements is a lax comma 2-category}); (3) assumptions (b) and (c) above; (4) the functoriality of $\yev$ and the fact that $\phi$ is marked, so that $\phi_{\lambda_{\pi_A}(u, (a_1, j, p_1j_1))\conc \id_{\ul j}}$ is an identity modification by \cororef{nullary cells description}(cm); (5) the axioms satisfied by the lax naturality cells of $\phi$ (\defref{lax natural transformation definition}) together with the definition of (op)cartesian morphism lifts and that of morphism composition in $g^\yon$ (\conref{2-category of elements}).
  	
  	The only"/if"/part follows similarly: if $\chi \hc \phi' = \phi$ then, reusing instances of parts of the equalities above as well as that $\rho_{\pi_B}((a, i, p_1i_1), p_1s_1) = (\id_a, \id_{(p_1s_1)_!i}, p_1s_1)$ in $g^\yon$ (\conref{2-category of elements}), we have
  	\begin{align*}
  		(\phi'_{\ul i})_a(i) &= \yev\bigpars{(\chi \hc \phi')_{(a, i, p_1i_1) \conc \ul i}} = \yev(\phi_{(a, i, p_1i_1) \conc \ul i}); \\
  		(\phi'_{\ul i})_a(s) &= (\phi'_{\ul i})_a(s) \of ((\phi'_{\id_{\ul i}})_a)_i =  \yev\bigpars{(\chi \hc \phi')_{(\id_a, s, \id_{p_1i_1}) \conc \id_{\ul i}}} = \yev(\phi_{(\id_a, s, \id_{p_1i_1}) \conc \id_{\ul i}}); \\
  		((\phi'_{\ul s})_a)_i &= (\phi'_{\ul j})_a(\id_{(p_1s_1)_!i}) \of ((\phi'_{\ul s})_a)_i = \yev\bigpars{(\chi \hc \phi')_{(\id_a, \id_{(p_1s_1)_!i}, p_1s_1) \conc \ul s}} \\
  		&=\yev({\phi_{\rho_{\pi_B}((a, i, p_1i_1), p_1s_1) \conc \ul s}}).
  	\end{align*}
  	
  	\emph{Step 2: $\phi'$ is well"/defined.} Next, we prove that the assignments directly above make $\phi'$ into a well"/defined lax natural transformation \mbox{$g \of p_1 \of \pi_{J_1} \Rar h \of q_n \of \pi_{J_n}$} of $2$"/functors $\prod_{\dash} \ul J \to \Cat^{\op A}$ (\cororef{nullary cells description}). That $s \mapsto (\phi'_{\ul i})_a(s)$ is functorial, that is, $\map{(\phi'_{\ul i})_a}{(gp_1i_1)(a)}{(hq_ni_n)(a)}$ is a well"/defined $1$"/functor, follows immediately from the axioms for the lax naturality cells of $\phi$ (\defref{lax natural transformation definition}).
  	
  	To show that $a \mapsto (\phi'_{\ul i})_a$ is $1$"/natural with respect to $\map u{a_0}{a_1}$ means proving the equation $(\phi'_{\ul i})_{a_0} \of (gp_1i_1)(u) = (hq_ni_n)(u) \of (\phi'_{\ul i})_{a_1}$ of $1$"/functors \mbox{$(gp_1i_1)(a_1) \to (hq_ni_n)(a_0)$}. That the equation holds at any morphism $\map sij$ in $(gp_1i_1)(a_1)$ is shown below, where the labelled identities follow from (1) \cororef{nullary cells description}(cm); (2) the definitions of cartesian lifts and morphism composition in $g^\yon$ (\conref{2-category of elements}) as well as the lax naturality axioms (\defref{lax natural transformation definition}); (3)~the naturality of the components of the modification $\phi_{(\id_{a_1}, s, p_1i_1) \conc \ul i}\colon \yon a_1 \Rrightarrow hq_ni_n$.
  	\begin{align*}
  		(\phi'_{\ul i})_{a_0} \of (gp_1i_1&)(u)(s) = (\phi'_{\ul i})_{a_0}(u^*s) = \yev(\phi_{(\id_{a_0}, u^*s, \id_{p_1i_1}) \conc \id_{\ul i}}) \\
  		&\overset{(1)}= \yev(\phi_{(\id_{a_0}, u^*s, \id_{p_1i_1}) \conc \id_{\ul i}} \hc \phi_{\lambda_{\pi_A}(u, (a_1, j, p_1i_1)) \conc \id_{\ul i}}) \\
  		&\overset{(2)}= \yev\bigpars{\phi_{\lambda_{\pi_A}(u, (a_1, i, p_1i_1)) \conc \id_{\ul i}} \hc (\phi_{(\id_{a_1}, s, \id_{p_1i_1}) \conc \id_{\ul i}} \of \yon u)} \\
  		&\overset{(1)}= \yev(\phi_{(\id_{a_1}, s, \id_{p_1i_1}) \conc \id_{\ul i}} \of \yon u) = \bigpars{(\phi_{(\id_{a_1}, s, \id_{p_1i_1}) \conc \id_{\ul i}})_{a_0}}_u \\
  		&= \bigpars{(\phi_{(\id_{a_1}, s, \id_{p_1i_1}) \conc \id_{\ul i}})_{a_0} \of (\yon a_1)(u)}_{\id_{a_1}} \\
  		&\overset{(3)}= \bigpars{(hq_ni_n)(u) \of (\phi_{(\id_{a_1}, s, \id_{p_1i_1}) \conc \id_{\ul i}})_{a_1}}_{\id_{a_1}} = (hq_ni_n)(u) \of (\phi'_{\ul i})_{a_1}(s)
  	\end{align*}
  	
  	The author leaves to the reader the proofs of the $2$"/naturality of $a \mapsto (\phi'_{\ul i})_a$, with respect to cells $\cell\alpha{u_0}{u_1}$ in $A$, and the $1$"/naturality of $i \mapsto ((\phi'_{\ul s})_a)_i$, with respect to morphisms $\map sij$ in $(gp_1i_1)(a)$; their arguments are similar to the one directly above. We conclude that each $\cell{\phi'_{\ul i}}{gp_1i_1}{hq_ni_n}$ is a well"/defined $2$"/natural transformation and that each $\cell{(\phi'_{\ul s})_a}{(hq_ns_n \of \phi'_{\ul i})_a}{(\phi'_{\ul j} \of gp_1s_1)_a}$ is a well"/defined $1$"/natural transformation. Naturality of the modification $\phi'_{\ul s}$ means that the equation $(\phi'_{\ul s})_{a_0} \of (gp_1i_1)(u) = (hq_ni_n)(u) \of (\phi'_{\ul s})_{a_1}$ of $1$"/natural transformations holds for any morphism $\map u{a_0}{a_1}$ in $A$. That their components at $i \in (gp_1i_1)(a_1)$ coincide is shown below, where (1) follows from \cororef{nullary cells description}(cm); (2) follows from the lax naturality axioms and \thmref{split two-sided 2-fibration description}(cml); and (3) combines several identities, analogous to the last five identities of the computation directly above.
  	\begin{align*}
  		\bigpars{(\phi'_{\ul s})_{a_0} \of (g&p_1i_1)(u)}_i = ((\phi'_{\ul s})_{a_0})_{u^* i} = \yev(\phi_{\rho_{\pi_B}((a_0, u^*i, p_1i_1), p_1s_1) \conc \ul s}) \\
  		&\overset{(1)}= \yev(\phi_{\rho_{\pi_B}((a_0, u^*i, p_1i_1), p_1s_1) \conc \ul s} \hc \phi_{\lambda_{\pi_A}(u, (a_1, (p_1s_1)_!i, p_1j_1)) \conc \id_{\ul j}}) \\
  		&\overset{(2)}= \yev\bigpars{(hq_ns_n \of \phi_{\lambda_{\pi_A}(u, (a_1, i, p_1i_1)) \conc \id_{\ul i}}) \hc (\phi_{\rho_{\pi_B}((a_1, i, p_1i_1), p_1s_1) \conc \ul s} \of \yon u)} \\
  		&\overset{(3)}= \bigpars{(hq_ni_n)(u) \of (\phi'_{\ul s})_{a_1}}_i
  	\end{align*}
  	
  	It remains to prove the three axioms required for the lax naturality cells $\phi'_{\ul s}$; see e.g.\ Section~I,2.4 of \cite{Gray74} or Definition~4.2.1 of \cite{Johnson-Yau21}. That \mbox{$\ul s \mapsto \phi'_{\ul s}$} is natural with respect to cells $\map{\ul\sigma}{\ul s}{\ul t}$ in $\prod_{\dash} \ul J$ follows from the equality below.
  	\begin{align*}
  		\bigpars{\bigpars{(hq_n\sigma_n \of \phi'_{\ul i}) \hc \phi'_{\ul t}}_a}_i &= \yev\bigpars{(hq_n\sigma_n \of \phi_{(a, i, p_1i_1) \conc \ul i}) \hc \phi_{\rho_{\pi_B}((a, i, p_1i_1), p_1t_1) \conc \ul t)}} \\
  		&= \yev\bigpars{\phi_{\rho_{\pi_B}((a, i, p_1i_1), p_1s_1) \conc \ul s} \hc \phi_{(\id_a, (p_1\sigma_1)_!i, \id_{p_1j_1}) \conc \id_{\ul j})}} \\
  		&= (\phi'_{\ul j})_a((p_1\sigma_1)_{!i}) \of ((\phi'_{\ul s})_a)_i = \bigpars{\bigpars{\phi'_{\ul s} \hc (\phi'_{\ul j} \of gp_1\sigma_1)}_a}_i
  	\end{align*}
  	Here the middle identity follows from the naturality of the lax naturality cells of $\phi$ with respect to the cell $(\id_{\id_a}, p_1\sigma_1) \conc \ul\sigma \in g^\yon \times_B \prod_{\dash} \ul J$, where
  	\begin{displaymath}
  		\cell{(\id_{\id_a}, p_1\sigma_1)}{(\id_a, (p_1\sigma_1)_{!i}, \id_{p_1j_1}) \of \rho_{\pi_B}((a, i, p_1i_1), p_1s_1)}{\rho_{\pi_B}((a, i, p_1i_1), p_1t_1)}
  	\end{displaymath}
  	is a cell in $g^\yon$ (\conref{2-category of elements}).
  	
  	That $\ul s \mapsto \phi'_{\ul s}$ preserves identities, that is, $\phi'_{\id_{\ul i}} = \id_{\phi'_{\ul i}}$, is easily checked. That it preserves the composite of $\map{\ul s}{\ul i}{\ul j}$ and $\map{\ul t}{\ul j}{\ul k}$ is shown below, where the middle identity follows from the fact that the lax naturality cells of $\phi$ do so, together with the splitting equation \defref{split op-2-fibration}(sm$\of$) for $\rho_{\pi_B}$. This completes the proof of $\phi'$ being a well"/defined lax natural transformation.
  	\begin{align*}
  		\bigpars{\bigpars{(hq_nt_n \of \phi'_{\ul s}&) \hc (\phi'_{\ul t} \of gp_1s_1)}_a}_i \\
  		&= \yev\bigpars{(hq_nt_n \of \phi_{\rho_{\pi_B}((a, i, p_1i_1), p_1s_1) \conc \ul s}) \hc \phi_{\rho_{\pi_B}((p_1s_1)_!(a, i, p_1i_1), p_1t_1) \conc \ul t}} \\
  		&= \yev(\phi_{\rho_{\pi_B}((a, i, p_1i_1), p(t_1 \of s_1)) \conc \ul t \of \ul s}) = ((\phi'_{\ul t \of \ul s})_a)_i
  	\end{align*}
  	
  	\emph{Step 3: $\phi'$ is marked.} Finally we show that $\phi'$, as defined in the first step, is marked in the sense of \defref{marked lax natural transformation} and \cororef{nullary cells description}. To show that it satisfies (cm) of the latter corollary, consider $\ul i \in \prod_{\dash} \ul J$ and $\map ub{p_1i_1}$ in $B$ as well as $i \in (gb)(a)$. By definition of $\phi'$
  	\begin{displaymath}
  		\bigpars{\bigpars{\phi'_{(\lambda_{p_1}(u, i_1), \id_{i_2}, \dotsc, \id_{i_n})}}_a}_i = \yev\bigpars{\phi_{(\rho_{\pi_B}((a, i, b), u), \lambda_{p_1}(u, i_1), \id_{i_2}, \dotsc, \id_{i_n})}}
  	\end{displaymath}
  	which is an identity cell, as required, by \cororef{nullary cells description}(im) applied to $\phi$ at $B$. That $\phi'$ satisfies conditions (im) and (om) of the corollary follows easily from the fact that $\phi$ does so, using that $\rho_{\pi_B}((a, i, p_1i_1), \id_{p_1i_1}) = \id_{(a, i, p_1i_1)}$ in $g^\yon$ (\conref{2-category of elements}). This completes the proof.
	\end{proof}
	
	In closing this section, we compare, in the case of the Yoneda $2$"/functor $\map\yon 1\Cat$ for the terminal $2$"/category $A = 1$, our notion of density (\defref{density}) to the notion of density considered in \cite{Mesiti24}.
	\begin{remark} [Partly recovering Mesiti's notion of density] \label{comparison with Mesiti's notion of density}
		Notice that $\map\yon 1\Cat$ simply picks out the terminal $1$"/category $1 \in \Cat$. Theorem~4.11 of \cite{Mesiti24} describes the density of $\yon$ as follows. Given any $2$"/functor $\map gB\Cat$ with $B$ a small $2$"/category, consider the marked lax natural transformation $\chi$ of \propref{2-category of elements is a lax comma 2-category}, that defines the $2$"/category of elements $g^\yon$ as the comma $2$"/category of $\map\yon 1\Cat$ and $g$. Mesiti proves that $\chi$ defines $g$ as the pointwise left Kan extension of the constant $2$"/functor $\map{\Delta_1}{g^\yon}\Cat$ at $1$ along the locally discrete split op"/$2$"/fibration $\map{\pi_B}{g^\yon}B$ (\conref{2-category of elements}), in the sense of Definition~3.12 of op.\ cit. The latter is unpacked in the proof of Theorem~4.11 of the reference as follows: for each object $b \in B$ and $1$"/category $X \in \Cat$, composition with $\chi$ induces an isomorphism of the $1$"/functor"/category $\Cat(gb, X)$ and the category of ``cartesian"/marked'' oplax natural transformations (Definition~2.16 of op.\ cit.) of the form $B(\pi_B, b) \Rar \Cat(\Delta_1, X)\colon \op{(g^\yon)} \to \Cat$.
		
		The bijection"/on"/objects obtained by restricting the latter isomorphism of categories can be recovered from the `$\ul H = (B)$ is empty'"/case of \defref{density} applied to the marked lax natural transformation $\chi$ associated to the $2$"/category of elements $(gb)^\yon$, as in the right"/hand side of the equation below, and which is a left Kan cell in $\LdSpTwoTwoFib$ by \thmref{Yoneda 2-functor is dense}, as follows. Firstly it is straightforward to check that the cartesian"/marked oplax natural transformations in question correspond precisely to marked lax natural transformations (\defref{marked lax natural transformation}) $\cell\phi{\Delta_1}{\map{\Delta_X}{g^\yon \times_B B \slash b}\Cat}$ as below left, where $B \slash b$ is the lax comma $2$"/category of $\id_B$ and $\map b1B$ (\exref{lax comma 2-category}). By \cororef{nullary cells description} such transformations correspond precisely to the nullary cells $\cell\phi{(g^\yon, B \slash b)}\Cat$ of $\LdSpTwoTwoFib$ as on the left"/hand side of the equation below. Using the universal properties of $\chi$ and the ``cocartesian cell'' in the right"/hand side (see Lemma~8.1 of \cite{Koudenburg20}, using that $B \slash b$ is the ``companion'' of $\map b1B$), this equation determines a bijection between such nullary cells $\phi$ and vertical cells (that is, $2$"/natural transformations) $\cell\psi{gb}X$, as in the right"/hand side. We recover the asserted bijection"/on"/objects by noticing that the latter are simply $1$"/functors $gb \to X$.
		\begin{displaymath}
			\begin{tikzpicture}[textbaseline]
				\matrix(m)[math35]{g^\yon \times_B B \slash b \\ \Cat \\};
				\path[map]	(m-1-1) edge[bend right=45] node[left] {$\Delta_1$} (m-2-1)
														edge[bend left=45] node[right] {$\Delta_X$} (m-2-1);
				\path[transform canvas={yshift=-4pt}]	([xshift=-8.5pt]$(m-1-1)!0.5!(m-2-1)$) edge[cell] node[above, xshift=-1pt] {$\phi$} ([xshift=9.5pt]$(m-1-1)!0.5!(m-2-1)$);
			\end{tikzpicture} \qquad \qquad \begin{tikzpicture}[textbaseline]
				\matrix(m)[math35]{ 1 & B & 1 \\ & \Cat & \\ };
					\path[map]	(m-1-1) edge[barred] node[above] {$g^\yon$} (m-1-2)
															edge node[below left] {$\yon$} (m-2-2)
											(m-1-2) edge[barred] node[above] {$B \slash b$} (m-1-3)
											(m-1-3) edge node[below right] {$X$} (m-2-2);
					\path				(m-1-2) edge[cell] node[right] {$\phi$} (m-2-2);
			\end{tikzpicture} = \begin{tikzpicture}[textbaseline]
				\matrix(m)[math35, column sep={1.75em,between origins}]{ 1 & & B & & 1 \\ & 1 & & 1 & \\ & & \Cat & & \\ };
					\path[map]	(m-1-1) edge[barred] node[above] {$g^\yon$} (m-1-3)
											(m-1-3) edge[barred] node[above] {$B \slash b$} (m-1-5)
											(m-2-2) edge[barred] node[above, inner sep=1pt] {$(gb)^\yon$} (m-2-4)
															edge node[left] {$\yon$} (m-3-3)
											(m-2-4) edge[bend right = 15] node[above left, xshift=8pt, yshift=2pt] {$gb$} (m-3-3)
											(m-2-4) edge[bend left = 45] node[right] {$X$} (m-3-3);
					\path				(m-1-1) edge[eq] (m-2-2)
											(m-1-5) edge[eq] (m-2-4);
					\path[transform canvas={yshift=0.25em, xshift=-4pt}]	(m-2-3) edge[cell] node[right, inner sep=2pt] {$\chi$} (m-3-3);
					\path[transform canvas={xshift=-6pt}]				(m-2-4) edge[cell] node[right, inner sep=2pt] {$\psi$} (m-3-4);
					\draw				($(m-1-3)!0.4!(m-2-3)$) node[font=\scriptsize] {$\cocart$};
			\end{tikzpicture}
		\end{displaymath}
	\end{remark}
	
	\section[The Yoneda 2-functor satisfies the Yoneda axiom]{The Yoneda $2$"/functor satisfies the Yoneda axiom} \label{Yoneda axiom section}
	Here we complete the proof of the Yoneda $2$"/functors $\map\yon A{\Cat^{\op A}}$ (\conref{Yoneda 2-functor}) forming formal Yoneda embeddings in the augmented virtual equipment $\SfLdSpTwoTwoFib$ of locally discrete split two"/sided $2$"/fibrations with small fibres (\defsref{locally discrete split two-sided 2-fibration}{small fibres}). Having shown that $\yon$ is dense in \thmref{Yoneda 2-functor is dense}, it remains to show that $\yon$ satisfies the `Yoneda axiom' below; this is Definition~4.5 of \cite{Koudenburg24}.
	\begin{definition} \label{Yoneda embedding}
		Let $A$ be a unital object (\defref{cartesian cells}) in an augmented virtual double category. A dense morphism (\defref{density}) $\map\yon A{\ps A}$ is called a \emph{Yoneda embedding} if it satisfies the \emph{Yoneda axiom}: for every horizontal morphism \mbox{$\hmap JAB$} there exists a vertical morphism $\map{\cur J}B{\ps A}$ equipped with a nullary cartesian cell (\defref{cartesian cells})
		\begin{displaymath}
			\begin{tikzpicture}
				\matrix(m)[math35, column sep={1.75em,between origins}]{A & & B \\ & \ps A. & \\};
				\path[map]	(m-1-1) edge[barred] node[above] {$J$} (m-1-3)
														edge[transform canvas={xshift=-2pt}] node[left] {$\yon$} (m-2-2)
										(m-1-3) edge[transform canvas={xshift=2pt}] node[right] {$\cur J$} (m-2-2);
				\draw				([yshift=0.333em]$(m-1-2)!0.5!(m-2-2)$) node[font=\scriptsize] {$\cart$};
			\end{tikzpicture}
	  \end{displaymath}
	\end{definition}
	
	\subsubsection*{Inverse Grothendieck construction}
	We begin by associating to each locally discrete split two"/sided $2$"/fibration $\hmap JAB$ with small fibres a $2$"/functor $\map{\cur J}B{\Cat^{\op A}}$. In \thmref{Yoneda axiom} below we will show that it comes equipped with a cartesian cell as above. The assignment $\dash \mapsto \cur\dash$ is called the \emph{inverse Grothendieck construction}. Restricting the construction below to $A = 1$, the terminal $2$"/category, recovers the op"/dual of Construction~3.2 of \cite{Lambert24}.
	\begin{construction}[Inverse Grothendieck construction] \label{curry J}
		Let $A$ and $B$ be (possibly large) $2$"/categories and let $J = \brks{A \xlar p J \xrar q B}$ be a locally discrete two"/sided $2$"/fibration with small fibres (\defsref{locally discrete split two-sided 2-fibration}{small fibres}). We will associate to $J$ a $2$"/functor $\map f{\op A \times B}\Cat$ that, under the cartesian closed structure on the $2$"/category $\TwoCat'$ of large $2$"/categories, $2$"/functors and $2$"/natural transformations, corresponds to a $2$"/functor $B \to \Cat^{\op A}$ that we denote by $\cur J$. The proposition below proves that the following assignments define a $2$"/functor $\map f{\op A \times B}\Cat$.
		\begin{enumerate}[label=-]
			\item Send objects $(a, b)$ to the fibres $f(a, b) \dfn J_{(a, b)}$ (\defref{small fibres}).
			\item Send morphisms $(a_0 \xrar u a_1, b_0 \xrar v b_1)$ to the $1$"/functors \mbox{$\map{f(u, v) \dfn v_!u^*}{J_{(a_1, b_0)}}{J_{(a_0, b_1)}}$} given by $f(u, v)(i) = v_!u^* i$ and $f(u, v)(i \xrar s j) = v_!u^* s$, the unique lift making the right square below commute.
			\begin{displaymath}
				\begin{tikzpicture}
					\matrix(m)[math35, column sep={4.5em,between origins}]{ i & u^*i & v_!u^*i \\ j & u^*j & v_!u^* j \\ };
					\path[map]	(m-1-1) edge node[left] {$s$} (m-2-1)
											(m-1-2) edge node[above] {$\lambda(u, i)$} (m-1-1)
															edge[dashed] node[right] {$u^*s$} (m-2-2)
															edge node[above] {$\rho(u^*i, v)$} (m-1-3)
											(m-1-3) edge[dashed] node[right] {$v_!u^* s$} (m-2-3)
											(m-2-2) edge node[below] {$\lambda(u, j)$} (m-2-1)
															edge node[below] {$\rho(u^*j, v)$} (m-2-3);
				\end{tikzpicture}
			\end{displaymath}
			\item Send cells $(u_0 \xRar\alpha u_1, v_0 \xRar\beta v_1)$ to the $1$"/natural transformation $f(\alpha, \beta) \dfn \beta_!\alpha^*$ below left, whose component at $i \in J(a_1, b_0)$ is the unique lift $(\beta_!\alpha^*)_i$ making the right square below right commute. Here $\alpha_!\lambda(u_0, i)$ is the target of the opcartesian lift of $\cell\alpha{p\lambda(u_0, i)}{u_1}$ and $\beta^*\rho(u_1^*i, v_1)$ is the source of the cartesian lift of $\cell\beta{v_0}{q\rho(u_1^*i, v_1)}$ (\defref{split op-2-fibration}(cl) and its coop"/dual).
			\begin{displaymath}
				\begin{tikzpicture}[baseline]
					\matrix(m)[math35]{J_{(a_1, b_0)} \\ J_{(a_0, b_1)} \\};
					\path[map]	(m-1-1) edge[bend right=45] node[left] {$v_{0!}u_0^*$} (m-2-1)
															edge[bend left=45] node[right] {$v_{1!}u_1^*$} (m-2-1);
					\path[transform canvas={yshift=-4pt}]	([xshift=-8.5pt]$(m-1-1)!0.5!(m-2-1)$) edge[cell] node[above, xshift=-1pt] {$\beta_!\alpha^*$} ([xshift=9.5pt]$(m-1-1)!0.5!(m-2-1)$);
				\end{tikzpicture} \qquad\qquad \begin{tikzpicture}[baseline]
					\matrix(m)[math35, column sep={6.5em,between origins}]
						{ i & u_0^*i & v_{0!}u_0^*i \\
								& u_1^*i & v_{1!}u_1^*i \\ };
					\path[map]	(m-1-2) edge node[above] {$\alpha_!\lambda(u_0, i)$} (m-1-1)
															edge[dashed] node[right] {$\alpha^*_i$} (m-2-2)
															edge node[above] {$\rho(u_0^*i, v_0)$} (m-1-3)
											(m-1-3) edge[dashed] node[right] {$(\beta_!\alpha^*)_i$} (m-2-3)
											(m-2-2) edge node[below left] {$\lambda(u_1, i)$} (m-1-1)
															edge node[below] {$\beta^*\rho(u_1^*i, v_1)$} (m-2-3);
				\end{tikzpicture}
			\end{displaymath}
		\end{enumerate}
	\end{construction}
	\begin{proposition}
		The assignments above combine into a well"/defined $2$"/functor $\map f{\op A \times B}\Cat$.
	\end{proposition}
	\begin{proof}
			That the images of $f(u, v)$ and the components of $f(\alpha, \beta)$ above land in the fibre $J_{(a_0, b_1)}$ is a consequence of \defref{split two-sided 2-fibration description}(v). That $f(u, v)$ as defined is $1$"/functorial follows from the uniqueness of lifts. Here we will prove the functoriality of $f$ on composites of cells, while the author leaves the naturality of $f(\alpha, \beta)$ and the $1$"/functoriality of $f$, whose proofs are similar, to the reader.
		
		To see that $f$ preserves vertical composition of cells consider cells \mbox{$u_0 \xRar{\alpha_0} u_1 \xRar{\alpha_1} u_2$} in $A$ and $v_0 \xRar{\beta_0} v_1 \xRar{\beta_1} v_2$ in $B$; we have to prove that
		\begin{displaymath}
			 \bigbrks{v_{0!}u_0^*i \xrar{(\beta_{0!}\alpha_0^*)_i} v_{1!}u_1^*i \xrar{(\beta_{1!}\alpha_1^*)_i} v_{2!}u_2^*i} = \bigpars{(\beta_0 \hc \beta_1)_!(\alpha_0 \hc \alpha_1)^*}_i
		\end{displaymath}
		for all $i \in J_{(a_1, b_0)}$. First notice that $\map{(\alpha_1^*)_i \of (\alpha_0^*)_i = (\alpha_0 \hc \alpha_1)^*_i}{u_0^*i}{u_2^*i}$, which follows from the equality below and the fact that factorisations through cartesian cells are unique. The identities below labelled `d' follow from the definitions of $(\alpha_0^*)_i$, $(\alpha_1^*)_i$ and $(\alpha_0 \hc \alpha_1)^*_i$, while those denoted `s' use the (coop"/duals of the) splitting equations (sc$\hc$) and (sc1) of \defref{split op-2-fibration}.
		\begin{align*}
			\lambda(u_2, i) \of (\alpha_1^*&)_i \of (\alpha_0^*)_i \overset{\textup{d,s}}= \alpha_{1!}\lambda(u_1, i) \of \id_{\id_{a_0}!}(\alpha_0^*)_i \overset{\textup s}= \alpha_{1!}\bigpars{\lambda(u_1, i) \of (\alpha_0^*)_i} \\
			&\overset{\textup d}= \alpha_{1!}\bigpars{\alpha_{0!}\lambda(u_0, i)} \overset{\textup s}= (\alpha_0 \hc \alpha_1)_!\lambda(u_0, i) \overset{\textup d}= \lambda(u_2, i) \of (\alpha_0 \hc \alpha_1)^*_i
		\end{align*}
		The required equality similarly follows from the equality below.
		\begin{align*}
			(\beta_{1!} \alpha_1^*&)_i \of (\beta_{0!}\alpha_0^*)_i \of \rho(u_0^*i, v_0) \overset{\textup d}= (\beta_{1!} \alpha_1^*)_i \of \beta_{0*}\rho(u_1^*i, v_1) \of (\alpha_0^*)_i \\
			&\overset{\textup s}= \beta_0^*\bigpars{(\beta_{1!}\alpha_1^*)_i \of \rho(u_1^*i, v_1)} \of (\alpha_0^*)_i \overset{\textup d}= \beta_0^*\bigpars{\beta_1^*\rho(u_2^*i, v_2) \of (\alpha_1^*)_i} \of (\alpha_0^*)_i \\
			&\overset{\textup s}= (\beta_0 \hc \beta_1)^*\rho(u_2^*i, v_2) \of (\alpha_1^*)_i \of (\alpha_0^*)_i = (\beta_0 \hc \beta_1)^*\rho(u_2^*i, v_2) \of (\alpha_0 \hc \alpha_1)^*_i \\
			&\overset{\textup d}= \bigpars{(\beta_0 \hc \beta_1)_!(\alpha_0 \hc \alpha_1)^*}_i \of \rho(u_0^*i, v_0)
		\end{align*}
		
		To show that $f$ preserves horizontal composition, consider horizontally composable cells $\cell{\alpha_0}{u_0}{u_0^\prime}$ and $\cell{\alpha_1}{u_1}{u_1^\prime}$ in $A$ and horizontally composable cells $\cell{\beta_0}{v_0}{v_0^\prime}$ and $\cell{\beta_1}{v_1}{v_1^\prime}$ in $B$. We have to show that, for any $i \in J_{(a_2, b_0)}$, the composite
		\begin{displaymath}
			v_{1!}u_0^*v_{0!}u_1^*i \xrar{v_{1!}u_0^*(\beta_{0!}\alpha_1^*)_i} v_{1!}u_0^*v^\prime_{0!}u_1^{\prime*}i \xrar{(\beta_{1!}\alpha_0^*)_{v^\prime_{0!}u_1^{\prime*}i}} v^\prime_{1!} u_0^{\prime*}v^\prime_{0!}u_1^{\prime*}i
		\end{displaymath}
		equals $\bigpars{(\beta_1 \of \beta_0)_!(\alpha_1 \of \alpha_0)^*}_i$. That their sources and targets coincide follows from \thmref{split two-sided 2-fibration description}(cml). In proving this equality, we use the following equalities, which we will prove first.
		\begin{enumerate}[label=(\alph*)]
			\item $(\alpha_1^*)_i \of \alpha_{0!}\lambda(u_0, u_1^*i) = \lambda(u^\prime_0, u_1^{\prime*}i) \of (\alpha_1 \of \alpha_0)^*_i$
			\item $(\alpha_0^*)_{v^\prime_{0!}u_1^{\prime*}i} \of u_0^*(\beta_{0!}\alpha_1^*)_i \of \rho(u_0^*u_1^*i, v_0) = \beta_0^*\rho(u_0^{\prime*}u_1^{\prime*}i, v^\prime_0) \of (\alpha_1 \of \alpha_0)^*_i$
		\end{enumerate}
		That (a) holds follows from the equality below, where the identities labelled `d' follow from the definitions of \conref{curry J} and those labelled `s' follow from the (coop"/duals of the) splitting equations (\defref{split op-2-fibration}).
		\begin{align*}
			\lambda(u^\prime_1, i) \of (\alpha_1^*&)_i \of \alpha_{0!}\lambda(u_0, u_1^*i) \overset{\textup d}= \alpha_{1!}\lambda(u_1, i) \of \alpha_{0!}\lambda(u_0, u_1^*i) \\
			&\overset{\textup s}= (\alpha_1 \of \alpha_0)_!\bigpars{\lambda(u_1 \of u_0, i)} \overset{\textup{d,s}}= \lambda(u^\prime_1, i) \of \lambda(u^\prime_0, u_1^{\prime*}i) \of (\alpha_1 \of \alpha_0)^*_i 
		\end{align*}
		That (b) holds follows similarly from the equality below, where the identities labelled `cml' follow from \thmref{split two-sided 2-fibration description}(cml).
		\begin{align*}
			\lambda&(u^\prime_0, v^\prime_{0!}u_1^{\prime*}i) \of (\alpha_0^*)_{v^\prime_{0!}u_1^{\prime*}i} \of u_0^*(\beta_{0!}\alpha_1^*)_i \of \rho(u_0^*u_1^*i, v_0) \\
			&\overset{\textup d}= \alpha_{0!}\lambda(u_0, v^\prime_{0!}u_1^{\prime*}i) \of u_0^*(\beta_{0!}\alpha_1^*)_i \of \rho(u_0^*u_1^*i, v_0) \\
			&\overset{\textup s}= \alpha_{0!}\bigpars{\lambda(u_0, v^\prime_{0!}u_1^{\prime*}i) \of u_0^*(\beta_{0!}\alpha_1^*)_i} \of \rho(u_0^*u_1^*i, v_0) \\
			&\overset{\textup d}= \alpha_{0!}\bigpars{(\beta_{0!}\alpha_1^*)_i \of \lambda(u_0, v_{0!}u_1^*i)} \of \rho(u_0^*u_1^*i, v_0) \\
			&\overset{\textup s}= (\beta_{0!}\alpha_1^*)_i \of \alpha_{0!}\bigpars{\lambda(u_0, v_{0!}u_1^*i) \of \rho(u_0^*u_1^*i, v_0)}\\
			&\overset{\textup{cml}}= (\beta_{0!}\alpha_1^*)_i \of \alpha_{0!}\bigpars{\rho(u_1^*i, v_0) \of \lambda(u_0, u_1^*i)} \overset{\textup s}= (\beta_{0!}\alpha_1^*)_i \of \rho(u_1^*i, v_0) \of \alpha_{0!}\lambda(u_0, u_1^*i) \\
			&\overset{\textup d}= \beta_0^*\rho(u_1^{\prime*}i, v^\prime_0) \of (\alpha^*_1)_i \of \alpha_{0!}\lambda(u_0, u_1^*i) \overset{\textup{(a)}}= \beta_0^*\rho(u_1^{\prime*}i, v^\prime_0) \of \lambda(u^\prime_0, u_1^{\prime*}i) \of (\alpha_1 \of \alpha_0)^*_i \\
			&\overset{\textup{s,cml}}= \lambda(u^\prime_0, v^\prime_{0!}u_1^{\prime*}i) \of \beta_0^*\rho(u_0^{\prime*}u_1^{\prime*}i, v^\prime_0) \of (\alpha_1 \of \alpha_0)^*_i
		\end{align*}
		Using equation (b), the required equation follows from the equality
		\begin{align*}
			(\beta_{1!}\alpha_0^*)_{v^\prime_{0!} u_1^{\prime*} i} \of v_{1!}&u_0^*(\beta_{0!} \alpha_1^*)_i \of \rho(u_0^*v_{0!}u_1^*i, v_1) \of \rho(u_0^*u_1^*i, v_0) \\
			&\overset{\textup d}= (\beta_{1!}\alpha_0^*)_{v^\prime_{0!} u_1^{\prime*} i} \of \rho(u_0^*v^\prime_{0!}u_1^{\prime*}i, v_1) \of u_0^*(\beta_{0!} \alpha_1^*)_i \of \rho(u_0^*u_1^*i, v_0) \\
			&\overset{\textup d}= \beta_1^*\rho(u_0^{\prime*}v^\prime_{0!}u_1^{\prime*}i, v^\prime_1) \of (\alpha_0^*)_{v^\prime_{0!} u_1^{\prime*} i} \of u_0^*(\beta_{0!} \alpha_1^*)_i \of \rho(u_0^*u_1^*i, v_0) \\
			&\overset{\textup{(b)}}= \beta_1^*\rho(u_0^{\prime*}v^\prime_{0!}u_1^{\prime*}i, v^\prime_1) \of \beta_0^*\rho(u_0^{\prime*}u_1^{\prime*}i, v^\prime_0) \of (\alpha_1 \of \alpha_0)^*_i \\
			&\overset{\textup s}= (\beta_1 \of \beta_0)^*\rho(u_0^{\prime*}u_1^{\prime*}i, v^\prime_1 \of v^\prime_0) \of (\alpha_1 \of \alpha_0)^*_i \\
			&\overset{\textup{d,s}} = \bigpars{(\beta_1 \of \beta_0)_!(\alpha_1 \of \alpha_0)^*}_i \of \rho(u_0^*v_{0!}u_1^*i, v_1) \of \rho(u_0^*u_1^*i, v_0).
		\end{align*}
		This completes the proof.
	\end{proof}
	
	\subsubsection*{The Yoneda $2$"/functor satisfies the Yoneda axiom}
	The following theorem is the main result of this section. It shows that the Yoneda $2$"/functor $\map\yon A{\Cat^{\op A}}$ (\conref{Yoneda 2-functor}) satisfies the Yoneda axiom (\defref{Yoneda embedding}) in the augmented virtual equipment $\SfLdSpTwoTwoFib$ of locally discrete split two"/sided $2$"/fibrations with small fibres (\defsref{locally discrete split two-sided 2-fibration}{small fibres}).
	\begin{theorem} \label{Yoneda axiom}
		Let $A$ be a locally small $2$"/category and $J = \brks{A \xlar p J \xrar q B}$ a locally discrete split two"/sided $2$"/fibration with small fibres. Consider the $2$"/functor $\map{\cur J}B{\Cat^{\op A}}$ of \conref{curry J}. The two"/sided $2$"/fibration $J$ forms the lax comma $2$"/category $\yon \slash \cur J$ (\exref{lax comma 2-category}) of the Yoneda $2$"/functor $\map\yon A{\Cat^{\op A}}$ and $\cur J$.
		\begin{displaymath}
			\begin{tikzpicture}[baseline]
				\matrix(m)[math35]{J & A \\ B & \Cat^{\op A} \\};
				\path[map]	(m-1-1) edge node[above] {$p$} (m-1-2)
														edge node[left] {$q$} (m-2-1)
										(m-1-2) edge node[right] {$\yon$} (m-2-2)
										(m-2-1) edge node[below] {$\cur J$} (m-2-2);
				\path				(m-1-2) edge[cell, shorten >= 9pt, shorten <= 9pt] node[below right] {$\chi$} (m-2-1);
			\end{tikzpicture} \qquad \qquad \qquad \qquad \begin{tikzpicture}[baseline]
					\matrix(m)[math35, column sep={5em,between origins}]{i & (qs)_!i \\ j & (ps)^*j \\};
					\path[map]	(m-1-1) edge node[above] {$\rho(i, qs)$} (m-1-2)
															edge[dashed] node[below left, inner sep=1pt] {$\widetilde{\id_{pi}}$} (m-2-2)
															edge node[left] {$s$} (m-2-1)
											(m-1-2) edge[dashed] node[right] {$\hat s$} (m-2-2)
											(m-2-2) edge node[below] {$\lambda(ps, j)$} (m-2-1); 
			\end{tikzpicture}
		\end{displaymath}
		
		Under the correspondence of \cororef{nullary cells description} the defining nullary cartesian cell \mbox{$\cell\chi J{\Cat^{\op A}}$} in $\SpTwoTwoFib$ (\defref{split two-sided 2-fibration}) corresponds to the marked lax natural transformation (\defref{marked lax natural transformation}) above left with components (take $f \dfn \cur Jqi$ in \conref{evaluation})
		\begin{displaymath}
			\cell{\chi_i \dfn \dash^* i}{\yon pi}{\cur Jqi} \quad \quad \text{and} \qquad \qquad \chi_{(i \xrar s j)} \dfn \dash^* \hat s\colon (qs)_!\dash^* i \Rrightarrow (ps \of \dash)^*j,
		\end{displaymath}
		where $\hat s$ is the unique lift in the diagram above right (see \remref{factorisation of morphisms in J}).
	\end{theorem}
	\begin{proof}
		By \propref{2-category of elements is a lax comma 2-category} the $2$"/category of elements $A \xlar{\pi_A} (\cur J)^\yon \xrar{\pi_B} B$ (\conref{2-category of elements}) of $\map{\cur J}B{\Cat^{\op A}}$ forms the lax comma $2$"/category $\yon \slash \cur J$. Thus it suffices to construct an  invertible $2$"/functor $\map\phi J{(\cur J)^\yon}$ over $A$ and $B$ that forms a horizontal $(1,1)$"/ary cell in $\SpTwoTwoFib$, that is, it satifies the conditions (pcm) and (pom) of \propref{cells of locally discrete split two-sided 2-fibrations}. Define $\phi$ as follows:
		\begin{enumerate}[label=-]
			\item $\phi(i) = (pi, i, qi)$;
			\item $\phi(i \xrar s j) = (ps, (qs)_!i \xrar{\hat s} (ps)^*j, qs)$ where $\hat s$ is the unique lift in the diagram above right;
			\item $\phi(\cell\sigma st) = (p\sigma, q\sigma)$.
		\end{enumerate}
		
		That $s \mapsto \phi(s)$ is well"/defined follows from \thmref{split two-sided 2-fibration description}(v). To see that $\phi$ is $1$"/functorial consider morphisms $\map sij$ and $\map tjk$ in $J$ and recall that the composite $\phi(t) \of \phi(s) = (pt, \hat t, qt) \of (ps, \hat s, qs)$ in $(\cur J)^\yon$ is (\conref{2-category of elements})
		\begin{displaymath}
			\bigpars{p(t \of s), (qt)_!(qs)_!i \xrar{(qt)_!\hat s} (qt)_!(ps)^*j = (ps)^*(qt)_!j \xrar{(ps)^*\hat t} (ps)^*(pt)^*k, q(t \of s)}.
		\end{displaymath}
		That this coincides with $\phi(t \of s)$ follows from the equality below, where the identities follow from the splitting equation \defref{split op-2-fibration}(sm$\of$), the definition of $\cur J$ (\conref{curry J}) and \thmref{split two-sided 2-fibration description}(cml).
		\begin{align*}
			\lambda(pt \of ps, k) \of (ps&)^*\hat t \of (qt)_!\hat s \of \rho(i, qt \of qs) \\
			&= \lambda(pt, k) \of \lambda(ps, (pt)^*k) \of (ps)^*\hat t \of (qt)_!\hat s \of \rho((qs)_!i, qt) \of \rho(i, qs) \\
			&= \lambda(pt, k) \of \hat t \of \lambda(ps, (qt)_!j) \of \rho((ps)^*j, qt) \of \hat s \of \rho(i, qs) \\
			&= \lambda(pt, k) \of \hat t \of \rho(j, qt) \of \lambda(ps, j) \of \hat s \of \rho(i, qs) = t \of s
		\end{align*}
		
		Showing that $\phi(\sigma) = (p\sigma, q\sigma)$ is well"/defined means proving that $(p\sigma)^*_j \of \hat s = \hat t \of (q\sigma)_{!i}$. This follows from the equality below, where the identity labelled `f' follows from \defref{discrete two-sided 1-fibration}(f). That $\phi$ is $2$"/functorial and a span morphism over $A$ and $B$ is clear. To see that $\phi$ satisfies \propref{cells of locally discrete split two-sided 2-fibrations}(pcm) notice that $\phi(\lambda(u, i)) = (u, \id_{u^*i}, \id_{qi})$, which is the cartesian lift of $\map ua{pi = (\pi_A \of \phi)(i)}$ (see \conref{2-category of elements}). That $\phi$ satisfies condition (pom) of the same proposition is similar. We conclude that $\phi$ forms a horizontal $(1,1)$"/ary cell $J \Rar (\cur J)^\yon$ in $\LdSpTwoTwoFib$.
		\begin{align*}
			\lambda(&pt, j) \of (p\sigma)^*_j \of \hat s \of \rho(i, qs) = (p\sigma)_!\lambda(ps, j) \of \hat s \of \rho(i, qs) \\
			&= (p\sigma)_!\bigpars{\lambda(ps, j) \of \hat s \of \rho(i, qs)} = (p\sigma)_!s \overset{\textup f}= (q\sigma)^*t = (q\sigma)^*\bigpars{\lambda(pt, j) \of \hat t \of \rho(i, qt)} \\
			&= \lambda(pt, j) \of \hat t \of (q\sigma)^*\rho(i, qt) = \lambda(pt, j) \of \hat t \of (q\sigma)_{!i} \of \rho(i, qs)
		\end{align*}
		
		It remains to prove that $\phi$ is invertible. To do so we check that $\phi$ is bijective on objects, morphisms and cells. For objects this follows from the fact that any $(a, i, b) \in (\cur J)^\yon$ satisfies $i \in (\cur J b)(a) = J_{(a, b)}$ (\conref{2-category of elements}), so that $a = pi$ and $b = qi$. For morphisms the inverse assignment is
		\begin{displaymath}
			\psi(pi \xrar u pj, v_!i \xrar s u^*j, qi \xrar v qj) \dfn \lambda(u, j) \of s \of \rho(i, v).
		\end{displaymath}
	Clearly, $\psi \of \phi$ is the identity on morphisms. That $\phi \of \psi$ is too follows from the fact that the $p$"/ and $q$"/images of $\psi(u, s, v)$ are $u$ and $v$, respectively, while $\widehat{\psi(u, s, v)} = s$ by the uniqueness of $\widehat\dash$. For cells the inverse assignment $\psi$ maps $(\cell\alpha{u_0}{u_1}, \cell\beta{v_0}{v_1})$ to the composite
		\begin{align*}
			\lambda(u_0, j&) \of s_0 \of \rho(i, v_0) \xRar{\lambda(\lambda(u_0, j), \alpha) \of s_0 \of \rho(i, v_0)} \alpha_!\lambda(u_0, j) \of s_0 \of \rho(i, v_0) \\
			&= \lambda(u_1, j) \of s_1 \of \beta^*\rho(i, v_1) \xRar{\lambda(u_1, j) \of s_1 \of \rho(\beta,\rho(i, v_1))} \lambda(u_1, j) \of s_1 \of \rho(i, v_1),
		\end{align*}
		where the middle identity is the equality
		\begin{align*}
			\alpha_!\lambda(u_0, j) \of s_0 \of \rho(i, v_0) &= \lambda(u_1, j) \of \alpha^*_j \of s_0 \of \rho(i, v_0) \\
			&= \lambda(u_1, j) \of s_1 \of \beta_{!i} \of \rho(i, v_0) = \lambda(u_1, j) \of s_1 \of \beta^*\rho(i, v_1)
		\end{align*}
		whose middle identity follows from the definition of cells in $(\cur J)^\yon$ (\conref{2-category of elements}). That $\phi \of \psi$ is the identity on cells follows from \thmref{split two-sided 2-fibration description}(v) and the fact that $s_0$ and $s_1$ are contained in the fibre $J_{(pi, qj)}$. That $\psi \of \phi$ is too is shown by
		\begin{align*}
			\psi(p\sigma, q\sigma) &= \bigpars{\lambda(\lambda(u_0, j), p\sigma) \of \widehat{s_0} \of \rho(i, v_0)} \hc \bigpars{\lambda(u_1, j) \of \widehat{s_1} \of \rho(q\sigma,\rho(i, v_1))} \\
			&= \lambda(s_0, p\sigma) \hc \rho(q\sigma, s_1) = \sigma
		\end{align*}
		where the middle identity follows from the (coop"/duals of the) splitting equations (sc$\of$) and (sc1) of \defref{split op-2-fibration} and the last identity from \defref{discrete two-sided 1-fibration}(f). This completes the proof.
	\end{proof}
	
	\subsubsection*{Main result}
	Combining \thmref{Yoneda 2-functor is dense} and \thmref{Yoneda axiom}, we obtain the following, which is the main result of this paper.	
	\begin{corollary} \label{main result}
		Let $A$ be a locally small $2$"/category. The Yoneda $2$"/functor \mbox{$\map\yon A{\Cat^{\op A}}$} (\conref{Yoneda 2-functor}) is a Yoneda embedding in the augmented virtual equipment $\SfLdSpTwoTwoFib$ of locally discrete split two"/sided $2$"/fibrations whose fibres are small (\defref{small fibres}), in the sense of \defref{Yoneda embedding}.
	\end{corollary}
	
	The universal properties of the Yoneda $2$"/functors imply the following bijective correspondence between cells of locally discrete split two"/sided $2$"/fibrations; this is Lemma~4.22 of \cite{Koudenburg24}. The cells denoted `cart' on either side of the equation below are the nullary cartesian cells described by \thmref{Yoneda axiom}, which are left Kan by \thmref{Yoneda 2-functor is dense}.
	\begin{corollary} \label{bijection between cells induced by Yoneda 2-functor}
		In the diagrams below let $J$ and $K$ be locally discrete split two"/sided $2$"/fibrations (\defref{locally discrete split two-sided 2-fibration}) and let $\ul H = (H_1, \dotsc, H_n)$ be a (possibly empty) path of locally discrete split two"/sided 2"/fibrations. Assume that $J$ and $K$ have small fibres (\defref{small fibres}) and that the $2$"/category $A$ is locally small. The equation below determines a bijection between multimorphisms $\phi$ and marked lax natural transformations $\psi$ in $\LdSpTwoTwoFib$ (\propref{cells of locally discrete split two-sided 2-fibrations} and \defref{marked lax natural transformation}) of the forms as shown.
		\begin{displaymath}
			\begin{tikzpicture}[textbaseline]
				\matrix(m)[math35, column sep={1.75em,between origins}]
					{	A & & B & & B_1 & \cdots & B_{n'} & & B_n \\
						& & & A & & D & & & \\
						& & & & \Cat^{\op A} & & & & \\};
				\path[map]	(m-1-1) edge[barred] node[above] {$J$} (m-1-3)
										(m-1-3) edge[barred] node[above] {$H_1$} (m-1-5)
										(m-1-7) edge[barred] node[above] {$H_n$} (m-1-9)
										(m-1-9) edge[transform canvas={xshift=2pt}] node[below right] {$s$} (m-2-6)
										(m-2-4) edge[barred] node[below, inner sep=2pt] {$K$} (m-2-6)
														edge[transform canvas={xshift=-2pt}] node[left] {$\yon$} (m-3-5)
										(m-2-6) edge[transform canvas={xshift=2pt}, ps] node[right] {$\cur K$} (m-3-5);
				\path				(m-1-1) edge[transform canvas={xshift=-1pt}, eq] (m-2-4)
										(m-1-5) edge[cell] node[right] {$\phi$} (m-2-5);
				\draw				([yshift=0.25em]$(m-2-5)!0.5!(m-3-5)$) node[font=\scriptsize] {$\cart$};
			\end{tikzpicture} \quad = \quad \begin{tikzpicture}[textbaseline]
				\matrix(m)[math35, column sep={1.75em,between origins}]
					{	A & & B & & B_1 & \cdots & B_{n'} & & B_n \\
						& & & & & D & & & \\
						& & \Cat^{\op A} & & & & & & \\};
				\path[map]	(m-1-1) edge[barred] node[above] {$J$} (m-1-3)
														edge[transform canvas={xshift=-2pt}] node[left] {$\yon$} (m-3-3)
										(m-1-3) edge[barred] node[above] {$H_1$} (m-1-5)
														edge[ps] node[right] {$\cur J$} (m-3-3)
										(m-1-7) edge[barred] node[above] {$H_n$} (m-1-9)
										(m-1-9) edge[transform canvas={xshift=2pt}] node[below right] {$s$} (m-2-6)
										(m-2-6) edge[transform canvas={xshift=2pt}, ps] node[below right] {$\cur K$} (m-3-3);
				\path				(m-1-6) edge[transform canvas={xshift=-1.5em,yshift=-0.5em}, cell] node[right] {$\psi$} (m-2-6);
				\draw				([yshift=1.15em,xshift=0.5em]$(m-1-2)!0.5!(m-3-2)$) node[font=\scriptsize] {$\cart$};
			\end{tikzpicture}
		\end{displaymath}
	\end{corollary}
	
	\subsubsection*{Grothendieck correspondence for locally discrete split two"/sided $2$"/fibrations}
	Restricting the previous result to cells with $s = \id_D$, the identity morphism, we obtain the following `Grothendieck correspondence' for locally discrete split two"/sided $2$"/fibrations with small fibres (\defref{small fibres}). In stating the correspondence we use the following notations.
	
	Given a $2$"/category $A$, the \emph{lax slice $1$"/category} $A \slash H(\SfLdSpTwoTwoFib)$\footnote{$H(\SfLdSpTwoTwoFib)$ here denotes the `multi"/bicategory', in the sense of Section~1.4 of \cite{Cockett-Koslowski-Seely03}, that is contained in $\SfLdSpTwoTwoFib$ as follows. Its objects are the large $2$"/categories, its morphisms are the small"/fibred locally discrete split two"/sided $2$"/fibrations $\hmap JAB$ and its multimorphisms are the horizontal unary multimorphisms $\ul J \Rar K$ of $\SfLdSpTwoTwoFib$ (\propref{cells of locally discrete split two-sided 2-fibrations}), that is, whose vertical boundaries are identity morphisms.} has as objects locally discrete split two"/sided $2$"/fibrations $\hmap JAB$ with small fibres, where $B$ is any $2$"/category, and as morphisms $J \to K$ pairs $(\ul H, \phi)$ comprising multimorphisms in $\SfLdSpTwoTwoFib$ of the form below left (\propref{cells of locally discrete split two-sided 2-fibrations}), where $\ul H$ is possibly empty. Composition is defined by $(\ul L, \chi) \of (\ul H, \phi) \dfn (\ul H \conc \ul L, \chi \of (\phi \conc \id_{\ul L}))$.
	
	Next, given any $2$"/category $P$, we denote by $\SfLdSpTwoTwoFib \vs P$ the \emph{vertical slice $1$"/category} consisting of $2$"/functors $\map fBP$ as objects and, as morphisms $f \to g$, pairs $(\ul H, \psi)$ comprising marked lax natural transformations of the form below right (\defref{marked lax natural transformation}). Composition is given by horizontal composition $\hc$ in $\SfLdSpTwoTwoFib$. Notice that $\SfLdSpTwoTwoFib \vs P$ contains the hom"/categories $\TwoCat'(B, P)$ as sub"/$1$"/categories of morphisms $f \to g$ with $\ul H = (B)$ empty.
	\begin{displaymath}
		\begin{tikzpicture}[baseline]
				\matrix(m)[math35, column sep={1.75em,between origins}]
					{	A & & B & & B_1 & \cdots & B_{n'} & & D \\
						& & & A & & D & & & \\};
				\path[map]	(m-1-1) edge[barred] node[above] {$J$} (m-1-3)
										(m-1-3) edge[barred] node[above] {$H_1$} (m-1-5)
										(m-1-7) edge[barred] node[above] {$H_n$} (m-1-9)
										(m-2-4) edge[barred] node[below] {$K$} (m-2-6);
				\path				(m-1-1) edge[transform canvas={xshift=-1pt}, eq] (m-2-4)
										(m-1-9) edge[transform canvas={xshift=1pt}, eq] (m-2-6)
										(m-1-5) edge[cell] node[right] {$\phi$} (m-2-5);
			\end{tikzpicture} \qquad\qquad \begin{tikzpicture}[baseline]
			\matrix(m)[math35, column sep={1.75em,between origins}]
				{B & & B_1 & & B_{n'} & & D \\ & & & P & & & \\};
			\path[map]	(m-1-1) edge[barred] node[above] {$H_1$} (m-1-3)
													edge node[below left] {$f$} (m-2-4)
									(m-1-5) edge[barred] node[above] {$H_n$} (m-1-7)
									(m-1-7) edge node[below right] {$g$} (m-2-4);
			\path				(m-1-4) edge[cell] node[right] {$\psi$} (m-2-4);
			\draw				(m-1-4) node {$\dotsb$};
		\end{tikzpicture}
	\end{displaymath}
	\begin{corollary}[Grothendieck correspondence for locally discrete split two-sided 2-fibrations] \label{equivalence from yoneda embedding}
		Let $A$ be a locally small $2$"/category. The $2$"/category of elements construction (\conref{2-category of elements}) and the inverse Grothendieck construction (\conref{curry J}) extend to form an equivalence of slice $1$"/categories
		\begin{displaymath}
			\cur{(\dash)}_A\colon A \slash H(\SfLdSpTwoTwoFib) \simeq \SfLdSpTwoTwoFib\vs \Cat^{\op A} \colon (\dash)^\yon_A.
		\end{displaymath}
		
		Under this equivalence the multimorphism $\map{(\ul H, \phi)}JK$ above left (\propref{cells of locally discrete split two-sided 2-fibrations}) corresponds to the marked lax natural transformation $\map{(\ul H, \cur\phi)}{\cur J}{\cur K}$ (\defref{marked lax natural transformation}) as above right (with $f = \cur J$ and $g = \cur K$) under the bijection of \cororef{bijection between cells induced by Yoneda 2-functor} (with $s = \id_D$). Similarly, the marked lax natural transformation $\map{(\ul H, \psi)} fg$ above right (with $P = \Cat^{\op A}$) corresponds to the unique multimorphism $\map{(\ul H, \psi^\yon)}{f^\yon}{g^\yon}$ as above left satisfying $\chi_g \of \psi^\yon = \chi_f \hc \psi$, where $\chi_f$ and $\chi_g$ are the cartesian cells defining $f^\yon$ and $g^\yon$ as the nullary restrictions $\Cat^{\op A}(\yon, f)$ and $\Cat^{\op A}(\yon, g)$, respectively (\propref{2-category of elements is a lax comma 2-category}).
	\end{corollary}
	\begin{proof}
		(Parts of) the proofs of Propositions~4.24 and 4.26 of \cite{Koudenburg24} apply almost verbatim, as follows. That the action $(\ul H, \phi) \mapsto (\ul H, \cur\phi)$ described above is functorial follows from the uniqueness of the factorisations through the left Kan cells considered in \cororef{bijection between cells induced by Yoneda 2-functor}. Thus a $1$"/functor $\map{\cur{(\dash)}_A}{A \slash H(\SfLdSpTwoTwoFib)}{\SfLdSpTwoTwoFib\vs \Cat^{\op A}}$, by the same corollary $\cur{(\dash)}_A$ is full and faithful.
		
		To show that $\cur{(\dash)}_A$ is essentially surjective consider a $2$"/functor $\map fB{\Cat^{\op A}}$ in its target; we claim that $f \iso \cur{(f^\yon)}$. This follows from the fact that the $2$"/category of elements $f^\yon$ is both the nullary restriction $\Cat^{\op A}(\yon, f)$, by \propref{2-category of elements is a lax comma 2-category}, as well as the nullary restriction $\Cat^{\op A}(\yon, \cur{(f^\yon)})$, by \thmref{Yoneda axiom}. Indeed, since the Yoneda $2$"/functor $\map\yon A{\Cat^{\op A}}$ is dense (\thmref{Yoneda 2-functor is dense}) we find that, by applying \defref{density}, both cartesian cells exhibiting these nullary restrictions are left Kan. We conclude that both $f$ and $\cur{(f^\yon)}$ form the left Kan extension of $\yon$ along $f^\yon$, so that $f \iso \cur{(f^\yon)}$ by uniqueness of such extensions.
		
		Finally, it is shown in e.g.\ (the proof of) Theorem~1 of Section~IV.4 of \cite{MacLane98} that the assignment $f \mapsto f^\yon$ extends to a pseudo"/inverse $(\dash)^\yon_A$ of $\cur{(\dash)}_A$. That the action of $(\dash)^\yon_A$ on morphisms is as defined in the corollary is easily checked.
	\end{proof}
	
	By restricting the result above to $A = 1$, the terminal $2$"/category, we obtain the Grothendieck correspondence for locally discrete split op"/$2$"/fibrations below, as follows. By \defref{locally discrete split op-2-fibration}, small"/fibred locally discrete split op"/$2$"/fibrations over a $2$"/category $B$ are small"/fibred locally discrete split two"/sided $2$"/fibrations of the form $1 \brar B$. Whence we write $\SpOpTwoFib_\textup{ld,sf,mm} \dfn 1 \slash H(\SfLdSpTwoTwoFib)$, where `mm' is short for `multimorphism'. Notice the difference between the morphisms of $\SpOpTwoFib_\textup{ld,sf,mm}$ and those of $\und{\und{\SpOpTwoFib}}$ (\defref{split op-2-fibration}): those of the former are horizontal multimorphisms while those of the latter are $(1, 1)$"/ary morphisms that need not be horizontal.
	
	Fixing, moreover, the base $2$"/category $B$ in the result above, we denote by $\SpOpTwoFib_\textup{ld,sf}(B) \subset \SpOpTwoFib_\textup{ld,sf,mm}$ the sub"/$1$"/category of small"/fibred locally discrete split op"/$2$"/fibrations $1 \brar B$, and the $(1, 1)$"/ary morphisms between them, that is, those morphisms $(\ul H, \phi)$ with $\ul H = (B)$ empty.
		
	The (coop"/dual of the) equivalence of $1$"/categories (b) below underlies the equivalence of $2$"/categories obtained in Theorem~3.7 of \cite{Lambert24}, which in turn is a restriction of a $3$"/categorical correspondence for bicategorical fibrations (Sections~5 and 6 of \cite{Bakovic10} and Theorem~3.3.12 of \cite{Buckley14}). As mentioned at the end of the \outlineref, the equivalence (a) below can be thought of as ``extending the functoriality'' of Lambert's Grothendieck correspondence to multimorphisms of the form $J \conc \ul H \Rar K$, where $J$ and $K$ are small"/fibred locally discrete split op"/$2$"/fibrations (\defref{locally discrete split op-2-fibration}) and $\ul H$ is a path of small"/fibred locally discrete split two"/sided $2$"/fibrations (\defref{locally discrete split two-sided 2-fibration}).
	\begin{corollary}[Grothendieck correspondence for locally discrete split op"/$2$"/fibrations] \label{locally discrete split op-2-fibration correspondence}
		The equivalence of \cororef{equivalence from yoneda embedding} restricts to the following equivalences of $1$"/categories:
		\begin{enumerate}[label=\textup{(\alph*)}]
			\item $\cur{(\dash)}_1\colon \SpOpTwoFib_\textup{ld,sf,mm} \simeq \SfLdSpTwoTwoFib\vs \Cat \colon (\dash)^\yon_1$;
			\item $\cur{(\dash)}_{1, B}\colon \SpOpTwoFib_\textup{ld,sf}(B) \simeq \TwoCat'(B, \Cat) \colon (\dash)^\yon_{1, B}$
		\end{enumerate}
	\end{corollary}
	\begin{proof}
		The equivalence (a) is obtained by restricting the equivalence of \cororef{equivalence from yoneda embedding} to $A = 1$. The equivalence (b) is obtained by, moreover, fixing the base $2$"/category $B$ and restricting to morphisms $(\ul H, \phi)$ with $\ul H = (B)$ empty.
	\end{proof}
	
	\begin{remark} \label{other formal results}
		Besides the corollaries above, all of the formal category theory developed in \cite{Koudenburg24} can be applied to the augmented virtual equipment $\SfLdSpTwoTwoFib$ of locally discrete split two"/sided $2$"/fibrations with small fibres (\defref{small fibres}), such as its results on exact cells, totality and cocompleteness; see its Sections~5, 6 and 7, respectively.
	\end{remark}
	
	\backmatter
	
	\begin{appendices}
	\section{On comma objects in sesquicategories} \label{comma objects appendix}
	Here we briefly recall the way in which sesquicategories can be regarded as enriched categories. Thus we obtain an enriched notion of comma object in sesquicategories, which we compare to the present notion of \defref{comma object}.
	
	Consider the closed symmetric monoidal structure $(\Box, 1, \inhom{\dash, \dash})$ on the $1$"/category $\und{\inCat{\Cls}}$ of class"/sized $1$"/categories with $\inhom{A, C}$ the $1$"/category of functors $A \to C$ and unnatural transformations (\defref{unnatural transformation definition}). The monoidal structure $\Box$ corresponding to $\inhom{\dash, \dash}$ is called the \emph{funny tensor product}; see Section~2 of \cite{Weber13} for an explicit description. Sesquicategories can alternatively be defined as class"/sized categories enriched over $(\und{\inCat{\Cls}}, \Box)$; see e.g.\ Section~3 of \cite{Stell94}. Thus we obtain enriched notions of functors of sesquicategories and of natural transformations between them. Moreover, we find that the $1$"/category $\und{\inCat{\Cls}}$, besides underlying the $2$"/category $\inCat{\Cls}$ of $2$"/natural transformations, is also enriched over itself with respect to $(\Box, 1, \inhom{\dash, \dash})$ and thus underlies too the sesquicategory $\inCat{\Cls}_\textup{un}$ of class"/sized $1$"/categories, functors and unnatural transformations (see \exref{unnatural transformation}). Every sesquicategory $\s$ then admits a $(\und{\inCat{\Cls}}, \Box)$"/enriched hom"/functor \mbox{$\map{\s(\dash, \dash)}{\op\s \Box \s}{\inCat{\Cls}_\textup{un}}$} (see e.g.\ Section~1.6 of \cite{Kelly82}). Given cells $\phi$ and $\psi$ in $\s$ notice that $\s(\dash, \dash)$ does not define an unnatural transformation $\s(\phi, \psi)$, since the pair $(\phi, \psi)$ does not form a cell in the funny tensor product $\op \s \Box \s$.
	
	Finally, since $\und{\inCat{\Cls}}$ is complete, the $1$"/category $\und{\inhom{\s, \L}}$ of functors of sesquicategories $\s \to \L$ and their natural transformations underlies a sesquicategory $\inhom{\s, \L}$ whose cells are modifications of natural transformations, whenever $\s$ is small (see e.g.\ Section~2.2 of \cite{Kelly82}). Hence we can consider $(\und{\inCat{\Cls}}, \Box)$"/enriched weighted limits in sesquicategories, in the sense of e.g.\ Section~3.1 of \cite{Kelly82} (therein called ``indexed limits''), as follows.
	\begin{definition}
		Let $\A$ be a small sesquicategory and $\map W\A{\inCat{\Cls}_\textup{un}}$ and $\map D\A\s$ be functors of sesquicategories. By the \emph{$W$"/weighted sesquilimit of $D$} we mean the $W$"/weighted $(\und{\inCat{\Cls}}, \Box)$"/enriched limit of $D$, that is, an object $L \in \s$ equipped with an isomorphism
		\begin{displaymath}
			\s(\dash, L) \iso \inhom{\A, \inCat{\Cls}_\textup{un}}(W, \s(\dash, D))
		\end{displaymath}
		of functors of sesquicategories $\op\s \to \inCat{\Cls}_\textup{un}$.
	\end{definition}
	
	\begin{example} \label{conical limit as enriched limit}
		Any $1$"/functor $\map D{\catvar I}{\und\s}$ into the $1$"/category $\und\s$ underlying $\s$ can be regarded as a functor $\map D{\catvar I}\s$ of sesquicategories, by regarding the $1$"/category $\catvar I$ as a locally discrete sesquicategory (a $2$"/category in fact). Taking \mbox{$\map{W \dfn \Delta 1}{\catvar I}{\inCat{\Cls}_\textup{un}}$} to be the constant functor at the terminal $1$"/category $1$, we recover the conical limit of the $1$"/functor $D$ (\exref{conical limit}) as the $\Delta 1$"/weighted sesquilimit of $D$.
	\end{example}
	
	Now take $\A = (\bullet \rightarrow \bullet \leftarrow \bullet)$ to be the ``walking cospan'' $1$"/category regarded as a locally discrete sesquicategory and $\map W\A{\inCat{\Cls}_\textup{un}}$ the functor that picks out the cospan $\1 \xrightarrow 0 \2 \xleftarrow 1 \1$, where $\1$ denotes the terminal $1$"/category and $\2 = (0 \to 1)$ denotes the ``walking arrow'' $1$"/category. Analogously to the classical $2$"/categorical notion of comma object (see e.g.\ Example~13 of \cite{Lack09}) we define a \emph{sesquicomma object} in a sesquicategory to be a sesquilimit weighted by this $W$.
	\begin{proposition}
		In a sesquicategory consider the cell $\pi$ on the left below. It defines the object $f \slash g$ as the sesquicomma object of the morphisms $\map fAC$ and $\map gBC$, in the sense above, if and only if the following conditions hold.
		\begin{enumerate}[label=\textup{(\alph*)}]
			\item The cell $\pi$ defines $f \slash g$ as the $1$"/universal comma object of $f$ and $g$, in the sense of \defref{comma object}.
			\item Consider cells $\phi$ and $\psi$ as in the middle below. By condition \textup{(a)} they induce morphisms $\phi'$, $\map{\psi'}X{f \slash g}$ as on the right. For any pair of cells $\xi_A$ and $\xi_B$ as in the middle there exists a unique cell $\xi'$ as on the right satisfying $\pi_A \of \xi' = \xi_A$ and $\pi_B \of \xi' = \xi_B$.
		\end{enumerate}
		\begin{displaymath}
			\begin{tikzpicture}
				\matrix(m)[math35]{f \slash g & A \\ B & C \\};
				\path[map]	(m-1-1) edge node[above] {$\pi_A$} (m-1-2)
														edge node[left] {$\pi_B$} (m-2-1)
										(m-1-2) edge node[right] {$f$} (m-2-2)
										(m-2-1) edge node[below] {$g$} (m-2-2);
				\path				(m-1-2) edge[cell, shorten >= 9pt, shorten <= 9pt] node[below right] {$\pi$} (m-2-1);
			\end{tikzpicture} \quad \qquad \begin{tikzpicture}
				\matrix(m)[math35]{X & A \\ B & C \\};
				\path[map]	(m-1-1) edge node[above] {$\phi_A$} (m-1-2)
														edge[bend left=40] node[above right, inner sep=0.5pt, yshift=3pt] {$\phi_B$} (m-2-1)
														edge[bend right=40] node[left] {$\psi_B$} (m-2-1)
										(m-1-2) edge node[right] {$f$} (m-2-2)
										(m-2-1) edge node[below] {$g$} (m-2-2);
				\path[transform canvas={yshift=-0.25em, xshift=0.25em}]	(m-1-2) edge[cell, shorten >= 9pt, shorten <= 9pt] node[below right, inner sep=3pt] {$\phi$} (m-2-1);
				\path				([xshift=8pt]$(m-1-1)!0.5!(m-2-1)$) edge[cell] node[above] {$\xi_B$} ([xshift=-8pt]$(m-1-1)!0.5!(m-2-1)$);
			\end{tikzpicture} \quad \qquad \begin{tikzpicture}
				\matrix(m)[math35]{X & A \\ B & C \\};
				\path[map]	(m-1-1) edge[bend left=40] node[above] {$\phi_A$} (m-1-2)
														edge[bend right=40] node[below left, inner sep=0.1pt, xshift=-3pt] {$\psi_A$} (m-1-2)
														edge node[left] {$\psi_B$} (m-2-1)
										(m-1-2) edge node[right] {$f$} (m-2-2)
										(m-2-1) edge node[below] {$g$} (m-2-2);
				\path[transform canvas={yshift=-0.25em, xshift=0.25em}]	(m-1-2) edge[cell, shorten >= 9pt, shorten <= 9pt] node[below right] {$\psi$} (m-2-1);
				\path				([yshift=8pt]$(m-1-1)!0.5!(m-1-2)$) edge[cell] node[right] {$\xi_A$} ([yshift=-8pt]$(m-1-1)!0.5!(m-1-2)$);
			\end{tikzpicture} \quad \qquad \begin{tikzpicture}
				\matrix(m)[math35]{X \\ f \slash g \\};
				\path[map]	(m-1-1) edge[bend right=45] node[left] {$\phi'$} (m-2-1)
														edge[bend left=45] node[right] {$\psi'$} (m-2-1);
				\path				(m-1-1) edge[cell] node[right, inner sep=2.5pt] {$\xi'$} (m-2-1);
			\end{tikzpicture}
		\end{displaymath}
	\end{proposition}
	Notice that in condition (b) the two middle composites above are not required to be equal.
	\begin{example}
		Comma objects in the sesquicategory $\Catun$ of unnatural transformations (\exref{comma objects for unnatural transformations}) are sesquicomma objects in the above sense. This easily follows from the definitions of $f \slash g$, $\phi'$ and $\psi'$. That the middle two composites above need not be equal reflects the fact that the morphisms of $f \slash g$ are squares of morphisms that need not commute; see \exref{comma objects for unnatural transformations}.
	\end{example}
	
	\begin{example}
		The lax comma $2$"/categories $f \slash g$ in the sesquicategory $\und{\TwoCatLax}$ of lax natural transformations between $2$"/functors (\exref{lax comma 2-category}) are $1$"/universal comma objects (\defref{comma object}) but they do not satisfy the $2$"/dimensional universal property (b) of the proposition above. Instead they satisfy the variation (b') below, which is a restriction of the $2$"/dimensional universal property satisfied by $f \slash g$ in the lax $3$"/category $\TwoCatLax$ (\remref{lax 3-category}), as follows. The more general universal property, which is Definition~4.4(ii) of \cite{Mesiti24}, asserts the existence of a unique cell $\xi'$ above right for any pair of cells $\xi_A$ and $\xi_B$ as in the middle above, together with a modification (\defref{modification}) $\Xi$ from the middle left composite to the middle right composite. Restricting the latter to the identity modification results in the following property.
		\begin{enumerate}
			\item[(b')] Consider cells $\phi$ and $\psi$ as in the middle above inducing, by condition \textup{(a)}, morphisms $\phi'$, $\map{\psi'}X{f \slash g}$ as on the right. For any pair of cells $\xi_A$ and $\xi_B$ making equal the two composites in the middle above there exists a unique cell $\xi'$ (\defref{strict cells}) as on the right satisfying $\pi_A \of \xi' = \xi_A$ and $\pi_B \of \xi' = \xi_B$.
		\end{enumerate}
		Notice that this implies that $\xi'$ is $\pi$"/costrict in the sense of \defref{strict cells}, that is, the interchange equality below holds.
		\begin{displaymath}
			(f \of \pi_A \of \xi') \hc (\pi \of \psi') = (\pi \of \phi') \hc (g \of \pi_B \of \xi')	
		\end{displaymath}
	\end{example}
	\end{appendices}
	
	\backmatter
	
	\bmhead*{Acknowledgements}
	The author is grateful to Bob Par\'e for his friendship throughout the years. He is thankful to the anonymous referee for their careful comments and suggestions, which led to several improvements of the paper. In particular, they noticed that Proposition~1.7 of the first version of this paper (uploaded to arXiv on February 10, 2026) was false.

	\bmhead*{Declarations}
	\begin{itemize}
		\item Funding declaration: none.
		\item Ethics declaration: not applicable.
	\end{itemize}
	
	%\bibliography{../_other/main}
	
	%\iffalse
	\bibliographystyle{sn-mathphys-ay}

%\fi
\end{document}